\documentclass[a4paper,12pt]{article}
\usepackage[TS1,OT1]{fontenc}  
\usepackage{amsmath,mathtools,amsfonts,amssymb,mathrsfs,hyperref,enumerate,enumitem,amsthm,fullpage,graphicx,caption,color,float,subfigure,verbatim,bm,bbm, xfrac,qtree,tikz}
\usepackage[retainorgcmds]{IEEEtrantools}
\usepackage[final]{microtype}
\usepackage{tikz-qtree}
\usetikzlibrary{arrows,automata,positioning}
\usepackage[style=numeric,backend=bibtex,maxbibnames=100]{biblatex} 
\theoremstyle{definition} \newtheorem{Definition}{Definition}[section]
\theoremstyle{plain} \newtheorem{Theorem}[Definition]{Theorem}
\theoremstyle{plain} \newtheorem{corollary}[Definition]{Corollary}
\theoremstyle{plain} \newtheorem{lemma}[Definition]{Lemma}
\theoremstyle{definition} \newtheorem{Remark}[Definition]{Remark}
\theoremstyle{definition}

\theoremstyle{definition}
\newtheorem{notation}[Definition]{Notation}
\theoremstyle{plain}
\newtheorem{proposition}[Definition]{Propostion}
\theoremstyle{plain} 
\newtheorem{claim}[Definition]{Claim}
\theoremstyle{plain} 
\theoremstyle{definition}
\newtheorem{example}[Definition]{Example}
\theoremstyle{plain} 
\theoremstyle{definition}
\newtheorem{construction}[Definition]{Construction}

\newcommand{\gen}[1]{\langle #1 \rangle}

\newcommand{\id}{\mathrm{id}}
\newcommand{\tid}{\mathbbm{1}}

\newcommand{\aut}[1]{\mathop{\mathrm{Aut}}({#1})}
\newcommand{\End}{\mathop{\mathrm{End}}\nolimits}
\newcommand{\out}[1]{\mathop{\mathrm{Out}}({#1})}

\newcommand{\im}[1]{\mathrm{Image}{(#1)}}
\newcommand{\dom}[1]{\mathrm{Domain}{(#1)}}
\newcommand{\ext}[2]{\mathrm{ext}_{#1,#2}}

\newcommand{\com}[1]{[#1,#1]}

\newcommand{\T}[1]{\mathcal{#1}}
\newcommand{\CCmr}[1]{\mathfrak{C}_{#1}}
\newcommand{\CCnr}{\mathfrak{C}_{n,r}}
\newcommand{\CCn}{\mathfrak{C}_{n}}

\newcommand{\Bnr}{\mathcal{B}_{n,r}}

\newcommand{\Mn}[1]{\mathfrak{M}(#1)}
\newcommand{\mn}[1]{\mathfrak{G}(#1)}
\newcommand{\Smn}[1]{\mathcal{M}_{n}}

\newcommand{\On}{\T{O}_{n}}
\newcommand{\Onr}{\T{O}_{n,r}}
\newcommand{\Ons}[1]{\T{O}_{n,#1}}

\newcommand{\xn}{X_{n}}
\newcommand{\xns}{\xn^{*}}
\newcommand{\xnp}{\xn^{+}}
\newcommand{\xnl}[1]{\xn^{#1}}
\newcommand{\xnk}[2]{X_{#1}^{#2}}

\newcommand{\xno}{X_n^{\omega}}
\newcommand{\xnz}{X_n^{\Z}}
\newcommand{\xnN}{X_n^{-\N}}
\newcommand{\xnn}{X_n^{\N}}

\newcommand{\wn}{W_{n}}
\newcommand{\wns}{\wn^{\ast}}
\newcommand{\wnl}[1]{\wn^{#1}}
\newcommand{\rwnl}[1]{\mathsf{W}_{n}^{#1}}

\newcommand{\spnprod}[1]{\ast_{\spn{n}}}

\newcommand{\pn}[1]{\mathcal{P}_{#1}}
\newcommand{\spn}[1]{\widetilde{\T{P}}_{#1}}
\newcommand{\hn}[1]{\mathcal{H}_{#1}}

\newcommand{\core}{\mathrm{Core}}

\newcommand{\ALn}[1]{\mathcal{AL}_{#1}}
\newcommand{\Ln}[1]{\mathcal{L}_{#1}}
\newcommand{\Kn}[1]{\mathcal{K}_{#1}}
\newcommand{\SLn}[1]{\widetilde{\mathcal{L}_{#1}}}
\newcommand{\SOn}[1]{\widetilde{\mathcal{O}_{#1}}}
\newcommand{\Un}[1]{\mathfrak{U}_{#1}}
\newcommand{\ASLn}[1]{\widetilde{\mathcal{ASL}_{#1}}}

\newcommand{\shift}[1]{\sigma_{#1}}
\newcommand{\Shift}[1]{\mathfrak{S}_{#1}}

\newcommand{\dotr}{{\bf{\dot{r}}}}

\newcommand{\rev}[1]{\overleftarrow{#1}}

\newcommand{\Z}{\mathbb{Z}}
\newcommand{\N}{\mathbb{N}}
\newcommand{\Q}{\mathbb{Q}}
\newcommand{\sym}[1]{\mathrm{Sym}(#1)}
\newcommand{\tran}[1]{\mathop{\mathrm{Tr}}(#1)}

\newcommand{\rot}{\mathmbox{rot}}
\newcommand{\simrot}{\sim_{{\rot}}}
\newcommand{\rotclass}[1]{[#1]_{\simrot}}
\newcommand{\sig}{\mathrm{sig}}
\newcommand{\rsig}{\protect\overrightarrow{\mathrm{sig}}}
\newcommand{\emptyword}{\varepsilon}
\newcommand{\ew}{\emptyword}
\newcommand{\nset}[1]{\{0,1,\ldots, #1\}}
\newcommand{\rec}[1]{(#1)\mathfrak{r}}
\newcommand{\mrec}[1]{(#1)\overline{\mathfrak{r}}}
\newcommand{\revrec}[1]{(#1){\protect\overleftarrow{\mathfrak{r}}}}

\renewcommand{\pmod}[1]{\ (\mathrm{mod}\ #1)}

\renewcommand{\restriction}{\mathord{\upharpoonright}}

\makeatletter
\renewcommand*{\eqref}[1]{%
  \hyperref[{#1}]{\textup{\tagform@{\ref*{#1}}}}%
}
\makeatother

\begin{document}

\title{Automorphisms of the two-sided shift and the Higman--Thompson groups III: extensions}
\author{Feyishayo Olukoya}
\date{}
\maketitle

\begin{abstract}
	A key result of the article \textit{Automorphisms of shift spaces and the Higman--Thompson groups: the two-sided case} \cite{BelkBleakCameronOlukoya} is  that the group of automorphisms of the two-sided dynamical system $\aut{\xnz, \shift{n}}$ is a central extension by a subgroup $\Ln{n}$ of the outer-automorphism group $\On$ of the Higman--Thompson group $G_{n,n-1}$. 
	
	The current article builds on the results of \cite{BelkBleakCameronOlukoya} by further  establishing connections between the groups $\aut{\xnz, \shift{n}}$ and $\On$. Our key aims are to interpret several important constructions in the theory of automorphisms of the shift dynamical system in terms of the group $\Ln{n}$ and secondly to extend results and techniques in $\aut{\xnz, \shift{n}}$ to the groups $\aut{G_{n,r}}$ of automorphisms of the Higman--Thompson group $G_{n,r}$ and their outerautomorphism groups $\Onr \cong \out{G_{n,r}}$. 
	
	Our mains results are a concrete realisation of the \textit{inert subgroup}, an important subgroup in the study of the groups $\aut{\xns, \shift{n}}$, as a subgroup $\mathcal{K}_{n}$ of $\Ln{n}$ using the technology developed in \cite{BelkBleakCameronOlukoya}. 
	
	Using the realisation above, and by passing to a subgroup $\mathcal{D}_{n}$ of $\mathcal{K}_{n}$ which lifts as a subgroup of $\aut{G_{n,r}}$ for all $1 \le r \le n-1$, we show that the $\aut{G_{n,r}}$ contains an isomorphic copy of $\aut{X_{m}^{\Z}, \shift{m}}$ for all $m \ge 2$. 
	A quick survey of the literature then yields that $\aut{G_{n,r}}$ contains isomorphic copies of  finite groups, finitely generated abelian groups, free groups, free products of finite groups,  fundamental groups of 2-manifolds, graph groups and countable locally finite residually finite groups \cite{KimRoush} to name a few. Another consequence of the embedding is  that both $\aut{G_{n,r}}$ and $\out{G_{n,r}}$ have undecidable order problem. 
	We note that the result for $\out{G_{n,r}}$ is proved in \cite{BelkBleakCameronOlukoya} using a slightly different method.
	
	Lastly, we extend a fairly straight-forward result for $\aut{\xnz, \shift{n}}$ to the groups $\Onr$. The homeomorphism $\rev{\phantom{a}}$, an involution, of $\xnz$ which maps a sequence $(x_i)_{i \in \Z}$ to the sequence $(y_{i})_{i \in \Z}$ defined such that $y_{i} = x_{-i}$ induces by conjugation an automorphism $\rev{\mathfrak{r}}$ of $\aut{\xnz, \shift{n}}$, and consequently, an automorphism of $\Ln{n}$. We extend the automorphism $\rev{\mathfrak{r}}$ to the group $\On$. As there is, to our knowledge, no index preserving action of $\On$ on $\xnz$, the extension to $\On$ requires a different proof. Our proof is combinatorial in nature and relies on the faithful representation of $\On$ in the permutation group on the set of equivalence classes of prime words under rotational equivalence.
	
	In a forthcoming article, we demonstrate that the group $\mathcal{O}_{n}$ is isomorphic to the mapping class group of the full two-sided shift over $n$ letters.

\end{abstract}

\section{Introduction}
A key result of the article \cite{BelkBleakCameronOlukoya} is  that the group of automorphisms of the two-sided dynamical system $\aut{\xnz, \shift{n}}$ is a central extension by a subgroup $\Ln{n}$ of the outer-automorphism group $\On$ of the Higman--Thompson group $G_{n,n-1}$. It is also shown in \cite{BleakCameronOlukoya} that the extension splits precisely when $n$ is not a power of a strictly smaller integer.  

The current article builds on the results of \cite{BelkBleakCameronOlukoya} by further  establishing connections between the groups $\aut{\xnz, \shift{n}}$ and $\On$. Our key aims are to interpret several important constructions in the theory of automorphisms of the shift dynamical system in terms of the group $\Ln{n}$ and secondly to extend results in $\aut{\xnz, \shift{n}}$ to the groups $\aut{G_{n,r}}$ of automorphisms of the Higman--Thompson group $G_{n,r}$ and their outerautomorphism groups $\Onr$.

 Relative to the first aim our main results are as follows.
 
 First we note briefly that elements of the group $\Ln{n}$, and related subgroups $\Ln{n,r} = \Ln{n} \cap \Onr$, are finite state machines satisfying called \textit{transducers} which satisfy a strong \textit{synchronizing} condition. Moreover all states of an element of $\Ln{n}$ induce continuous and injective maps on the Cantor space $\xnn$ of infinite sequences in the alphabet $\xn$. Thus, we may speak of the image of a state. For a  word $\nu \in \xns$, the subset of $\xnn$ consisting of all elements with a prefix $\nu$ is open and the collection of all such open sets forms a basis for the Cantor space $\xnn$. Write $\Kn{n} \le \Ln{n}$ for the set of those elements $T \in \Ln{n}$ such that the image of a state (and consequently any state) of $T$ is a union of a power of $n$ basic open sets. By results in \cite{BelkBleakCameronOlukoya}, $\Kn{n}$ is in fact a subgroup of $\Ln{n,1}$ and, consequently, a subgroup of $\Ln{n,r}$ for any $1 \le r \le n-1$ (using the inclusion $\Ons{1} \le \Ons{r}$ \cite{AutGnr})
 
 Our first result draws a connection to the \textit{inert subgroup} of $\aut{\xnz, \shift{n}}$ which is the kernel of the  \textit{dimension representation} \cite{KriegerDimension, BoyleMarcusTrow}. Briefly, the  dimension representation is a map $\mathfrak{d}_{n}$ from $\aut{\xnz}$ to the finitely generated abelian group $\Z^{r-1}$ where $r$ is the number of prime divisors of $n$. 
 
 Our result can be stated as follows:

 \begin{Theorem}\label{thm:intro1}
 	Let $n$ be a natural number bigger than $2$. Let $p_1, p_2, \ldots, p_r$ be distinct prime numbers and let $l_1, l_2, \ldots, l_r$ be elements of $\N_{1}$ such that $n = p_1^{l_1}p_2^{l_2}\ldots p_{r}^{l_r}$  and set $l = \gcd(l_1, l_2,\ldots, l_r)$. Then:
 	\begin{enumerate}[label=\alph*.]
 		\item $\Kn{n} \cong \mathrm{Inert}_{n}$ 
 		\item  $\Ln{n} \cong \Kn{n} \rtimes (\Z^{r-1} \times \Z/l\Z)$
 		\item $\aut{\xnz} \cong \Kn{n} \rtimes \Z^{r}$
 	\end{enumerate}
 \end{Theorem}

We prove this result by extending a homomorphism $\rsig$ (see \cite{OlukoyaAutTnr,BelkBleakCameronOlukoya}) from $\Ln{n}$ unto the subgroup $\Un{n,n-1}$ of the group of units of $\Z_{n-1}$ generated by the divisors of $n$, to a map $\rsig_{\omega}$ from $\Ln{n}$ unto the group  $\Mn{n}$ defined as the quotient of the submonoid of $\N$ generated by the prime divisors of $n$ by a certain congruence.  This congruence relates two elements of the submonoid where one is a power of $n$ times the other. The resulting quotient monoid is a finitely generated abelian group isomorphic to $\Z^{r-1} \times \Z/l \Z$ where $n= p_1^{L_1}p_2^{L_2} \ldots p_{r}^{L_r}$ for distinct prime numbers $p_{i}$ and $l = \gcd(l_1, \ldots, l_r)$. In particular we may write $\Ln{n} = \Kn{n} \times (\Z^{r-1} \times \Z/ l\Z)$ where $\Kn{n}$ is the kernel of the homomorphism $\rsig_{\omega}$. This decomposition is similar to well known results about the quotient of the group $\aut{\xnz, \shift{n}}$ by the kernel $\mathrm{Inert}_{n}$ of the \emph{dimension representation}, the \emph{ inert subgroup}. We note that in that decomposition, we have $\aut{\xnz, \shift{n}} = \mathrm{Inert}_{n} \rtimes \Z^{r}$ where $r$ is the number of prime divisors of $n$ (\cite{KriegerDimension, BoyleMarcusTrow}). Using the fact that $\aut{\xnz, \shift{n}}$ embeds naturally into the group $\aut{X_{n^{m}}^{\Z}, \shift{n^{m}}}$ for $m \in \N$ and lifting the homomorphism from $\Ln{n}$ to $\Un{n,n-1}$ one obtains a homomorphism from $\aut{\xnz, \shift{n}}$ into the group $\Pi_{k \in \N}\Un{n^{k}, n^{k-1}}$. The kernel of this homomorphism  turns out to be isomorphic to the group $\Kn{n}$ and in this way we obtain Theorem~\ref{thm:intro1}. {We should mention that the description of the group $\Kn{n}$ given above is known also to Belk and Bleak using different methods \cite{BelkBleakPrivate}. }

Theorem~\ref{thm:intro1} paves the way for accomplishing our second aim by yielding a means of concretely realising and identify elements of the inert subgroup. Taking advantage of this new realisation, we revisit various marker constructions existing in the literature (\cite{BoyleFranksKitchens, BoyleKrieger, VilleS18, Hedlund69, BoyleLindRudolph88}). We show that there is a combinatorially defined infinite index subgroup $\mathcal{D}_{n}$ of $\Kn{n}$ such that elements of $\Kn{n}$ arising from the marker constructions in the articles above cited give rise to elements in the subgroup $\mathcal{D}_{n}$. The subgroup $\mathcal{D}_{n}$ in the way it is defined can naturally be identified with a subgroup of $\aut{G_{n,1}}$ and generally of $\aut{G_{n,r}}$. This fact, together with results in the articles cited above, yields the following:

\begin{Theorem}\label{thm:intro2}
	For any $1 \le r \le n-1$ and any $m \ge 2$, $\out{G_{n,r}}$ and $\aut{G_{n,r}}$ contain an isomorphic copy of $\aut{X_{m}^{\Z}, \shift{m}}$.
\end{Theorem}

We note that the article \cite{BelkBleakCameronOlukoya} shows that $\aut{X_{m}^{\Z}, \shift{m}}$ embeds as a subgroup of $\Onr$ but says nothing about an embedding in $\aut{G_{n,r}}$. The result above gives a new embedding into $\Onr$, specifically into $\mathcal{D}_{n}$, and lifts the embedding to $\aut{G_{n,r}}$  by lifting $\mathcal{D}_{n}$. A quick survey of the literature now demonstrates that $\aut{G_{n,r}}$ and $\Onr$ contain isomorphic copies of the following groups: finite groups, finitely generated abelian groups, free groups \cite{BoyleLindRudolph88,Hedlund69}; free products of finite groups,  fundamental groups of 2-manifolds, graph groups and countable locally finite residually finite groups \cite{KimRoush} to name a few.

We further note that since the order problem is undecidable in $\Kn{n} \cong \mathrm{Inert}_{n}$, we obtain from Theorem~\ref{thm:intro2} another proof of the result in \cite{BelkBleakCameronOlukoya} that the order problem is undecidable in the group $\Onr$ and $\aut{G_{n,r}}$ for any $1 \le r \le n-1$. 

Lastly, we extend a fairly straight-forward result for $\aut{\xnz, \shift{n}}$ to the groups $\Onr$. The homeomorphism $\rev{\phantom{a}}$, an involution, of $\xnz$ which maps a sequence $(x_i)_{i \in \Z}$ to the sequence $(y_{i})_{i \in \Z}$ defined such that $y_{i} = x_{-i}$ induces by conjugation an automorphism $\rev{\mathfrak{r}}$ of $\aut{\xnz, \shift{n}}$. Since this automorphism preserves the centre of the group, we have an induced automorphism of $\Ln{n}$. We extend the automorphism $\rev{\mathfrak{r}}$ to the group $\On$ (that is the extension coincides with $\rev{\mathfrak{r}}$ on $\Ln{n}$).

\begin{Theorem}\label{thm:intro3}
	There is a non-trivial automorphism $\rev{\mathfrak{r}}$ of order $2$ of the group $\On$.
\end{Theorem}

 There is no obvious action of $\On$ on $\xn^{\Z}$ by homeomorphisms, thus the proof of Theorem~\ref{thm:intro3} is not as straight-forward as for $\aut{\xnz, \shift{n}}$. In particular, the proof outlined above does not go through and we need to do something different to prove Theorem~\ref{thm:intro3}. Our proof approach makes use of the representation of elements of $\On$ by transducers and the result in  \cite{BelkBleakCameronOlukoya} that there is a faithful representation of $\On$ in the symmetric group $\sym{\rwnl{\ast}}$ of the set  $\rwnl{\ast}$. (Here $\rwnl{\ast}$ is the set of equivalence classes of prime words on an alphabet of size $n$ under the equivalence relation which relates words which are rotations of each other.) We show that reversing the arrows in a \emph{Moore} diagram for a transducer for $\On$ gives a faithful action of $\On$ on $\rwnl{\ast}$ by considering the action on circuits. We then give an algorithm which recovers from the `reversed transducer' a representative in $\On$  with the same action on $\rwnl{\ast}$. It is of interest to know whether the map $\rev{\mathfrak{r}}$ restricts to an automorphism of $\Ons{1}$. More generally, one can ask if the group $\Ons{1}$ is a characteristic subgroup of $\On$. We hope to address these in subsequent articles.  Indeed, as a first step, in a forthcoming article we demonstrate that the group $\On$ is isomorphic to the mapping class group of the full shift (see \cite{BoyleChuy,BoyleCarlsenEilers,BoyleFlowquivPosFact,Chuysurichay} for instance) over $n$ letters.

 The paper is organised as follows.  In Section~\ref{Preliminaries} we give the necessary defintions and background results. Section~\ref{sec:extmorph} considers the various extensions of the homomorphism from $\Ln{n}$ to the group $\Un{n-1,n}$ and develops consequences of these. Section~\ref{Section:revanddimgroup} establishes the isomorphism between the group $\Kn{n}$ and the group $\mathrm{Inert}_{n}$. In this section we also extend the automorphism $\rev{\mathfrak{r}}$ to the group $\On$. In Section~\ref{Section:markersandhomeostate}, where the paper concludes, we consider a combinatorially defined subgroup $\mathcal{D}_{n}$ of $\Kn{n}$; we show that the index of $\mathcal{D}_{n}$ in $\Kn{n}$ is infinite and demonstrate that various marker constructions yield elements of $\aut{\xnz, \shift{n}}$ which reside in $\mathcal{D}_{n}$. Pulling together results of the literature then yields various embedding results for the groups $\aut{G_{n,r}}$ and $\out{G_{n,r}}$.
 
 \section*{Acknowledgements}
 The author wishes to acknowledge support from EPSRC research grant EP/R032866/1, Leverhulme Trust Research Project Grant RPG-2017-159 and the LMS Early Career Fellowship grant ECF-1920-105.

\section{Preliminaries}\label{Preliminaries}
We introduce notation and terminology that we make use of in this work. These can be divided into two categories: those related to the free monoid over an alphabet and the Cantor space of singly- and bi-infinite sequences over a finite alphabet; and those  pertaining to transducers. 

There is one minor piece of notation which does not fit in the above categories that is work mentioning. For $i \in \N$, we write $\N_{i}$ for the set of natural number bigger than or equal to $i$.

\subsection{Words and Cantor space}\label{Section:wordsandcantor}
Throughout $\xn$ will denote the set $\nset{n}$.  Let $X$ be a finite set, we write $X^{\ast}$ for the set of all finite strings (including the empty string $\emptyword$) over the alphabet $X$. Write $X^{+}$ for the set $X^{\ast} \backslash \{\ew\}$. For an element $w \in X$, we write $|w|$ for the length of $w$. Thus, the empty string is the unique element of $X^{\ast}$ such that $|\ew| = 0$. For elements $x,y \in X^{\ast}$ we write $xy$ for the concatenation of $x$ and $y$.  Give $x \in X^{+}$ and $k \in \N$, we write $x^{k}$ for the word $w_1w_2\ldots w_{k}$ where $w_i = x$ for all $1 \le i \le k$.

Given $k \in \N$ we write $X^{k}$ for the set of all words in $X^{+}$ of length exactly $k$.

Let $w \in X^{+}$ and write $w= w_1 w_2 \ldots w_{r}$ for $w_i \in X$, $1 \le i \le r$. Then for some $j \in \N$, $1 \le j \le r$ we say that $w' = w_{j}\ldots w_{r}w_1 \ldots w_{r}$ is the \emph{$j$'th rotation of $w$}. A word $w'$ is called a \emph{rotation} of $w$ if it is a $j$\textsuperscript{th} rotation of $w$ for some $1 \le j \le |w|$. If $j > 1$, then $w'$ is called a \emph{non-trivial rotation of} $w$. The word $w$ is said to be a prime word if $w$ is not equal to a non-trivial rotation of itself. Alternatively, $w$ is a prime word if there is no word $\gamma \in X^{+}$ such that $w = \gamma ^{|w|/|\gamma|}$. 

The relation on $X^{\ast}$ which relates two words if one is a  rotation of the other is an equivalence relation. We write $\simrot$ for this equivalence relation. It will be clear from the context which alphabet is meant.

We write $\wns$ for the set of all prime words over the alphabet $\xn$. For $k \in \N$, $\wnl{k}$ will denote the set of all prime words of length $k$ over the alphabet. We shall write $\rwnl{\ast}$ for the set of equivalence classes $\wns/\simrot$; $\rwnl{+}$  for the set of equivalence classes $\wnl{+}/\simrot$; $\rwnl{k}$ for the set of equivalence classes $\wnl{k}/\simrot$.

 For $x,y \in X^{\ast}$ we write, $x \le y$ if $x$ is a prefix of $y$, in this case we write $y-x$ for the word $z \in X^{\ast}$ such that $xz = y$. If $x \not \le y$ and $y \not \le x$, then we say that $x$ and $y$ are \emph{incomparable} and write $x \perp y$ to denote this.

An \emph{antichain} is a subset $W$ of $X^{\ast}$ consisting of pairwise incomparable elements. That is for $u,v \in W$, $u \not\perp v$ if and only if $u=v$. An antichain $W$ is called \textit{complete} if every element of $X^{*}$ of long enough length has a prefix in $W$.

A \emph{bi-infinite sequence} over the alphabet $X$ is a map $x: \Z \to X$ and we write such a sequence as $\ldots x_{-2}x_{-1}x_{0}x_1 x_2\ldots$ where $x_{i} = x(i)$ for $ i \in \Z$.  In a similar way a \emph{(positive) singly infinite sequence} is a map $x: \N \to  X$ and, analogously, we write such a sequence as $x_{0}x_{1}\ldots$. A \emph{(negative)) singly infinite sequence} is a map $: x : -\N \to X$ and is denoted $\ldots x_{-2}x_{-1}x_{0}$. We write $X^{\N}, X^{-N}$ and $X^{\Z}$ respectively for the sets of positive singly infinite sequences, negative singly infinite sequences and bi-infinite sequences.

For $x \in (X^{\N} \sqcup X^{-\N} \sqcup X^{\Z} \sqcup X^{\ast})$ and $i \le j \in \Z$ where $ j \le |x|$ and $i$ is bigger than or equal to the smallest index of $x$ ($0$ if $x \in X^{\N} \sqcup X^{\ast}$ and $- \infty$ otherwise), we write $x_{[i,j]}$ for the sequence $x_{i}x_{i+1} \ldots x_{j}$.

We may concatenate, in a natural way, a string with a positive infinite sequence or a negative infinite sequence with a positive infinite sequence. Given an element $\gamma \in X^{+}$ and a positive infinite sequence $x$ with prefix $\gamma$, we write $x - \gamma$ for the positive infinite sequence $y$ such that $\gamma y = x$.

Let $\rev{\phantom{a}}: X^{\ast} \to X^{\ast}$ be defined as follows: $\rev{ \ew } = \ew$, and for $w \in X^{+}$ with $w = w_1 \ldots w_r$, $w_i \in X$ for all $1 \le i \le r$, $\rev{w} = w_rw_{r-1} \ldots w_{1}$. We also write $\rev{\phantom{a}}$ for the map from $X^{\N} \sqcup X^{-\N}\to X^{\N} \sqcup X^{-\N}$ defined by $\rev{x_0 x_1 \ldots} = \ldots x_{2}x_{1}x_{0}$ and $\rev{\ldots x_{2}x_{1}x_{0}} =x_0 x_1 \ldots$ . In particular for any prefix $w$ of $x_{0}x_{1}\ldots$. $\rev{w}$ is a suffix of  $\rev{x_{0} x_{1}\ldots }$ and vice versa. We note that $\rev{\phantom{a}}$ is an involution.

Taking the product topology on each of the sets $\xno$, $\xnN$ and $\xnz$, makes each homeomorphic to Cantor space. We introduce some notation for basic open sets in these spaces.  Let $\gamma \in \xns$ be a word. We denote by $U^{+}_{\gamma}$ the subset of $\xno$ consisting of all elements with prefix $\gamma$; $U^{-1}_{\gamma}$ the subset of $\xnN$ consisting of all elements with suffix $\gamma$;  we use the notation $U^0_{\gamma}$ for the subset of all elements of  $\xnz$ such that $x_{0}\ldots x_{|\gamma|-1} = \gamma$.

The shift map  $\shift{n}: \xnz \sqcup \xnN \to \xnz \sqcup \xnN$ is the map which sends a sequence $x \in \xnz \sqcup \xnN$ to the sequence $y \in \xnz \sqcup \xnN$ defined by $y_{i} = x_{i-1}$, for all $i$ for which $x_i$ is defined.  We retain the symbol $\shift{n}$ for  the restrictions to the subspaces $\xnz$ and $\xnN$.  We note that $\shift{n} \xnN \to \xnN$ is a continuous surjection, and  $\shift{n}: \xnz \to \xnz$ is a homeomorphism.

Let $F(X_n,m)$ denote the set of functions from $X_n^m$ to $X_n$. Then, for
all $m,r>0$, and all $f\in F(X_n,m)$, we define a map $f_r:X_n^{m+r-1}\to X_n^{r}$
as follows.
\begin{quote}
	Let $x=x_{-m-r+2}\ldots x_0$. For $-r+1 \le i\le 0$, set
	$y_i=(x_{i-m+1}x_{i-m+2}\ldots x_{i})f$. Then $xf_r=y$, where $y=y_{-r+1}\ldots y_{0}$.
\end{quote}

Thus, as in the paper \cite{BleakCameronOlukoya}, we have a ``window'' of length $m$ which slides along the sequence $x$ and at the $i$'th step we apply the map $f$ to the symbols visible in the window. We think of the map $f$ as acting on the rightmost letter in the viewing window according to the $m-1$ digits of history.  This procedure can
be extended to define a map $f_\infty:\xnz\to\xnz$, by setting $xf_\infty=y$
where $y_i=(x_{i-m+1}\ldots x_i)f$ for all $i\in\mathbb{Z}$; and similarly for
$\xnN$.

A function $f\in F(X_n,m)$ is called \emph{right permutive} if, for distinct
$x,y\in X_n$ and any fixed block $a\in X_n^{m-1}$, we have $(ax)f\ne(ay)f$. Analogously, a function $f \in  F(X_n, m)$ is called \emph{left permutive} if  the map from $X_n$ to itself given by $x\mapsto(xa)f$ is a permutation for all $a\in X_n^{m-1}$. If $f$ is not right permutive, then
the induced map $f_\infty$ from $\xnN$ to itself is not injective. Moreover, it not always the case that a right permutive map $f \in F(X_n, m)$ induces a bijective map $f_{\infty}: \xnN\to \xnN$, however, a right permutive map always induces a surjective map from $\xnN$ to itself.

\begin{Remark}
	Observe that, if $f\in F(X_n,m)$ and $k\ge1$, then the map $g\in F(X_n,m+k)$
	given by $(x_{-m-k+1}\ldots x_0)g=(x_{-m+1}\ldots x_0 )f$, satisfies
	$g_\infty=f_\infty$.
	\label{F(X_n,m)containedinF(X_n,m+1)}
\end{Remark}

\medskip

The following result is due to Curtis, Hedlund and Lyndon~\cite[Theorem 3.1]{Hedlund69}: 

\begin{Theorem}
	Let $f\in F(X_n,m)$. Then $f_\infty$ is continuous and commutes with the
	shift map on $\xnz$.
	\label{t:hed1}
\end{Theorem}

A continuous function from $\xnz$ to itself which commutes with the shift map
is called an \emph{endomorphism} of the shift dynamical system
$(\xnz,\shift{n})$. If the function is invertible, since $\xnz$ is compact and Hausdorff, its inverse is continuous:
it is an \emph{automorphism} of the shift system. The sets of endomorphisms and
of automorphisms are denoted by $\End(\xnz,\shift{n})$ and $\aut{\xnz,\shift{n}}$
respectively. Under composition, the first is a monoid, and the second a group.

Analogously, a continuous
function from $\xnN$ to itself which commutes with the shift map on this
space is an \emph{endomorphism of the one-sided shift} $(\xnN,\shift{n})$; if it is invertible, it is an \emph{automorphism} of this shift system. The sets of
such maps are denoted by $\End(\xnN,\shift{n})$ and $\aut{\xnN,\shift{n}}$; again
the first is a monoid and the second a group.

Note that $\shift{n}\in\aut{\xnz,\shift{n}}$, whereas
$\shift{n}\in\End(\xnN,\shift{n})\setminus\aut{\xnN,\shift{n}}$.

Define
\begin{eqnarray*}
	F_\infty(X_n) &=& \bigcup_{m\ge0}\{f_\infty:f\in F(X_n,m)\},\\
	RF_\infty(X_n) &=& \bigcup_{m\ge0}\{f_\infty:f\in F(X_n,m), f
	\mbox{ is right permutive}\}.
\end{eqnarray*}

Theorem~\ref{t:hed1} shows that $F_\infty(X_n)\subseteq\End(\xnz)$. In fact $F_{\infty}(X_{n})$ and $RF_{\infty}$ are submonoids of $\End(\xnz)$. Note that $\shift{n}\in F_\infty(X_n)$, but $\shift{n}^{-1}$ is not an element of $F_{\infty}(X_{n})$. Now, \cite[Theorem 3.4]{Hedlund69} shows:

\begin{Theorem}
	$\End(\xnz,\shift{n}) = \{ \shift{n}^{i} \phi \mid i \in \mathbb{Z}, \phi \in F_{\infty}(X_{n}) \}$.
	\label{t:hed2}
\end{Theorem}

The following result is a corollary:
\begin{Theorem}
	$RF_\infty$ is a submonoid of $\End(\xnz, \shift{n})$ and $\aut{\xnN,\shift{n}}$ is the largest inverse closed subset of $RF_\infty(X_n)$.
\end{Theorem}

\subsection{Transducers}

In this section we state a result of \cite{BelkBleakCameronOlukoya} which shows how the elements of $\aut{\xnz, \shift{n}}$ can be
described by a pair consisting of a possibly \emph{asynchronous} transducer, and combinatorial data called an \emph{annotation}. We begin with a general definition of automata and transducers. 

\subsection{Automata and transducers}

An \emph{automaton}, in our context, is a triple $A=(X_A,Q_A,\pi_A)$, where
\begin{enumerate}
	\item $X_A$ is a finite set called the \emph{alphabet} of $A$ (we assume that
	this has cardinality $n$, and identify it with $X_n$, for some $n$);
	\item $Q_A$ is a finite set called the \emph{set of states} of $A$;
	\item $\pi_A$ is a function $X_A\times Q_A\to Q_A$, called the \emph{transition
		function}.
\end{enumerate}
We regard an automaton $A$ as operating as follows. If it is in state $q$ and
reads symbol $a$ (which we suppose to be written on an input tape), it moves
into state $\pi_A(a,q)$ before reading the next symbol. As this suggests, we
can imagine that the inputs to $A$ form a string in $\xn^\mathbb{N}$; after
reading a symbol, the read head moves one place to the right before the next
operation.

We can extend the notation as follows. For $w\in X_n^m$, let $\pi_A(w,q)$ be
the final state of the automaton with initial state $q$ after successively
reading the symbols in $w$. Thus, if $w=x_0x_1\ldots x_{m-1}$, then
\[\pi_A(w,q)=\pi_A(x_{m-1},\pi_A(x_{m-2},\ldots,\pi_A(x_0,q)\ldots)).\]
By convention, we take $\pi_A(\varepsilon,q)=q$.

For a given state $q\in Q_A$, we call the automaton $A$ which starts in
state $q$ an \emph{initial automaton}, denoted by $A_q$, and say that it is
\emph{initialised} at $q$.

An automaton $A$ can be represented by a labelled directed graph, whose
vertex set is $Q_A$; there is a directed edge labelled by $a\in X_a$ from
$q$ to $r$ if $\pi_A(a,q)=r$.

A \emph{circuit} in the automaton $A$ is therefore a word $w \in \xns$ and a state $q \in Q_{A}$ such that $\pi_{T}(w, q) = q$. We say that $w$ is a \emph{circuit based at $q$}. If $|w|= 1$, we say that $w$ is a \emph{loop based at $q$} and $q$ is called a \emph{$w$ loop state}.  A circuit $w$ based at a state $q$ is called \emph{basic} if, writing $w = w_1 \ldots w_{|w|}$,  for all $1 \le i < j < |w|$, $\pi_{T}(w_1\ldots w_i, q) \ne \pi_{T}(w_1\ldots w_{j}, q)$.

A \emph{transducer} is a quadruple $T=(X_T,Q_T,\pi_T,\lambda_T)$, where
\begin{enumerate}
	\item $(X_T,Q_T,\pi_T)$ is an automaton;
	\item $\lambda_T:X_T\times Q_T\to X_T^*$ is the \emph{output function}.
\end{enumerate}
Such a transducer is an automaton which can write as well as read; after
reading symbol $a$ in state $q$, it writes the string $\lambda_T(a,q)$ on an
output tape, and makes a transition into state $\pi_T(a,q)$. An \emph{initial
	transducer} $T_q$ is simply a transducer which starts in state $q$.

We note that the transducers  defined above are \emph{deterministic}. In other words the transition function is in fact a function. In Subsection~\ref{Section:revarrowaut} we consider general \emph{non-deterministic transducers} in which there  might be multiple ways of reading a word from a state. As the more general definition is only required in that subsection we state it there.

In the same manner as for automata, we can extend the notation to allow
the transducer to act on finite strings: let $\pi_T(w,q)$ and $\lambda_T(w,q)$
be, respectively, the final state and the concatenation of all the outputs
obtained when the transducer reads the string $w$ from state $q$.

A transducer $T$ can also be represented as an edge-labelled directed graph.
Again the vertex set is $Q_T$; now, if $\pi_T(a,q)=r$, we put an edge with
label $a|\lambda_T(a,q)$ from $q$ to $r$. In other words, the edge label
describes both the input and the output associated with that edge.

For example, Figure~\ref{fig:shift2} describes a transducer over the alphabet
$X_2$.

\begin{figure}[htbp]
	\begin{center}
		\begin{tikzpicture}[shorten >=0.5pt,node distance=3cm,on grid,auto]
		\tikzstyle{every state}=[fill=none,draw=black,text=black]
		\node[state] (q_0)   {$a_1$};
		\node[state] (q_1) [right=of q_0] {$a_2$};
		\path[->]
		(q_0) edge [loop left] node [swap] {$0|0$} ()
		edge [bend left]  node  {$1|0$} (q_1)
		(q_1) edge [loop right]  node [swap]  {$1|1$} ()
		edge [bend left]  node {$0|1$} (q_0);
		\end{tikzpicture}
	\end{center}
	\caption{A transducer over $X_2$ \label{fig:shift2}}
\end{figure}
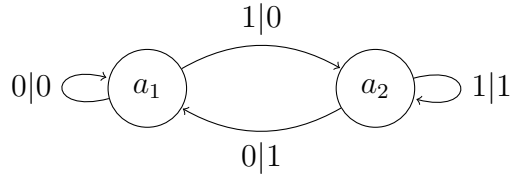

A transducer $T$ is said to be \emph{synchronous} if $|\lambda_T(a,q)|=1$
for all $a\in X_T$, $q\in Q_T$; in other words, when it reads a symbol, it
writes a single symbol. It is \emph{asynchronous} otherwise. Thus, an
asynchronous transducer may write several symbols, or none at all, at a given
step. Note that this usage of the word differs from that of Grigorchuk
\emph{et al.}~\cite{GNSenglish}, for whom ``asynchronous'' includes ``synchronous''.
The transducer of Figure~\ref{fig:shift2} is synchronous.

We can regard an automaton, or a transducer, as acting on an infinite string
in $\xno$, where $X_n$ is the alphabet. This action is given by iterating
the action on a single symbol; so the output string is given by
\[\lambda_T(xw,q) = \lambda_T(x,q)\lambda_T(w,\pi_T(x,q)).\]

Throughout this paper, we will (as in \cite{GNSenglish}) make the following
assumption:

\paragraph{Assumption} A transducer $T$ has the property that, when it reads
an infinite input string starting from any state, it writes an infinite  output string.

\medskip

The property above is equivalent to the property that any circuit in the underlying automaton has non-empty concatenated output. 

From the assumption, it follows that the transducer writes an infinite output string on reading any infinite input string from any state.
Thus $T_q$ induces a map $w\mapsto\lambda_T(w,q)$ from $\xno$ to itself; it is
easy to see that this map is continuous. If it is a
homeomorphism, then we call the state $q$ a \emph{homeomorphism state}. We write $\im{q}$ for the image of the map induced by $T_{q}$.

Two states $q_1$ and $q_2$ are said to be \emph{$\omega$-equivalent} if the
transducers $T_{q_1}$ and $T_{q_2}$ induce the same continuous map. (This can
be checked in finite time, see~\cite{GNSenglish}.)  More generally, we say that two
initial transducers $T_q$ and $T'_{q'}$ are \emph{$\omega$-equivalent} if they
induce the same continuous map on $\xno$.

A transducer is said to be \emph{weakly minimal} if no two states are
$\omega$-equivalent. For a synchronous transducer $T$, two states $q_1$ and $q_2$ are $\omega$-equivalent if $\lambda_T(a, q_1) = \lambda_T(a,q_2)$ for any finite word $a \in X_n^{*}$. Moreover, if $q_1$ and $q_2$ are $\omega$-equivalent states of a synchronous transducer, then for any finite word $a \in X_{n}^{p}$, $\pi_{T}(a, q_1)$ and $\pi_{T}(a, q_2)$ are also $\omega$-equivalent states. 

For a state $q$ of $T$ and a word $w\in X_n^*$, we let $\Lambda(w,q)$ be the
greatest common prefix of the set $\{\lambda(wx,q):x\in\xno\}$. The state
$q$ is called a \emph{state of incomplete response} if
$\Lambda(\varepsilon,q)\ne\varepsilon$; the length $|\Lambda(\varepsilon,q)|$ of the string $\Lambda(\varepsilon,q)$ is the \emph{extent of incomplete response} of the state $q$. Note that for a state $q \in  Q_{T}$, if the initial transducer $T_{q}$  induces a map from $\xno$ to itself with image size at least $2$, then $|\Lambda(\varepsilon, q)| < \infty$. 

Let  $T$ be a transducer and $q \in Q_{T}$ a state. Suppose $T_{q}$ has image size exactly one. By definition the number of states of $T_{q}$ is finite, therefore there is a state $p \in Q_{p}$, accessible from $q$ and a word $\Gamma \in \xnp$ such that $\pi_{T}(\Gamma, p)= p$. Let $x = \lambda_{T}(\Gamma, p) \in \xnp$ and let $Z_{x} = (\xn, \{x\}, \pi_{Z_{x}}, \lambda_{Z_{x}})$ be the transducer such that $\lambda_{Z_{x}}(i, x) = x$ for all $i \in \xn$. Then $T_{q}$, is $\omega$-equivalent to a transducer $T'_{q'}$ which contains $Z_{x}$ as a subtransducer and satisfies the following condition: there is a $k \in \N$ so that, for any word $y \in \xns$, $\pi_{T'}(y,q')= x$.

We say that an initial transducer $T_{q}$ is \emph{minimal} if it is weakly minimal, has no states of
incomplete response and every state is accessible from the initial state $q$. 

A non-initial transducer $T$  is called minimal if for any state $q \in Q_{T}$ the initial transducer $T_{q}$ is minimal.  Therefore a non-initial transducer $T$ is minimal if it is weakly minimal, has no states of incomplete response and is strongly connected as a directed graph. 

Observe that the transducer $Z_{x}$, for any $x \in \xnp$, is not minimal as  its only state is, by definition, a state of incomplete response. However, it will be useful when stating some results later on to have a notion of minimality for the single state transducers $Z_{x}$ considered as initial and non-initial transducers. We use the word minimal for this new meaning of minimality appealing to the context to clarify any confusion. The definition below, which perhaps appears somewhat contrived, is a consequence of this.  

Let $x= x_1 x_2 \ldots x_{r} \in \xns$ be a word. The transducer $Z_{x}$ is called \emph{ minimal as an initial transducer} if $x$ is a prime word; $Z_{x}$ is called \emph{minimal as a non-initial transducer} if  $Z_{x}$ is minimal as an initial transducer and $x$ is the minimal  element, with respect to the lexicographic ordering, of the set $\{ x_{i} \ldots x_{r}x_{1}\ldots x_{i-1} \vert 1 \le i \le r \}$. Note that if $|x| = 1$ then $Z_{x}$ is minimal as an initial transduce if and only if it is minimal as a non-initial transducer. In most applications, we consider the transducers $Z_{x}$ as non-initial transducers and  in this case we will omit the phrase `as a non-initial transducer'.

Two (weakly) minimal non-initial transducers $T$ and $U$  are said to be  \emph{$\omega$-equal} if there is a bijection $f: Q_{T} \to Q_{U}$, such that for any $q \in Q_{T}$, $T_{q}$ is $\omega$-equivalent to $U_{(q)f}$. Two (weakly) minimal initial transducers $T_{p}$ and $U_{q}$ are said to be $\omega$-equal if there is a bijection  $f: Q_{T} \to Q_{U}$, such that $(p)f = q$ and for any $t \in Q_{T}$,  $T_{t}$ is $\omega$-equivalent to $U_{(t)f}$. We shall use the symbol `$=$' to represent $\omega$-equality of initial and non-initial transducers. Two non-initial transducers are said to be $\omega$-equivalent if they have $\omega$-equal minimal representatives.

In the class of synchronous transducers, the  $\omega$-equivalence class of  any transducer has a unique weakly minimal representative.
Grigorchuk \textit{et al.}~\cite{GNSenglish} prove that the $\omega$-equivalence
class of an initialised transducer $T_q$ has a unique minimal representative, if one permits infinite outputs from finite inputs,
and give an algorithm for computing this representative.  The first step of this algorithm is to create a new transducer $T_{q_{-1}}$ which is $\omega$-equivalent to $T_{q}$ and has no states of incomplete response. As we do not allow our transducers to write infinite strings on finite inputs,  we impose an additional condition: we apply the process for removing incomplete response only in instances where all states accessible from the initial one have finite extent of incomplete response.  

\begin{proposition}\label{prop:algorithmforremovingincompleteresponse}
	
	Let $A_{q_0} = (X_n, Q_A, \pi_A, \lambda_A)$ be a finite initial transducer such that $|\Lambda(\varepsilon, q)|<\infty$ for all $q \in  Q_{A}$. Set $Q'_A:= Q \cup\{q_{-1}\}$, where $q_{-1}$ is a symbol not in $Q_{A}$.  Define:
	\begin{IEEEeqnarray}{rCl}
	\pi'_A(x,q_{-1}) &=& \pi_{A}(x,q_0), \label{transition same}\\
	\pi'_A(x,q) &=& \pi_{A}(x,q),\\
	\lambda_A(x, q_{-1}) &= & \Lambda(x,q_0),\\
	\lambda'_A(x,q) &=& \Lambda(x, q) - \Lambda(\varepsilon,q), \label{redifining output}
	\end{IEEEeqnarray}
	for all $x \in X_n$ and $q \in Q_{A}$.
	
	Then, the transducer $A' = (X_n, Q'_A, \pi'_A, \lambda'_A)$ with the initial state $q_{-1}$ is a transducer which is $\omega$-equivalent to the transducer $A_{q_0}$ and has no states of incomplete response.
	
	In  addition, for every $w \in X_n^{\ast}$, the following inequality holds:
	\begin{equation}
	\lambda'_A(w, q_{-1}) = \Lambda(w, q_0). \label{resulthasnoincompleteresponse}
	\end{equation}
\end{proposition}

\medskip

Throughout this article, as a matter of convenience, we shall not distinguish between $\omega$-equivalent transducers. Thus, for example, we introduce various groups as if the elements of those groups are transducers and not $\omega$-equivalence classes of transducers. 

Given two transducers $T=(X_n,Q_T,\pi_T,\lambda_T)$ and
$U=(X_n,Q_U,\pi_U,\lambda_U)$ with the same alphabet $X_n$, we define their
product $T*U$. The intuition is that the output for $T$ will become the input
for $U$. Thus we take the alphabet of $T*U$ to be $X_n$, the set of states
to be $Q_{T*U}=Q_T\times Q_U$, and define the transition and rewrite functions
by the rules
\begin{eqnarray*}
	\pi_{T*U}(x,(p,q)) &=& (\pi_T(x,p),\pi_U(\lambda_T(x,p),q)),\\
	\lambda_{T*U}(x,(p,q)) &=& \lambda_U(\lambda_T(x,p),q),
\end{eqnarray*}
for $x\in S_n$, $p\in Q_T$ and $q\in Q_U$. Here we use the earlier 
convention about extending $\lambda$ and $\pi$ to the case when the transducer
reads a finite string.  If $T$ and $U$ are initial with initial states $q$ and $p$ respectively then the state $(q,p)$ is considered the initial state of the product transducer $T*U$.

We say that an initial transducer $T_q$ is \emph{invertible} if there is an initial transducer $U_p$ such that $T_q*U_p$ and $U_p*T_q$ each induce the identity map
on $\xno$. We call $U_p$ an \emph{inverse} of $T_q$.  When this occurs we will denote $U_p$ as $T_q^{-1}$.

In automata theory a synchronous (not necessarily initial) transducer $$T = (X_n, Q_{T}, \pi_{T}, \lambda_T)$$ is \emph{invertible} if for any state $q$ of $T$, the map $\rho_q:=\pi_{T}(\centerdot, q): X_{n} \to X_{n}$ is a bijection. In this case the inverse of $T$ is the transducer $T^{-1}$ with state set $Q_{T^{-1}}:= \{ q^{1} \mid q \in Q_{T}\}$, transition function $\pi_{T^{-1}}: X_{n} \times Q_{T^{-1}} \to Q_{T^{-1}}$ defined by $(x,p^{-1}) \mapsto q^{-1}$ if and only if $\pi_{T}((x)\rho_{p}^{-1}, p) =q$, and output function  $\lambda_{T^{-1}}: X_{n} \times Q_{T^{-1}} \to X_{n}$ defined by  $(x,p) \mapsto (x)\rho_{p}^{-1}$. 

In this article, we will come across synchronous transducers which are not invertible in the automata theoretic sense but which nevertheless induce self-homeomorphims of the space $\xnz$ and so are invertible in a different sense.

\subsection{Synchronizing automata and bisynchronizing transducers}\label{Subsection:synchandbisynch}

Given a natural number $k$, we say that an automaton $A$ with alphabet $X_n$
is \emph{synchronizing at level $k$} if there is a map
$\mathfrak{s}_{k}:X_n^k\mapsto Q_A$ such that, for all $q$ and any word
$w\in X_n^k$, we have $\pi_A(w,q)=\mathfrak{s}_{k}(w)$; in other words, if the
automaton reads the word $w$ of length $k$, the final state depends only on
$w$ and not on the initial state. (Again we use the extension of $\pi_A$ to
allow the reading of an input string rather than a single symbol.) We
call $\mathfrak{s}_{k}(w)$ the state of $A$ \emph{forced} by $w$; the map $\mathfrak{s}_{k}$ is called the \emph{synchronizing map at level $k$}. An automaton $A$ is called \emph{strongly synchronizing} if it is synchronizing at level $k$ for some $k$.

We remark here that the notion of synchronization occurs in automata theory
in considerations around the \emph{\v{C}ern\'y conjecture}, in a weaker sense.
A word $w$ is said to be a \emph{reset word} for $A$ if $\pi_A(w,q)$ is
independent of $q$; an automaton is called \emph{synchronizing} if it has
a reset word~\cite{Volkov2008,ACS}. Our definition of ``synchonizing at level $k$''/``strongly synchronizing"
requires every word of length $k$ to be a reset word for the automaton.

If the automaton $A$ is synchronizing at level $k$, we define the
\emph{core} of $A$ to be the set of states forming the image of the map
$\mathfrak{s}_{k}$. The core of $A$ is an automaton in its own right, and is also 
synchronizing at level $k$. We denote this automaton by $\core(A)$. We say that an
automaton or transducer is \emph{core} if it is equal to its core.

The following result was proved in Bleak \textit{et al.}~\cite{AutGnr} .

\begin{proposition}
	Let $T_q$, $U_p$ be synchronous initial transducers which (as automata) are synchronizing at levels $j$, $k$ respectively, Then $T*U$ is synchronizing at level at most $j+k$.
	\label{p:synchlengthsadd}
\end{proposition}

Let $T_q$ be an initial transducer which is invertible with inverse $T_q^{-1}$. If $T_q$ is synchronizing at level $k$, and $T_q^{-1}$ is synchronizing at level $l$, 
we say that $T_q$ is \emph{bisynchronizing} at level $(k,l)$. If $T_q$ is
invertible and is synchronizing at level~$k$ but not bisynchronizing, we say
that it is \emph{one-way synchronizing} at level~$k$.

For a non-initial synchronous and invertible transducer  $T$ we also say $T$ is bi-synchronizing (at level $(k,l)$) if both $T$ and its inverse $T^{-1}$ are synchronizing at levels $k$ and $l$ respectively.

\begin{notation}
	Let $T$ be a transducer which is synchronizing at level $k$ and Let $l \ge k$ be any natural number. Then for any word $\Gamma \in X_{n}^{l}$, we write $q_{\Gamma}$ for the state $\mathfrak{s}_{l}(\Gamma)$, where $\mathfrak{s}_{l}: X_{n}^{l} \to Q_{T}$ is the synchronizing map at level $l$.
\end{notation}

\subsection{Groups and monoids of Transducers}

We define several monoids and groups whose elements consists of $\omega$-equivalence classes of transducers which appear first in the papers \cite{AutGnr,BelkBleakCameronOlukoya}.  In practise, as each $\omega$-equivalence class contains a unique minimal element as exposited above, we represent the elements of these monoids will be minimal transducers.  These monoids and groups play a central role in what follows

\subsubsection{The monoid \texorpdfstring{$\spn{n}$}{Lg}}
Let $T$ be a transducer which is synchronous and which (regarded as an
automaton) is synchronizing at level $k$, then the core of $T$ (similarly denoted 
$\core(T)$) induces a continuous map $f_T:\xnz\to\xnz$. The function $f_{T}: \xnz \to \xnz$ induced by a strongly synchronizing transducer $T$ is defined as follows.

Let $T$ be a transducer which is core, synchronous and which is synchronizing at level $k$. The map $f_{T}$ maps an element $x \in \xnz$ to the sequence $y$ defined by $y_{i} = \lambda_T(x_i, q_{x_{i-k}x_{i-k+1}\ldots x_{i-1}})$.  We observe that for an element  $x \in  \xnz$, and $i \in \Z$, if $ (x)f_{T} = y$, then, by definition,  $y_{i}y_{i+1}\ldots = \lambda_{T}(x_i x_{i+1}\ldots, q_{x_{i-k}x_{i-k+1}\ldots x_{i-1}})$.

Clearly, if $T$ is synchronizing at level $k$, then it is synchronizing at
level~$l$ for all $l\ge k$; but the map $f_T$ is independent of the level
chosen to define it.

Now strongly synchronizing transducers may induce endomorphisms of the shift \cite{BleakCameronOlukoya}:

\begin{proposition}\label{prop:pntildeisinendo}
	Let $T$ be a minimal synchronous transducer which is  synchronizing at 
	level $k$ and which is core. Then $f_T\in\End(\xnz,\shift{n})$.
\end{proposition}

The transducer in Figure~\ref{fig:shift2} induces the shift map on $\xnz$. More generally, let  $\Shift{n} = (\xn, Q_{\Shift{n}}, \pi_{\Shift{n}}, \lambda_{\Shift{n}})$ be the transducer defined as follows. Let $Q_{\Shift{n}}:= \{0,1,2,\ldots, n-1\}$, and let $\pi_{\Shift{n}}: \xn \times Q_{\Shift{n}}\to Q_{\Shift{n}}$  and $\lambda_{\Shift{n}}: \xn \times Q_{\Shift{n}}\to \xn$ be defined by $\pi_{\Shift{n}}(x, i) = x$ and $\lambda_{\Shift{n}}(x, i) = i$ for all $x \in X_{n}$ $i \in Q_{\shift{n}}$. Then $f_{\Shift{n}} = \shift{n}$.

In \cite{AutGnr}, the authors show that the set $\spn{n}$ of weakly
minimal finite synchronous core transducers is a monoid; the monoid operation
consists of taking the product of transducers and reducing it by removing non-core states and identifying $\omega$-equivalent states to obtain a weakly minimal and synchronous representative. Let $\mathcal{P}_n$ be the subset
of $\spn{n}$ consisting of transducers which induce automorphisms
of the shift. (Note that these may not be minimal.) Clearly $\Shift{n} \in \pn{n}$.

The following result is proved in \cite{BleakCameronOlukoya}.

\begin{corollary}\label{cor:FinftyisoPn}
	The monoid $F_{\infty}(X_{n})$ is isomorphic to $\spn{n}$.
\end{corollary}

\subsubsection{The monoids \texorpdfstring{$\Smn{n}$, $\SOn{n}$ and $\SLn{n}$}{Lg}} \label{Subsection:MnOnLn}

Let $\Smn{n}$ be the set of all {\bfseries{non-initial}}, minimal, strongly synchronizing and core transducers. Recall that for $x \in \xnp$, $Z_{x}= (\xn, \{x\}, \pi_{Z_{x}}, \lambda_{Z_{x}})$ is the transducer such that $\lambda_{Z_{x}}(i,x) = x$ for all $i \in \xn$. Moreover $Z_{x} \in \Smn{n}$ for a prime words $x = x_0 x_1 \ldots x_r \in \xnp$ if and only if $x$ is the minimal, with respect to the lexicographic ordering, element of the set $\{ x_{i}\ldots x_{r}x_1 \ldots x_{i-1} \vert 1 \le i \le r \}$.  If $T \in \Smn{n}$ is a strongly synchronizing and core transducer, such that some state (and so all states by connectivity) has infinite extent of incomplete response, then $T = Z_{x}$ for some minimal non-initial transducer $Z_{x}$.   Let  $\SOn{n}$ be the subset of $\Smn{n}$ consisting of  all  transducers $T$ satisfying the following conditions:

\begin{enumerate}[label = {\bfseries S\arabic*}]
	\item for any state $t \in Q_{T}$, the initial transducer $T_{t}$ induces a injective map $h_{t}: \xno \to \xno$ and,
	\item for any state $t \in Q_{T}$ the map $h_{t}$ has clopen image which we denote by $\im{t}$.
\end{enumerate}
Note that an element of $\Smn{n}$ with more than one state has no states of incomplete response. The single state identity transducer is an element of $\Smn{n}$ and we denote it by   $\tid$. 

Following \cite{BelkBleakCameronOlukoya} we define a binary operation on the set $\Smn{n}$ as follows. For two transducers $T,U \in \Smn{n}$, the product transducer $T \ast U$ is strongly synchronizing. Let $TU$ be the transducer obtained from $\core(T\ast U)$ as follows. Fix a state $(t,u)$ of $\core(T\ast U)$. Then $TU$ is the core of the minimal representative of the initial transducer $\core(T\ast U)_{(t,u)}$. That is, if any state (and so all states by connectivity) of $\core(T\ast U)$ has finite extent of incomplete response  we remove the states of incomplete response from $\core(T\ast U)_{(t,u)}$ to get a new transducer initial $(TU)'_{q_{-1}}$ which is still strongly synchronizing and has all states accessible from the initial state. Then we identify the $\omega$-equivalent states of $(TU)'_{q_{-1}}$ to get a minimal, strongly synchronizing initial transducer $(TU)_{p}$. The transducer $TU$ is the core of $(TU)_{p}$. Otherwise, for any state $(t,u) \in \core(T,U)$, $\core(T,U)_{(t,u)}$ is $\omega$-equivalent to a transducer of the form $Z_{x}$ for some prime word $x = x_1 \ldots x_r \in \xnp$ and we set $TU = Z_{\overline{x}}$ where $\overline{x}$ is the minimal element, with respect to the lexicographic ordering, of the set $\{ x_i \ldots x_{r}x_1 \ldots _{i-1}\vert 1 \le i \le r \}$.

It is a result of \cite{BelkBleakCameronOlukoya} that $\Smn{n}$ together with the binary operation  $\Smn{n} \times \Smn{n} \to \Smn{n}$ given by $(T,U) \mapsto TU$ is a monoid. (The single state transducer $\tid$ which induces the identity permutation on $X_n$ is the identity element.)

A consequence of a result in \cite{AutGnr} is that an element  $T \in \Smn{n}$ is an element of $\SOn{n}$ if and only if  for any state $q \in Q_T$ there is a minimal, initial transducer $U_{p}$ such that the minimal representative of the product $(T\ast U)_{(p,q)}$ is strongly synchronizing and has trivial core.  If the transducer $U_{p}$ is strongly synchronizing as well then $\core(U_{p}) \in \SOn{n}$ and satisfies $T\core(U_p) = \core(U_p)T =\tid$. 

The group $\On$ is the largest inverse closed subset of $\SOn{n}$. 

We give below a procedure for constructing the inverse of an element $T \in \SOn{n}$.

\begin{construction}\label{construction:inverse}
	Let $T \in \On$. For $q \in Q_{T}$ define a map $L_{q}: \xnp \to \xns$ by setting $(w)L_{q}$ to be the greatest common prefix of the set $(U^{+}_{w})h_{q}^{_1}$.  Set $Q_{T'}$ to be the set of pairs $(w, q)$ where $w \in \xns$, $q \in Q_{T}$, $U^{+}_{w} \subset \im{q}$ and $(w)L_{q}   = \ew$. Define maps $\pi_{T'}: \xn \times Q_{T'} \to Q_{T'}$ and $\lambda_{T'}: \xn \times Q_{T'} \to \xns$ as follow:  $\pi_{T'}(a, (w,q)) = wa - \lambda_{T}((wa)L_{q}, q)$ and $\lambda_{T'}(a, (w,q)) =(wa) L_{q}$. Let $T' = \gen{ \xn, Q_{T'}, \pi_{T'}, \lambda_{T'}}$. It is shown in \cite{AutGnr} that $T'$ is a finite transducer and, setting  $U$ to be the minimal representative of $T'$, $UT = TU = \tid$.
\end{construction}

Let $\SLn{n}$ be those elements $T \in \Smn{n}$ which satisfy the following additional constraint:

\begin{enumerate}[label= {\bfseries SL\arabic*}]
	\setcounter{enumi}{2}
	\item for all strings $a \in X_n^{\ast}$ and any state $q \in Q_{T}$ such that $\pi_{T}(a,q) = q$, $|\lambda_{T}(a,q)| = |a|$.
	\label{Lipshitzconstraint}
\end{enumerate}

The set $\SLn{n}$ is a submonoid of $\Smn{n}$ \cite{AutGnr}. Set $\Ln{n} := \On \cap \SLn{n}$, the largest inverse closed subset of $\SLn{n}$.

For each $1 \le r \le n-1$ the group $\On$ has a subgroup $\Ons{r}$ where $\Ons{n-1} = \Ons{n}$; we define $\Ln{n,r}:= \Ons{r} \cap \Ln{n}$. The following theorem is from \cite{AutGnr}.

\begin{Theorem}
	For $1 \le r \le n-1$, $\Ons{r} \cong \out{G_{n,r}}$.
\end{Theorem}
\begin{comment}
In subsequent sections we establish a map from the monoid $\pn{n}$ to the group $\On$. Theorem~\ref{t:hed2} and Corollary~\ref{cor:FinftyisoPn} show that elements of the quotient $\aut{\xnz, \shift{n}}/\gen{\shift{n}}$ can be represented by elements of $\pn{n}$, therefore the above mentioned map will be crucial in demonstrating that the group $\aut{\xnz, \shift{n}}/ \gen{\shift{n}}$ is isomorphic to $\Ln{n}$. 

In what follows it will be necessary to distinguish between the binary operation in $\spn{n}$ and the operation in $\SLn{n}$. For two elements $P, R \in \spn{n}$, we write $P \spnprod{n} R$ for the element of $\spn{n}$ which is obtained by taking the full transducer product $P * R$, identifying $\omega$-equivalent states and taking the core of the resulting weakly minimal transducer.
\end{comment}

In \cite{BelkBleakCameronOlukoya} it is shown that any element of $\End(\xnz, \shift{n})$ can be represented by a pair $(T, \alpha)$ where $T \in \SLn{n}$ and $\alpha$ is an \emph{annotation} of $T$. We give a brief exposition of this result.

\begin{Definition}\label{def:annotatingLn}
	Let $L \in \SLn{n}$. An \emph{annotation} of $L$ is a map $\alpha: Q_{L} \to \Z$ which obeys the following rule: for any state $q \in Q_{L}$ and any word $\Gamma \in \xns$
	
	\begin{equation}
	(\pi_{L}(\Gamma, q))\alpha = (q)\alpha + |\lambda_{L}(\Gamma, q)| - |\Gamma|. \label{eqn:annotationrule}
	\end{equation}
\end{Definition}

\begin{notation}
	Let $T \in \SLn{n}$ and suppose that $T$ admits an annotation $\alpha$ which assigns $0$ to every state of $T$. Then we call $\alpha$ the \emph{zero annotation} and denote it by $\bm{0}$.
\end{notation}

We are now in a position to define an action of $\SLn{n}$ on $X_n^{\mathbb{Z}}$.

	Let $L \in \SLn{n}$, let $\alpha$ be an annotation of $L$ and let $k \in \mathbb{N}$ be the synchronizing level of $L$. Let $(L, \alpha) : X_n^{\Z} \to X_n^{\Z}$ be defined as follows: for $x \in X_n^{\Z}$ and $i \in \mathbb{Z}$ let $q_j$, $1\le j \le |L|$, be the state of $L$ forced by ${x_{i-k}\ldots x_{i-1}}$. Set $r: = (q_j)\alpha$ and $w: = \lambda_{L}(x_i,q_j)$. Let $y \in X_n^{\Z}$ be defined by ${y_{i+r}y_{i+r+1}\ldots y_{i+r+|w|}} := w$. Set $(x)(L, \alpha) := y$.

The following results are from \cite{BelkBleakCameronOlukoya}.

\begin{proposition} \label{prop:annotationsofLareinEnd(Xn,sigma)}
	Let $L \in \SLn{n}$, and let $\alpha$ be an annotation of $L$. Then $(L,\alpha)$ is an element of $\End(X_n^{\Z}, \shift{n})$.
\end{proposition}

\begin{corollary}\label{cor:annotationsinthesamecoset}
	Let $L \in \SLn{n}$, $i \in \Z$ and $\alpha$ and $\beta$ be any annotations of $L$, then there is an $i \in \Z$ such that $(L, \alpha) = \shift{n}^{i}(L, \beta)$. In particular $\beta = \alpha + i$.
\end{corollary}
 
 Let $L, M \in \SLn{n}$, $\alpha$ an annotation of $L$ and $\beta$ an annotation of $M$. Let  $(LM)$ be the core of the transducer product of $L$ and $M$ with incomplete response removed, and let $LM$ be the product in $\SLn{n}$ of $L$ and $M$. We note that the states of $(LM)$ are in bijective correspondence with the states of $L \ast M$ and the states of $LM$ are $\omega$-equivalence classes of states of $(LM)$. Define a map from $\overline{\alpha+ \beta}:Q_{LM} \to \Z$ as follows. Let $q \in Q_{LM}$ and let $(s,t) \in Q_{L} \times Q_{M}$ be the state of $(LM)$ such that  $(LM)_{(s,t)}$ is $\omega$-equivalent to $LM_{q}$. Then we set $(q)\overline{\alpha + \beta} = (s)\alpha + (t)\beta + |\Lambda(\epsilon,(s,t))|$ where $|\Lambda(\epsilon,(s,t))|$ is the extent of incomplete response of $(s,t)$ as a state of $L \ast M$. It is shown in \cite{BelkBleakCameronOlukoya} that $\overline{\alpha + \beta}$ is an annotation of $LM$.  
 
 Set $\ASLn{n}$ to be the set of pairs $(L, \alpha)$ for $L \in \SLn{n}$ and $\alpha$ an annotation of $L$. Let $\ALn{n}$ be the subset of $\ASLn{n}$ consisting of elements $(L, \alpha)$ such that $L \in \Ln{n}$.The following result is from \cite{BelkBleakCameronOlukoya}
 
 \begin{Theorem}\label{theorem:repsbyLnandannotation}
 	The set $\ASLn{n}$ with the binary operation mapping $(L, \alpha)(M, \beta) \to (LM, \overline{\alpha + \beta})$ is a monoid isomorphic to $\End(\xnz, \shift{n})$. Moreover, the subset $\ALn{n}$ of $\ASLn{n}$ is a subgroup isomorphic to $\aut{\xnz, \shift{n}}$.
 \end{Theorem}

In this work we alternate between the description of elements of $\End(\xnz, \shift{n})$ given in Section~\ref{Section:wordsandcantor} and the description as pairs of consisting of an element $L$ of $\SLn{n}$ and an annotation $\alpha$ of $L$.

The paper \cite{BelkBleakCameronOlukoya} gives a faithful representation of the monoid $\Smn{n}$ in the full transformation monoid $\tran{\rwnl{\ast}}$.  We define this representation below noting that its restriction to the group $\Ln{n}$ yields the so called \emph{periodic orbit} of \cite{BoyleKrieger}.

Let $\rotclass{\gamma} \in  \rwnl{\ast}$ and $T \in \Smn{n}$. There is a unique state $q$ of $T$ such that $\pi_{T}(\gamma, q) = q$. Let $\nu \in \wnl{\ast}$ be a prime word such that $\lambda_{T}(\gamma, q)$ is a power of $\gamma$. We set $(\rotclass{\gamma})(T)\Pi = \rotclass{\nu}$.  The map $\Pi: \Smn{n} \to \tran{\rwnl{\ast}}$ is an injective homomorphism.
\section{A homomorphism}\label{sec:extmorph}

\subsection{On $\Ln{n}$}
In the paper \cite{OlukoyaAutTnr} a homomorphism is defined from $\On \to \Un{n-1}$ which restricts to a homomorphism from $\Ln{n} \to \Un{,n-1n}$. We make use of an observation in the paper \cite{BelkBleakCameronOlukoya} to extend this homomorphism from $\Ln{n}$ to a quotient of a submonoid  of $\N$ by a certain congruence.
We later show that the kernel of this homomorphism is isomorphic to the \emph{inert subgroup} consisting of elements of $\aut{\xnz, \shift{n}}$ belonging to the kernel of the \emph{dimension representation}.  We define these terms in Section 4 where we make explicit the connection to the dimension group. 

We begin by defining first the map $\rsig: \On \to \Un{n-1}$.

\begin{Definition}\label{Def:rsig}
Let $T \in \On$ and  $q \in Q_{T}$. Let $s \in \N_{1}$ be such that there exists distinct pairwise incomparable elements $\mu_1, \mu_2, \ldots, \mu_s \in \xns$ satisfying $\bigcup_{1 \le i \le s} U^{+}_{\mu_i} = \im{q}$. Set $(T)\rsig \in \Un{n-1}$ to be equal to $s \pmod{n-1}$.
\end{Definition}

The following results are from \cite{OlukoyaAutTnr}:

\begin{Theorem}\label{thm:sigOn}
	Let $T \in \On$ then $(T)\rsig$ depends only on $T$ moreover and induces a epimorphism from $\On$ onto the group $\Un{n-1,n}$.
\end{Theorem}

\begin{proposition}\label{Proposition:sigdeterminesmembership}
	Let $T \in \On$, and $1 \le r < n$, then $T \in T \in \Onr$ if and only if $r (T)\sig \equiv r \mod{n-1}$.
\end{proposition}

We compute the images under $\rsig$ of some elements of $\Ln{n}$. The elements in this example will be important later on.

\begin{example} \label{Example:imagegenerators}
	Let $n \in \N_{2}$, $p_1, p_2, \ldots, p_{l}, r_1, r_2, \ldots, r_l \in \N_{1}$ be such that  $p_1^{r_1}p_2^{r_2}\ldots p_l^{r_{l}}$ is the prime decomposition of $n$. For $1 \le i \le l$ fix an alphabet $X_{p_i}$, $1 \le i \le l$, of size $p_i$, we also assume that for $i \ne j$, $X_{p_i} \cap X_{p_j} = \emptyset$ . Set $X_{n}:= X_{p_1}^{r_1} \times X_{p_{2}}^{r_{2}} \ldots \times X_{p_l}^{r_l}$. Fix $1 \le i \le l$ and let $N_{i} = p_1^{r_1} \cdot p_2^{r_2} \cdot \ldots \cdot p_{i}^{r_{i}-1} \cdot \ldots \cdot p_{l}^{r_l}$. 
	
	Define a transducer $T(p_i, N_i)$ as follows. The state set of $T(p_i, N_i)$ is precisely the set $\{  q_x \mid x \in X_{p_i}\}$. The transition and output function are given as follows: for $ x \in X_{p_i}$ and $\gamma:= \gamma_1 \gamma_{2} \ldots \gamma_{l} \in X_{n}$,  $\gamma_j \in X_{p_j}^{r_j}$ for all $1 \le j \le l$, set $\pi_{T(p_i,N_i)}(\gamma, q_x) = q_{y}$ if and only if $y$ is the last  letter of $\gamma_{i}$; set $ \lambda_{T(p_i,N_i)}(\gamma, q_x) := \gamma_{1} \gamma_{2} \ldots \gamma_{i}' \gamma_{i+1} \ldots \gamma_{l}$ where $\gamma'_{i}$ is obtained from $\gamma_{i}$ by deleting its  letter and appending  letter $x$ to the resulting word, i.e $\gamma'_i = x\nu$ where $\nu$ is the length $r_i-1$ prefix of $\gamma_{i}$. 
	
	It is clear that $T(p_i, N_i)$ is an element of  $\SLn{n}$, we show that is in facts an element of $\Ln{n}$ by examining its action on $\xnz$.
	
	The action of  $T(p_i,N_i)$, with the trivial annotation  on $X_{n}^{\Z}$ is best seen  as follows. Let $(\gamma_j)_{j \in \Z } \in \xnz$, and for $j \in \Z$, write $\gamma_j = {\gamma_j^{1}\gamma_{j}^{2}\ldots \gamma_{j}^{l}}$, $\gamma_j^{k} \in X_{p_k}^{r_k}$, $1 \le k \le l$. Then $(\gamma_{j})_{j \in \Z})(T(p_i, N_i), 0)$ is the sequence $(\delta_j)_{j \in \Z}$ defined as follows. For $j \in \Z$, $\delta_j = {\gamma_{j}^{1} \gamma_{j}^{2} \ldots \gamma_{j}^{i-1} \delta_{j}^{i} \gamma_{j}^{i+1}\ldots \gamma_{j}^{l}}$ where $\delta _{j}^{i}$ is obtained from $\gamma_{j}^{i}$ by deleting its  last letter and appending  the last letter of $\gamma_{j-1}^{i}$ to the resulting word. From this, one sees that $(T(p_i, N_i), 0)$ is a shift-commuting homeomorphism of $\xnz$. Moreover, note that  $(\gamma_{j})_{j \in \Z})(T(p_i, N_i), 0)^{r_i}$ is the sequence $(\nu_{j})_{j \in \Z}$ such that $\nu_j  = { \gamma_{j}^{1}\gamma_{j}^{2} \ldots \gamma_{j}^{i-1}\gamma_{j-1}^{i}\gamma_{j}^{i+1}\ldots \gamma_{j}^{l}}$.  
	
	Observe that for $1 \le i_1, i_2 \le l$ such that $i_1 \ne i_2$, $(T(p_{i_1}, N_{i_1}),0)$ and $(T(p_{i_2}, N_{i_2}),0 )$ commute.
	
	Fix a subset  $i_1, i_2, \ldots, i_{t}$ of $\{1,2,\ldots, l \}$ such that $i_{j}\le i_{j'}$ whenever $j \le j'$, and fix  $s_{i_j}$ between $1$ and $r_{i_{j}}$.  Set $d := p_{i_1}^{s_{i_1}}p_{i_2}^{s_{i_2}}\ldots p_{i_t}^{s_{i_t}}$,  and  let $e \in \N $ be such that $de = n$. Generalising the construction above, define  a transducer $T(d,e) \in \Ln{n}$ as follows. 
	
	The set of states of $T(d,e)$ is precisely the set $\{q_{x} \mid x \in X_{p_{i_1}}^{s_{i_1}} \times X_{p_{i_2}}^{s_{i_2}} \times \ldots \times X_{p_{i_t}}^{s_{i_t}}\}$. The transition and output are as follows. Given $\gamma = \gamma_1\gamma_{2} \ldots \gamma_{l} \in X_{n}$, $\gamma_{j} \in X_{p_j}^{r_j}$ for all $1 \le j \le l$, and $x = x_{i_1} \ldots x_{i_t} \in X_{p_{i_1}}^{s_{i_1}} \times \ldots \times X_{p_{i_t}}^{s_{i_t}}$, set $\pi_{T(d,e)}(\gamma, q_x) = q_y$ if and only if $y$ is equal to the word  $\overline{\gamma}_{i_{1}}\overline{\gamma}_{i_{2}}\ldots \overline{\gamma}_{i_{t}}$, for $\overline{\gamma}_{i_j}$, $1 \le j \le t$, the length $s_{i_j}$ suffix of $\gamma_{i_j}$. Set $\lambda_{T(d,e)}(\gamma, q_x) = \delta_2 \delta_{2} \ldots \delta_{l}$, where for $j \in \{ i_1, i_2, \ldots, i_{t}\}$, $\delta_{j}$ is obtained from $\gamma_{j}$ by deleting the last $s_{j}$ letters of $\gamma_{j}$ and appending  $x_{j}$ to the resulting word. 
	
	The reader may verify that $T(d,e)$ is precisely  the product $(T(p_{i_1}, N_{i_1}),0)^{s_{i_1}}(T(p_{i_2}, N_{i_2}),0)^{s_{i_2}}\ldots (T(p_{i_t}, N_{i_t}),0)^{s_{i_t}}$.
	
	Observe that $(T(n, 1), 0)$ is the shift-map $\shift{n}$, one may also think  of $(T(1,n),0)$ (although not defined) as representing the identity map.
	
	We now compute the value of $(T(d,e),0)\rsig$ for $d,e \in \N$ such that $de = n$. 
	
	Let $d = p_{i_1}^{s_{i_1}}p_{i_2}^{s_{i_2}}\ldots p_{i_t}^{s_{i_t}}$, where  $i_1 < i_2 < \ldots < i_{t}$ are elements of $\{1,2,\ldots, l\}$. Fix  $x:= x_{i_1} \ldots x_{i_t} \in X_{p_{i_1}}^{s_{i_1}} \times \ldots \times X_{p_{i_t}}^{s_{i_t}}$  a state of $T(d,e)$. Let $V \subseteq \xn$ be the set of those elements $\gamma_{1} \gamma_{2} \ldots \gamma_{l}$ such that  for $1 \le j \le t$, $x_{i_j}$ is a prefix of $\gamma_{i_j}$. Notice that $|V| = e$ and the set ${\{ U^{+}_{\nu} \mid \nu \in V \}}$ is precisely the image of  the state $q_{x}$. Therefore, $(T(d,e),0)\rsig = e \pmod{n-1}$.

	\begin{comment}$d$ be a divisor of $n$ and $e \in \N_{1}$ be such that $de = n$. Define an transducer $T(d,e)$  as follows. Set $Q_{T(d,e)}:= \{1,2,\ldots, d\}$, and let $\pi_{T(d,e)}$ and $\lambda_{T(d,e)}$ be given as follows: for $1 \le i \le d$ a state of $T(d,e)$, given $x \in \xn$, write $x = je + k$, for $1 \le j \le d$ maximal  and $0 \le k \le e$, then  set $\pi_{T(d,e)}(x,i ) := j$ and $\lambda_{T(d,e)}(x,i) := (i-1)e +k $. It is easy to see that $T$ is synchronous and strongly synchronizing at level  $1$.  Furthermore, it is shown in \cite{OlukoyaAutTnr} that $T(d,e) \in \Ln{n}$, therefore $(T(d,e),0)$ represents an element of $\aut{\xnz,\shift{n}}$. By construction, we have that $(T(d,e))\rsig = e \pmod{n-1}$, since the image of a state $1 \le i \le d$, is precisely the set $\{ ((i-1)e + k)\rho \mid \rho \in \xno, 0 \le k\le e-1 \} \subseteq \xno$. Notice that $(T(n,1), 0)$ is precisely the shift map $\shift{n}$ and $T(1,n)$ is the identity map.
	\end{comment}
\end{example}

We define the submonoid of  $\N$ and the congruence we will be quotienting by.

Fix $n \in \N_{2}$. Let $p_1, p_2, \ldots, p_{r}$ be distinct prime number  and $l_1, l_2, \ldots, l_r \in \N_{1}$ such that $n = p_1^{l_1}p_2^{l_2}\ldots p_{r}^{l_r}$. Let $\Mn{n} := \gen{p_1, p_2,\ldots, p_r}$ be the submonoid of $\N$ generated by the distinct primes in the prime decomposition of $n$. Let $\sim$ be the congruence on $\Mn{n}$ defined by $a \sim b$ if and only if there is some $l \in \Z$ such that $an^{l} = b$. 

The group $\mn{n}:= \Mn{n}/\sim$ is  a finitely generated abelian group and is in fact isomorphic, in a natural way to the quotient of the subgroup of $\Q\backslash \{0\}$ generated by the prime divisors of $n$ by the group generated by $n$. Below we show that $\mn{n}$ is isomorphic to $\Z^{r-1} \times  \Z/l\Z$  where $n = p_1^{l_1}p_2^{l_2}\ldots p_{r}^{l_r}$ and   $l = \gcd(l_1, l_2, \ldots, l_r)$.

\begin{lemma}\label{lemma:characterisationofmn}
	Let $n \in \N_{2}$, $p_1, p_2, \ldots, p_{r}$ be distinct prime numbers and $L_1, L_2, \ldots, L_r \in \N_{1}$ such that $n = p_1^{L_1}p_2^{L_2}\ldots p_{r}^{L_r}$. Let $l= \gcd(L_1, L_2, \ldots, L_r)$ and set $l_i = L_i/l$ for $1 \le i \le r$. Then group $\mn{n}$ is isomorphic to $\Z^{r-1} \times \Z/l\Z$ where $\Z^{r-1} \cong \gen{p_1, \ldots, p_{r-1}}$ and $\Z/l\Z \cong \gen{ p_{1}^{l_1} \ldots p_{r}^{l_r}}$. 
\end{lemma}
\begin{proof}
	Let $t \in \mn{n}$ be an non-trivial element of finite order. Write $t = p_1^{s_1}p_2^{s_2} \ldots p_r^{s_r}$. We may assume that $\gcd(s_1, s_2, \ldots, s_r) = 1$ since otherwise $t$ is a power of an element $t'$ of $\mn{n}$ and we may replace $t$ with $t'$. Let $k \in \N$  be such that $t^{k} =1 \in \mn{n}$. This means that $t^{k} = p_1^{ks_1}p_2^{ks_2} \ldots p_r^{ks_r} \in \gen{n}$. This is true if and only if there is an $m \in \Z$ such that $k s_i = m L_i$ for some $1 \le i \le r$. Let $d = \gcd(k,m)$ and set $k' = k/d$ and $m' = m/d$. It follows that $k'| L_i$ for all $i$ and so $k' |l$. Write $l = s k'$, then and $L_i/k' = sl_i$. Therefore $t =  (p_1^{l_1}p_2^{l_2} \ldots p_r^{l_r})^{sm'} \in \gen{p_1^{l_1}p_2^{l_2} \ldots p_r^{l_r}}$.
	
	Next observe that $\gen{p_1, \ldots, p_{r-1}} \cong \Z^{r-1}$ and $\gen{p_1, \ldots, p_{r-1}} \cap \gen{p_1^{l_1}p_{2}^{l_2} \ldots p_{r}^{l_r}} = 1$. This is a straight-forward noting that for $t \in \gen{p_1^{l_1}p_{2}^{l_2} \ldots p_{r}^{l_r}}$ and   $u \in \gen{p_1, \ldots, p_{r-1}}$, $tu^{-1} \notin \gen{n}$ unless $t=u = 1$.
	
\end{proof}
\begin{Remark}
	We note that $\mn{n}$ is a torsion group if and only if $n$ is a power of strictly smaller integer. Moreover, $\mn{n}$ is trivial  if and only if $n$ is prime.
\end{Remark}
 We now define an epimorphism from $\Ln{n}$ to the group $\mn{n}$. We require the following result is from \cite{BelkBleakCameronOlukoya}.

\begin{proposition}\label{Prop:lnpn}
	Let $T \in \Ln{n}$, then there is an element $P \in \pn{n}$ such that the minimal representative of $P$ is $T$. Moreover, letting $[q]$ denote the state of $T$ corresponding to a state $q \in Q_{P}$, the map from $Q_{T} \to \N$ given by $[q] \mapsto  \Lambda(\emptyword, q)$ is a well-defined annotation of $T$. 
\end{proposition}

\begin{Remark}\label{rem:consequencesofsynchronousrep}
There are a few consequences of this proposition. Firstly, observe that by the strong synchronizing condition and results in \cite{BelkBleakCameronOlukoya}, for a given element $P \in \pn{n}$, there is a number $l \in \N$, and a divisor $s \in \N$ of some power of $n$, such that for any state $q \in Q_{P}$, there are elements $\mu_1, \mu_2, \ldots, \mu_s \in \xn^{l}$ and $\bigcup_{1 \le i \le s} U^{+}_{\mu_i} = \im{q}$. The second thing to observe, is that if $T$ is the minimal representative of $P$, then for every $q \in Q_{P}$, the image of the  state $[q]$ is the set $\{ x - \Lambda(\ew, q) \mid x \in \im{q} \}$. Let $\alpha: Q_{T}  \to \N$ be the annotation of $T$ such that $[q] \mapsto \Lambda(\emptyword, q)$. Then, we see that for every state $t \in Q_{T}$, there are words $\mu'_1, \mu'_2, \ldots, \mu'_{s} \in \xn^{l- ([q])\alpha}$ such that $\bigcup_{1 \le i \le s} U^{+}_{\mu'_{i}} = \im{[q]}$. Set $s_{T}$ to be the smallest number such that $s$ is equal to $s_T$ times some power  of $n$. Thus $s_T$ and $s$ represent the same element of $\mn{n}$. The last thing we note is the following fact: let $\beta$ be any other annotation of $T$, then for any states $[p], [q] \in Q_{T}$, $([p])\beta + l-([p])\alpha = ([q])\beta + l-([q])\alpha$. This follows from the fact that there is an $m \in \Z$ such that $\beta = \alpha + m$ by definition of annotations. 
\end{Remark}

\begin{Definition}
	Define a map $\rsig_{\omega}: \Ln{n} \to \mn{n}$ by  $T \mapsto s_{T}$.
\end{Definition}

We have the following result:

\begin{proposition}\label{prop:lnsplit}
	The following things hold:
	\begin{enumerate}
		\item  the map $\rsig_{\omega}$ is a surjective homomorphism; 
		\item  $\ker(\rsig_{\omega}) \unlhd \Ln{n,1}$;
		\item setting $\Kn{n}:= \ker(\rsig_{\omega})$, we have $\Ln{n}  \cong \Kn{n} \rtimes \mn{n}$. Thus if $n$ is prime $\Ln{n} = \Kn{n}$.
	\end{enumerate}
	
\end{proposition}
\begin{proof}
	We first show that $\rsig_{\omega}$ is a homomorphism. 
	
	Let $\overline{U},\overline{T} \in \Ln{n}$ and let  $T, U \in \pn{n}$ be such that $\overline{T}$ and $\overline{U}$ are the minimal representatives of $T$ and $U$ respectively.
	
	Let $s_{T}$ and $s_{U}$ be divisors of some power of $n$ and $l_{T}, l_{U} \in \N$, such that for any pair of states $q \in  Q_{T}$, $p \in   Q_{U}$, there are words $\mu_1, \ldots, \mu_{s_{T}} \in \xn^{l_T}$, $\nu_1 \ldots \nu_{s_{U}} \in \xn^{l_{U}}$, and, $\im{q} =  \bigcup_{1 \le i \le  s_{T}} U^{+}_{\mu_i}$ and   $\im{p} = \bigcup_{1 \le j \le s_{U}} U^{+}_{\nu_j}$.
		
   Let $(q,p)$ be a state of $T \ast U$ such that $(q,p) \in \core(T \ast U)$. Let $\mu_1, \ldots, \mu_{s_{T}} \in \xn^{l_T}$, satisfy $\im{q} =  \bigcup_{1 \le i \le  s_{T}} U^{+}_{\mu_i}$. For each $1 \le i \le s_{T}$, let $p_i = \pi_{T}(\mu_i, p)$, and let $\nu_{1,i}, \nu_{2,i}, \ldots, \nu_{s_{U},i} \in \xn^{l_{u}}$ satisfy, $\im{p_i} = \bigcup_{1 \le j \le s_{U}} U^{+}_{\nu_{j,i}}$. Then it follows that $\im{(q,p)} = \bigcup_{1 \le i \le s_{T}, 1 \le j \le S_{U}}U^{+}_{\lambda_{U}(\mu_i,p) \nu_{j,i}}$. Thus  we see that for any state $(q',p')$  of $\core(T\ \ast U)$, there are $\tau_1, \tau_2, \ldots, \tau_{s_{T}s_{U}} \in \xn^{l_{T}+l_{U}}$ such that $\im{(q', p')} = \bigcup_{1 \le i \le s_{T}s_{U}}U^{+}_{\tau_i}$.
   	
   	Finally observe that if $s$ is any divisor of a power of $N$ and $l \in \N$, are  such that any state $(p',q') \in Q_{P}$, there are elements $\tau_1, \tau_2, \ldots, \tau_s \in \xn^{l}$ and $\bigcup_{1 \le i \le s} U^{+}_{\mu_i} = \im{q}$. Then either $s = n^{l_{U}+ l_{T} - l}s_{T}s_{U}$ or $sn^{l- l_{U}-l_{T}} = s_{T}s_{U}$. In particular, $s$ and $s_Ts_{U}$ correspond to the same element of $\mn{n}$. By Remark~\ref{rem:consequencesofsynchronousrep} the result follows, since $TU$, the product of $T$ and $U$ in $\Ln{n}$, is precisely the minimal representative of $\core(T \ast U)$.
   	
   	Surjectivity of $\rsig_{\omega}$ is a consequence of examples in constructed in the paper \cite{OlukoyaAutTnr}. In that paper, for each prime $p$ dividing $n$, an example is constructed of an element $T \in \Ln{n}$ for which $(T)\sig_{\omega} =  p$. Alternatively, the transducers in Example~\ref{Example:imagegenerators} are also valid. 
   	
   	Let $n = p_{1}^{a_1}p_2^{a_2}\ldots p_{r}^{a_r}$, and write $N_i = n/p_{i}$ for each $1 \le i \ne r$. As noted in Example~\ref{Example:imagegenerators}, the elements  $T(p_i, N_i)$ commute and $(T(p_i,N_i))\rsig_{\omega} = N_i$. In particular, the map $\rsig_{\omega}$ is an isomorphism from $\gen{ T(p_i, N_i) \mid 1 \le  1 \le r}$  to the group  $\mn{n}$, since  the product  $\Pi_{1 \le i \le r} T(p_i, N_i)^{l_i} = \id$ in $\Ln{n}$ and a word $w \in \{ T(p_i, N_i) \mid 1 \le i \le r\}$ represents the trivial element of $\Ln{n}$ if and only if it is of the form $\Pi_{1 \le i \le r} T(p_i, N_i)^{al_i}$ for some $a \in \N$. Thus, the extension $$ 1 \to \Kn{n} \to \Ln{n} \to  \mn{n}$$ splits.
   	
   	A straightforward application of Proposition~\ref{Proposition:sigdeterminesmembership} shows that $\Kn{n} \le \Ln{n,1}$. This is because $T \in \Ln{n}$ is an element of $\Kn{n}$ if and only if for any state $q \in Q_{T}$, the image of $q$ can be written as a union of a power of $n$ basic open sets.
  \end{proof}

\begin{notation}
	Write $\Kn{n}$ for the group $\ker(\rsig_{\omega})$.  Observe that an element $T \in \Ln{n,1}$ is an element of $\Kn{n}$ if and only if for some state $q \in Q_{T}$ (and so for any state of $T$), the image of $q$ can be written as a union of a power of $n$ basic open sets.
\end{notation}

\subsection{On \texorpdfstring{$\aut{\xnz, \shift{n}}$}{Lg}}

The paper \cite{BelkBleakCameronOlukoya} gives an epimorphism homomorphism from the group $\aut{\xnz, \shift{n}}$ to the subgroup $\Un{n-1,n}$ of the group $\Un{n-1}$ of units of $\Z_{n-1}$ generated by the divisors of $n$. We make use of the following  proposition to extend this homomorphism to the direct product  $\Pi_{k \in \N_{1}} \Un{n^{k}-1, n^{k}}$. We show that the kernel of this new homomorphism is isomorphic to the group $\Kn{n}$. In a later section we show that this homomorphism is precisely the \emph{dimension representation}, thus making the group $\Kn{n}$ isomorphic to the so called \emph{inert subgroup}.

We begin by setting up some notation.

Fix $n,m \in \N$ and let $d \in \N_{1}$ be such that $n^{d} = m$. There is an embedding from the monoid $\End(\xnz, \shift{n})$ into the monoid $\End(X_{m}^{\Z}, \shift{m})$ which is given by `increasing the alphabet size' as follows. Order the set $\xn^{d}$ in the lexicographic ordering and let $\beta_{n,m}: X_{m} \to \xn^{d}$  be the bijection that sends a number $i$ to the $i$\textsuperscript{th} element of $\xn^{d}$. Given $l \in \N$, we extend $\beta_{n,m}$ to a bijection $(\beta_{n,m})_{l}: X_{m}^{l}\to \xn^{dl}$ by $\gamma_1\gamma_2 \ldots \gamma_l \mapsto (\gamma_1)\beta_{n,m}(\gamma_2)\beta_{n,m}\ldots(\gamma_l)\beta_{n,m}$. Given a  sequence $x \in X_{m}^{\Z}$, write $(x)(\beta_{n,m})_{\infty}$ for the sequence $y \in \xn^{\Z}$ such that $(x_{i})\beta_{n,m} = y_{id}y_{id+1}\ldots y_{(i+1)d-1}$. Note that $(\beta_{n,m})_{\infty}: X_{m}^{\Z} \to \xnz $ is a homeomorphism. Moreover $\shift{m}(\beta_{n,m})_{\infty} = (\beta_{n,m})_{\infty} \shift{n}^{d}$. Thus, given $\phi \in \End(\xnz, \shift{n})$, the map $(\beta_{n,m})_{\infty}\phi (\beta_{n,m})_{\infty}^{-1}$ is an element of $\End((X_{m})^{\Z},\shift{m})$. Write $\shift{n,m}$ for the map $(\beta_{n,m})_{\infty}\shift{n} (\beta_{n,m})_{\infty}^{-1}$ and note that $\shift{n,m}^{d} = \shift{m}$. Let $\ext{n}{m}: \aut{\xnz, \shift{n}} \to \aut{\xnk{m}{}}$ denote the map which sends an element $\psi$ to $\beta_{\infty}\psi \beta_{\infty}^{-1}$. Henceforth we no longer distinguish between the maps $\beta_{n,m}$, $(\beta_{n,m})_{\infty}$ and $(\beta_{n,m})_{l}$, $l \in \N$, as it will normally be clear from the arguments which is meant. The following result is standard and so we do not give a proof.

\begin{proposition}\label{prop:extmorph}
	Let $n, m \in \N_{2}$ such that $m$ is a power of $n$. Let $C_{\End(X_{m}^{\Z}, \shift{m})}(\shift{n,m}):= \{ \psi \in \End(X_{m}^{\Z}, \shift{m}) \mid \psi \shift{{n,m}} = \shift{n,m}\psi \}$ and $C_{\aut{X_{m}^{\Z}, \shift{m}}}(\shift{n,m}):= \{ \psi \in \aut{X_{m}^{\Z}, \shift{m}} \mid \psi \shift{n,m} = \shift{n,m}\psi \}$. Then $C_{\End(X_{m}^{\Z}, \shift{m})}(\shift{n,m})$ is precisely the monoid $(\End(\xnz, \shift{n}))\ext{n}{m}$ and $C_{\aut{X_{m}^{\Z}, \shift{m}}}(\shift{n,m})$ is precisely the group $(\aut{\xnz, \shift{n}})\ext{n}{m}$. 
\end{proposition}

 Clearly the map the composition of the natural map from $\aut{\xnz, \shift{n}}$ to $\aut{ \xnz, \shift{n}}/ \gen{\shift{n}}$, yields a homomorphism from $\aut{\xnz, \shift{n}}$ to $\Un{n-1,n}$. We will also denote this map by $\rsig$. Below, we explicitly define this map $\rsig: \aut{\xnz, \shift{n}} \to \Un{n-1,n}$: we give two equivalent definitions: the first using the representation of elements of $\aut{\xnz, \shift{n}}$ as an element of $\gen{\shift{n}}F_{\infty}$ and the second using the representation in terms of elements of $\ALn{n}$, that is in terms of bi-synchronizing core transducers. Both points of view will be useful to us.

\begin{Definition}
	Let $m \in \N$ and  $f\in F(\xn,m)$ so that $f_{\infty} \in F_{\infty}(\xn)$. Let $\gamma = \gamma_{-m+1}\ldots \gamma_{-1} \gamma_{0} \in \xnl{m}$ and consider the clopen set $U^{0}_{\gamma}:= \{ x \in \xnz \mid x_{-m+1}\ldots x_{-1}x_0 = \gamma\}$. The image $(U^{0}_{\gamma})f_{\infty}$ is clopen and so the set $((U^{0}_{\gamma})f_{\infty})_{> 0}:= \{  x_{1}x_{2}\ldots \mid  \exists y \in (U^{0}_{\gamma})f_{\infty} :  y_{1}y_{2}\ldots = x_{1}x_{2}\ldots \}$ is clopen also. Let $r \in \N$, be minimal such that there exists elements $\nu_1, \nu_2 \ldots \nu_r \in \xns$ satisfying $\cup_{1 \le i \le r} U^{+}_{\nu_i} = {((U^{0}_{\gamma})f_{\infty})_{> 0}}$. Set $(f)\rsig = r \pmod{n-1}$. More generally for $k \in \Z$,  set $(\shift{n}^{k}f_{\infty}) = (f_{\infty})\rsig$.
	
	Alternatively, given $\psi = \shift{n}^{k}f_{\infty}$, $k \in \N$ and $f \in F(\xn, m)$ for some $m \in \N$, an element of $\aut{ \xnz, \shift{n}}$, $\psi = (L, \alpha)$ for some annotation $L \in \Ln{n}$ and an annotation $\alpha$ of $L$.  Let $q \in Q_{L}$ be an arbitrary state, and let $r \in \N$ be minimal such that there are $\nu_1, \nu_2 \ldots \nu_r \in \xns$ satisfying $\cup_{1 \le i \le r} U^{+}_{\nu_i} = \im{q}$, then set $(L,\alpha)\rsig = r \pmod{n-1}$.
\end{Definition}

 Below we make use of Proposition~\ref{prop:extmorph} to define an extension of the homomorphism $\rsig$ and explore some implications of this. For reasons which become clear later on we also write $\sig_{\omega}$ for this homomorphism. We require the following properties of the  elements $T(p_i, N_i)$ of Example~\ref{Example:imagegenerators}.

\begin{lemma}\label{lem:canomit0}
	Let $r \in \N_{2}$, $p_1 \ldots, p_r$ be distinct primes, $l_1, \ldots, l_r \in \N_{1}$ and  $ n= p_{1}^{l_1}p_2^{l_2}\ldots p_{r}^{l_r}$. Fix $1 \le i \le r$, a word $w= w_1w_2 \ldots w_{l} \in (\{1, \ldots, r\}\backslash \{i\})^\ast$, then the product $U:= T(p_{w_1}, N_{w_1}) \ast T(p_{w_2}, N_{w_2}) \ast \ldots \ast T(p_{w_l}, N_{w_l})$ has no  states of incomplete response. Consequently, setting $S$ to be the minimal representative of $U$, we have $(T(p_{w_1}, N_{w_1}),0) (T(p_{w_2}, N_{w_2}),0)  \ldots (T(p_{w_l}, N_{w_l}),0) = (S,0)$.
\end{lemma}
\begin{proof}
	We use the description of the elements $T(p_i, N_i)$ as given in Example~\ref{Example:imagegenerators}. That is, let $X_{p_a}$, $1 \le a \le r$ be alphabets of size $p_a$, and we think of an element of $\xn$ as a word $u_ru_{r-1}\ldots u_{1}$ for $u_a \in  X_{p_a}^{l_a}$, $1 \le  a \le r$. Let $ a \in \{1,\ldots, r\}\backslash\{i\}$ and $x \in  X_{p_a}$ so that $q_x$ is a state of $T(p_a, N_a)$. Then for any element  $\gamma = \gamma_r \gamma_{r-1} \ldots \gamma_1$, $\gamma_b \in X_{p_b}^{l_b}$, $1 \le b \le r$, $\lambda_{T(p_a,N_a)}(\gamma, q_x)$, is the word obtained from $\gamma$, by deleting the last letter of $\gamma_{a}$ and appending $x$ as a prefix to the resulting word. We thus observe that $T(p_a, N_a)$ leaves the subword $\gamma_i$ unchanged.
	
	Let $w_1,w_2,\ldots w_{l} \in \{1,\ldots, r\}\backslash\{i\}$, set $U:= T_{w_1} \ast T_{w_2} \ast \ldots \ast T_{w_l}$ and let $q$ be an arbitrary state of $U$. Then, given words $\gamma = \gamma_{r}\ldots\gamma_1, \delta= \delta_{r}\ldots \delta_{1} \in \xn$, $\delta_a, \gamma_a \in X_{p_a}^{l_a}$, such that the $\gamma_i \ne \delta_i$, the previous paragraph implies that $\lambda_{U}(\delta, q) \ne \lambda_{U}(\gamma, q)$. Thus we see that $U$ has no states of incomplete response. Setting $S' = \core(U)$ and $S$ to be the minimal representative of $S'$. We have that both $S$ and $S'$ are synchronous since $S'$ has no states of incomplete response. Moreover, as $S$ is precisely $S'$ with $\omega$-equivalent states identified, the equality $(S,0) = (S', 0)$ is valid. To conclude we observe that $(S',0)$ is precisely the map $(T(p_{w_1}, N_{w_1}),0) (T(p_{w_2}, N_{w_2}),0)  \ldots (T(p_{w_l}, N_{w_l}),0)$.
\end{proof}

\begin{corollary}\label{cor:omit0}
	Let $r \in \N_{2}$, $p_1 \ldots, p_r$ be distinct primes, $l_1, \ldots, l_r \in \N_{1}$ and  $ n= p_{1}^{l_1}p_2^{l_2}\ldots p_{r}^{l_r}$. Fix $1 \le i \le r$, then the subgroup $\gen{ (T(p_j, N_{j}),0) \mid j \in \{1,2,\ldots, r\} \backslash \{i\} }$ of $\aut{\xnz, \shift{n}}$ is isomorphic to the subgroup $\gen{ T(p_j, N_{j}) \mid j \in \{1,2,\ldots, r\} \backslash\{i\}}$ of $\Ln{n}$.
\end{corollary}

In the statement of the theorem below, we write $\Kn{n}$ for the kernel of the map $\rsig_{\omega}:  \Ln{n} \to \mn{n}$ and we write $\ker(\rsig_{\omega})$ for the kernel of the map $\rsig_{\omega}: \aut{\xnz, \shift{n}} \to \Pi_{k \in \N} \Un{n^{k}-1,n^k}$. As we will see, these two groups are isomorphic.
\begin{Theorem}\label{thm:exthom}
	The following things hold:
	\begin{enumerate}
		\item \label{exthom1}There is a non-trivial homomorphism  $\rsig_{\omega}: \aut{\xnz, \shift{n}} \to \Pi_{k \in \N} \Un{n^{k}-1, n^{k}}$, given by $\psi \mapsto (((\psi)\ext{n}{n^{k}})\rsig)_{k \in \N}$.
		\item \label{exthom2} Thus, for each $k \in \N$, there is a homomorphism $\rsig_{[k]}: \aut{\xnz, \shift{n}} \to \Pi_{1 \le j \le k } \Un{n^{j}-1, n^{j}}$ and an epimorphism $\rsig_{k}: \aut{\xnz, \shift{n}} \to  \Un{n^{k}-1, n^{k}}$. 
		\item\label{exthom3} The equality:
		\begin{IEEEeqnarray*}{rCl}
		(\aut{\xnz,\shift{n}})\rsig_{\omega} &=&  (\gen{(T(d,e),0) \mid d,e \in \N, de = n })\rsig \nonumber \\ &=& \gen{(d\pmod{n-1},d,d \ldots) \mid d \ \mathrm{ is \ a \ prime \ divisor \ of  } \ n }
		\end{IEEEeqnarray*}
		 is valid.  In  particular, for any $\psi \in \aut{\xnz, \shift{n}}$, there exist prime numbers $e_1, e_2, \ldots, e_l$,  depending only on $\psi$, such that, taking $d_i \in \N$ satisfying  $d_ie_i = n$, we have $$((T(d_1,e_1),0)^{-1}(T(d_2,e_2),0)^{-1} \ldots (T(d_l,e_l),0)^{-1}\psi)\rsig_{\omega} =  (1,1,1,\ldots).$$
		
	\item\label{exthom4} The image $(\aut{\xnz, \shift{n}})\rsig_{\omega}$ of $\aut{\xnz, \shift{n}}$ under $\rsig_{\omega}$ is a finitely generated torsion free abelian group of rank equal to the number of prime divisors of $n$.
	
	\item \label{exthom5} Let $p_1, p_2, \ldots, p_r$ be distinct prime numbers and let $l_1, l_2, \ldots, l_r$ be elements of $\N_{1}$ such that $n = p_1^{l_1}p_2^{l_2}\ldots p_{r}^{l_r}$. Then,
	  \begin{enumerate}
	  	\item \label{exthom5a}  $\aut{\xnz, \shift{n}} \cong \ker(\rsig_{\omega}) \rtimes \Z^{r} \cong \Kn{n} \rtimes \Z^{r}$;
	  	\item \label{exthom5b} and $\aut{\xnz,\shift{n}} \cong \Ln{n} \times \Z$ if and only if $\gcd(l_1, l_2, \ldots, l_r) = 1$ (i.e $n$ is not a power of a strictly smaller integer.)
	  \end{enumerate}

	\end{enumerate} 
\end{Theorem}
\begin{proof}
 Parts \ref{exthom1} and \ref{exthom2} arise by combining Theorem~\ref{thm:sigOn} and Proposition~\ref{prop:extmorph}.
 
 We now consider \ref{exthom3}. We begin with the last equality in the first sentence. 
 
 Let $k \in \N$, and let $d,e \in \N$ be such that $de = n$. Define a transducer $T(d,e)_{k} = \gen{\xnl{k}, Q_{T(d,e)}, \pi_{T(d,e)_{k}}, \lambda_{T(d,e)_{k}}}$ as follows: for $\gamma \in \xnl{k}$ and $q \in Q_{T(d,e)}$, $\pi_{T(d,e)_{k}}(\gamma,q) = \pi_{T(d,e)}(\gamma, q)$ and $\lambda_{T(d,e)_{k}}(\gamma,q) = \lambda_{T(d,e)}(\gamma, q)$. Then $(T(d,e)_{k},0)$ defines an element of $\aut{\xnk{n^{k}}{\Z}, \shift{n^{k}}}$ moreover, upto relabelling the alphabet, we have $(T(d,e)_{k},0)=((T(d,e),0))\ext{n}{n^{k}}$. 
 
 It is straightforward to see that $((T(d,e)_{k},0))\rsig = e n^{k-1}$. This is because, as computed in Example~\ref{Example:imagegenerators}, for a state $q \in Q_{T(d,e))}$, there is a subset $V_{q}$ of $\xn$ length $e$ such that $\im{q} = \cup_{\nu \in V_q} U^{+}_{\nu}$. It therefore follows that $\im{q}$, thought of now as a state of $T(d,e)_{k}$, is precisely:  $$\bigcup_{ \nu \in V_{q}, \mu \in \xnl{k-1}}  U^{+}_{\nu \mu}. $$ We therefore  see that $$((T(d,e),0))\rsig_{\omega} = (e \pmod{n-1}, en\pmod{n^2-1}, \ldots, en^{k-1}\pmod{n^{k}-1},\ldots),$$ and so  $$((T(d,e),0)^{-1})\rsig_{\omega} = (d \pmod{n-1}, d,d,d,\ldots).$$
  
  We now consider the remainder of  part \ref{exthom3}. Note that since $\shift{n} = (T(n,1),0)$ it suffices to show that for any element $f_{\infty} \in F_{\infty}(\xn)$, there is an element $\psi \in \gen{ (T(d,e),0) \mid d,e \in \N, de = n }$ such that $(\psi f_{\infty}) \rsig_{\omega} = (1,1,1,\ldots)$.
 
 Let $f_{\infty} \in F_{\infty}(\xn)$ and let $m \in \N$ be such that $f \in F(\xn,m)$.  Let $k \in \N$ and define a map $g$ from $\xnl{mk} \to \xnl{k}$ by setting $(\gamma_{m+k-1} \ldots \gamma_{1}\gamma_0)g = (\gamma_0\gamma_1\ldots\gamma_{m+k-1})f_{k}$. Thus one may define a map $g_{\infty}: {(\xnl{k})^{\Z}} \to (\xnl{k})^{\Z}$. Observe that, up to relabelling the alphabet,  $g_{\infty} = (f_{\infty})\ext{n}{n^{k}}$. We thus move freely between the maps $g_{\infty}$ and $(f_{\infty})\ext{n}{n^{k}}$.
  
 Fix a word $\Gamma:= \gamma_{-mk+1} \ldots \gamma_{-1} \gamma_0 \in (\xn^{km})$ and consider the clopen set $U^{0}_{\Gamma} \subset \xnk{n^{k}}{\Z}$. We dually consider $U^{0}_{\Gamma}$ both as a subset of $\xnk{n^{k}}{\Z}$ and $\xnz$. 
 
 Note that $((U^{0}_{\Gamma})g_{\infty})_{> 0} = \{ \rho_{1}\rho_{2}\ldots \mid \rho_i \in  \xn^{k}, i \in \N_{1} \mathrm{ \ and \ } \rho_{1}\rho_{2}\ldots \in ((U^{0}_{\Gamma})f_{\infty})_{>0} \}$. That is, $((U^{0}_{\Gamma})g_{\infty})_{> 0}$ is precisely the set $((U^{0}_{\Gamma})f_{\infty})_{>0}$ partitioned into blocks of length $k$ (from left to right). Thus, $((f_{\infty})\ext{n}{n^{k}})\rsig$ is   congruent modulo $n^{k}-1$ to the minimal number $r$ such that there exist $\mu_1, \mu_2, \ldots,\mu_r \in (\xnl{k})^{\ast}$ satisfying ${\left(\bigcup_{1 \le i \le r} U^{+}_{\mu_i}\right)} = ((U^{0}_{\Gamma})f_{\infty})_{>0}$.
 
 By a result in \cite{BelkBleakCameronOlukoya} there are minimal numbers $l,s \in \N$ such that $s$ is a divisor of some power of $n$ and there are $\mu_1, \mu_2, \ldots, \mu_s \in \xnl{l}$ such that $((U^{0}_{\Gamma})f_{\infty})_{>0} = {\left(\bigcup_{1 \le i \le s} U^{+}_{\mu_i}\right)}$. Let $l^{-1} := -l \pmod{k}$, and let $d_1, d_2,\ldots, d_{l}, e_1, e_2, \ldots, e_l \in \N$ satisfy  $d_i e_i = n$ for all $1 \le i \le l$ and  $s =d_1d_2\ldots d_{l}$. Observe that $${\left(\bigcup_{1 \le i \le s, \ \nu \in \xn^{l^{-1}}} U^{+}_{\mu_i \nu}\right)} = ((U^{0}_{\Gamma})f_{\infty})_{> 0},$$ 
 therefore $$(f_{\infty})\rsig_{k} = sn^{l^{-1}} \pmod{n^{k}-1}.$$  Consider the map $\psi = (T(d_1,e_1),0)^{-1}(T(d_2,e_2),0)^{-1} \ldots (T(d_l,e_l),0)^{-1}f_{\infty}$. It follows that $(\psi)\rsig_{k} = e_1\ldots e_{l} \cdot  s  \cdot  n^{l^{-1}} \pmod{n^{k}-1}$. However, since $e_1\ldots e_{l} \cdot s =  n^{l}$ and  $l + l^{-1} \equiv 0 \pmod{k}$ it follows that $(\psi)\rsig_{k} = 1$. 	Moreover, as the product $(T(d_1,e_1),0)^{-1}(T(d_2,e_2),0)^{-1} \ldots (T(d_l,e_l),0)^{-1}$ depends only on $f_{\infty}$, then for any $k' \in \N$, $(\psi)\rsig_{k'} = 1$.
 
 We may now deduce that $\rsig_{k}$ is surjective for any $k \in \N$. For, as  $\Un{n^{k}-1,n^{k}}$ is precisely the subgroup of $\Un{n^{k}-1}$ generated by the divisors of $n^{k}$, it is therefore also generated by the divisors of $n$ and by  Part~\ref{exthom3}, $\rsig_{k}$ is unto the divisors of $n$.
 
 Part~\ref{exthom4}  follows since the generators are torsion free and commute. 
 
Write $n = p_1^{l_1} \ldots p_{r}^{l_r}$, where $l_{i} >0$ and $p_{i}$ is prime for all $1 \le i \le r$; set $N_{i} = n/p_{i}$.  By Part~\ref{exthom3},  $\aut{\xnz, \shift{n}} = \ker(\rsig_{\omega}) \gen{ (T(p_{i},N_{i}),0)  | 1 \le i \le r  }$ and so we have $\aut{\xnz, \shift{n}} = \ker(\sig_{\omega}) \rtimes \gen{ (T(p_{i},N_{i}),0)  | 1 \le i \le r  }$.  It was demonstrated in  Example~\ref{Example:imagegenerators} that $\gen{ (T(p_{i},N_{i}),0) | 1 \le i \le r }$ isomorphic to $\Z^{r}$. 

We now show that $\ker(\rsig_{\omega}) \cong  \Kn{n}$. 

Let  $\shift{n}^{-m} f_{\infty} \in \ker(\rsig_{\omega})$ for $f_{\infty} \in F_{\infty}(\xn)$. There is an element $T \in \Ln{n}$ and an annotation $\alpha$ such that $f_{\infty} = (T, \alpha)$. Moreover, by Proposition~\ref{Prop:lnpn} we may assume $\alpha: Q_{T} \to \N$. Let $\kappa$ be the canonical annotation of $T$, then there is a $j \in \Z$ such that $(T,\kappa) \shift{n}^{j} = f_{\infty}$. Thus, by results in \cite{BelkBleakCameronOlukoya} the map $\kappa +j : Q_{T}  \to \Z$ by $p \mapsto (p)\kappa + j$ is an annotation of $T$ and is equal to $\alpha$. Since there is a state of $T$ which takes value zero under $\kappa$,  as $\kappa$ is the canonical annotation, it follows  that  $j \ge 0$.

 Let $q \in Q_{T}$ be a state of $T$ such that  $(q)\kappa = 0$. Let $k \in  \N$ be such that $T$ is  synchronizing at level  $k$ and let $l \in \N$ be such that there are elements $\mu_1, \ldots, \mu_{s} \in \xn^{l}$ satisfying $\bigcup_{1 \le i \le s} U^{+}_{\mu_i} = \im{q}$. It simplifies our argument to assume that $l >m$ and that $k$ is large enough such that for any word $\gamma \in \xn^{kl}$, and any state $p \in Q_{T}$, $|\lambda_{T}(\gamma, p)| \ge j$. 
 
 Let $t \in \N$ be such that $ts = n^{l}$, since $s$ must be a divisor of a power of $n$, $t$ exists.  Fix, words $\Gamma, \Delta \in \xn^{kl}$ such that the state of $T$ forced by ${\Gamma}$ is $q$, and consider the clopen set $(U^{0}_{\Delta\Gamma })(T,\kappa))_{>0}$.  Since $(q)\kappa = 0$, it follows that $(U^{0}_{\Delta\Gamma})(T,\kappa))_{>0} = {\im{q}}$.
 
  Let $\rho$ be the length $j$ suffix of the word ${ \lambda_{T}({\Gamma}, q_{{\Delta}}) }$.  Since by definition $(x)\shift{n}^{j}$    is the sequence $y$ satisfying $y_{i} = x_{i-j}$ for all $i \in \Z$, it follows that $(U^{0}_{\Gamma})(T,\kappa+j))_{>0} =  {\bigcup_{1 \le i \le s} U^{+}_{\rho\mu_i} }$. Thus $(f_{\infty})\rsig_{l+j} = s$. 
  
  Let $a \in \N$ be such that $a \ge \max\{l+j, m\}$. Let $b \in \N$ be such that $b \equiv -(l+j) \pmod{a}$, then, $(f_{\infty})\rsig_{a} = sn^{b}$. Therefore we see that $(\shift{n}^{-m}f_{\infty})\rsig_{a} = n^{m}sn^{b} \pmod{n^{a}-1}$.  This now implies that $((\shift{n}^{-m}f_{\infty})\sig_{a})^{-1} = n^{a-m} t \mod{n^{a}-1}$.  
  
  If $n^{a-m}t \ge n^a$, then $t \ge n^{m}$ and $n^{a-m}t \equiv 1 \pmod{n^{a}-1}$ (since $\shift{n}^{-m}f_{\infty} \in \ker(\rsig_{\omega})$). Since this equation $n^{a-m}t \equiv 1 \pmod{n^{a}-1}$ is valid for all $a \in  \N_{1}$, it must be the case that $t=n^m$. 
  
  If $n^{a-m}t \le n^a$, then, as $n^{a-m}t \equiv 1 \pmod{n^{a}-1}$ for all $a \in \N$,   we must have that $n^{a-m}t = n^{a}$ and  $t = n^{m}$. Thus, in either case, it follows that $s = n^{l-m}$, which implies that  $T \in \Kn{n}$. 

On the other hand, given $T \in  \Kn{n}$ we show that there is a unique annotation $\alpha$ of $T$ such that $(T,\alpha) \in \ker(\rsig_{\omega})$.

Fix $T \in \Kn{n}$. Let $\kappa$ be the canonical annotation of $T$ and  $q \in Q_{T}$ be a state such that  $(q)\kappa = 0$. Since $T \in \Kn{n}$ there is are  numbers $a,b \in \N_{1}$ and words $\mu_1, \mu_2, \ldots, \mu_{n^{a}} \in \xn^{b}$ such that $\bigcup_{1 \le i \le n^{l}} U^{+}_{\mu_i} = \im{q}$.  Thus, it follows that for arbitrary $d \in \N_{1}$, and $e \in \Z_{d}$ such that $b+e \equiv 0 \mod{d}$, $((T,\alpha))\rsig_{d} = n^{a+e} \mod {n^{d} - 1}$. It is now clear that $(\shift{n}^{a-b}(T,\alpha)) \rsig_{\omega} = (1,1,1,\ldots)$, since $(\shift{n}^{a-b})\rsig_{\omega} = (n^{b-a}, n^{b-a}, \ldots)$. 

Observe that $\shift{n}^{a-b}(T,\alpha)) = (T, \alpha +a-b)$. Moreover,  for any other annotation $\beta$ of $T$, there is an integer $i \in \Z$ such that $(T, \beta) = (T, \alpha) \shift{n}^{i}$. Since $\shift{n}$ is not an element of $\ker(\rsig_{\omega})$, it follows that $\alpha$ is the unique annotation of $T$ such that $(T,\alpha+ a-b) \in \ker(\rsig_{\omega})$.

Thus we see that the map from $\ker(\rsig_{\omega}) \to \Ln{n}$ by $(T,\alpha) \mapsto T$ is an isomorphism. This proves \ref{exthom5a}. 

We now demonstrate \ref{exthom5b}. 

Fix $1 \le i \le r$ and suppose that $n$ is not a power of a strictly smaller integer, that is $\gcd(l_1, l_2, \ldots, l_r) = 1$. The proof of Proposition~\ref{prop:lnsplit} and Lemma~\ref{lemma:characterisationofmn} shows that,  $\Ln{n} \cong \Kn{n} \rtimes \gen{T(p_j, N_j) \mid 1 \le j \le r, \ j \ne i}$. Therefore, by Corollary~\ref{cor:omit0} the map from the subgroup $\ker(\rsig_{\omega}) \gen{(T(p_i, N_i),0) \mid 1 \le j \le r, \ j \ne i}$ of $\aut{\xnz, \shift{n}}$ to $\Ln{n}$ sending  $(T,\alpha) U \to  TU$ for $U \in \gen{T(p_i, N_i) \mid 1 \le i \le r-1}$  is an isomorphism. Now as the group $\aut{\xnz, \shift{n}}/ \gen{\shift{n}} \cong \Ln{n}$, we see that the extension $1 \to \gen{\shift{n}} \to \aut{\xnz, \shift{n}} \to \Ln{n}$ splits and so $\aut{\xnz, \shift{n}} \cong \ln \times \Z$.

If $n$ is a power of a proper integer, then $\shift{n} \in \aut{\xnz, \shift{n}}$ has proper roots. For example setting $d$ to be the $\gcd(l_1, l_2, \ldots, l_r)$, $$\prod_{1 \le i \le r} (T(p_i, N_i),0)^{\frac{l_i}{d}}$$  is a root of $\shift{n}$. This now means that $\aut{\xns, \shift{n}}$ cannot be written as $\gen{\shift{n}}$ and some subgroup $L \le \aut{\xnz, \shift{n}}$, since $L$ must necessarily have non-trivial intersection with $\shift{n}$.
\end{proof}

\begin{corollary}\label{cor:exthm}
	Let $n \in \N_{2}$, $p_1, p_2, \ldots, p_r$ be distinct primes, and $l_1, \ldots, l_r \in \N_{1}$, such that $n = p_1^{l_1} p_2^{l_2} \ldots p_{r}^{l_r}$. Set $G:= (\gen{(\shift{n})\rsig_{\omega}})\rsig_{\omega}^{-1}$  a subgroup of $\aut{\xnz, \shift{n}}$ and  $l =  \gcd(l_1, l_2, \ldots, l_r)$. Then, 
	\begin{enumerate}[label=\arabic*.]
		\item $G$ is a normal subgroup of $\aut{\xnz, \shift{n}}$;
		
		\item $G \cong \Kn{n} \times \Z$ and there is a short-exact sequence: $$1  \to G \to \aut{\xnz, \shift{n}} \to \Z^{r-1} \times \Z/l \Z;$$      
		\item if $n$ is prime, $G = \aut{\xnz, \shift{n}} \cong \Ln{n} \times \Z$. \label{cor:exthm5}
	\end{enumerate}
	  
\end{corollary}
\begin{proof}
	For $1 \le i \le r$ let $N_{i} = n/p_{i}$.
	
	By Theorem~\ref{thm:exthom}, $(\aut{\xnz,\shift{n} })\rsig_{\omega} = \gen{ (T(N_i,p_i),0)) \rsig_{\omega} \mid 1 \le i \le r }$. The quotient  $(\aut{\xnz,  \shift{n}})\rsig_{\omega} / (\gen{\shift{n}})\rsig_{\omega}$,  is the quotient of $\Z^{r}$ by the subgroup $\{ (ml_1, ml_2, \ldots, ml_{r})\mid m \in \Z \}$. It is not hard to see that the torsion subgroup of $\Z^{r}/ \gen{(l_1, l_2, \ldots, l_r)}$ is precisely the subgroup $\gen{(\sfrac{l_1}{l}, \sfrac{l_2}{l}, \ldots, \sfrac{l_r}{l})}$. By the third isomorphism theorem, we have: $$\dfrac{\left(\dfrac{\Z^{r}}{\gen{(l_1, l_2, \ldots, l_r)}}\right)}{\left(\dfrac{ \gen{(\sfrac{l_1}{l}, \sfrac{l_2}{l}, \ldots, \sfrac{l_r}{l})}}{ \gen{(l_1, l_2, \ldots, l_r)}}\right)} \cong \dfrac{\Z^{r}}{\gen{(\sfrac{l_1}{l}, \sfrac{l_2}{l}, \ldots, \sfrac{l_r}{l})}}.$$ 
	As the group $\gen{(1,1,\ldots,1, 0)}$ is a finite index subgroup of $\Z^{r}/\gen{(\sfrac{l_1}{l}, \sfrac{l_2}{l}, \ldots, \sfrac{l_r}{l})}$, we deduce that $\Z^{r}/\gen{(\sfrac{l_1}{l}, \sfrac{l_2}{l}, \ldots, \sfrac{l_r}{l})} \cong \Z^{r-1}$.
	
	Let $\psi: \aut{\xnz,\shift{n} } \to (\aut{\xnz,  \shift{n}})\rsig_{\omega} / (\gen{\shift{n}})\rsig_{\omega} $, be  the composition of $\rsig_{\omega}$ with the natural map from $(\aut{\xnz,\shift{n} })\rsig_{\omega}$ to the quotient $(\aut{\xnz,  \shift{n}})\rsig / (\gen{\shift{n}})\rsig_{\omega}$. It is clear that $\ker(\psi)$ is precisely the group $G$. Moreover, it is also clear that,  $\gen{ \{((T(N_i,p_i),0) \psi \mid 1 \le i < r\} \sqcup \{ (T(N_r, p_r),0)^{j}\psi \mid 1 \le j < l_{r} \} } = (\aut{\xnz,\shift{n} })\psi$. This is because, $(\Pi_{1 \le i \le r}(T(N_i,p_i),0)^{l_i})\psi \in \gen{(\shift{n})\psi}$ which is trivial. Thus we see that, there is a short exact sequence $$1  \to G \to \aut{\xnz, \shift{n}} \to (\aut{\xnz,  \shift{n}})\rsig_{\omega} / (\gen{\shift{n}})\rsig_{\omega}.$$
	
	In particular, any element of $\aut{\xnz, \shift{n}}$, can be written as a product $gt\shift{n}^{j}$, for $g \in G$, $t$ a product of an element of $\gen{ \{((T(N_i,p_i),0) \psi \mid 1 \le i < r\} }$ with an element of $ \{ (T(N_r, p_r),0)^{j} \mid 1 \le j < l_{r} \}$, and $j \in \Z$.

	The homomorphism $\rsig_{\omega}\restriction_{G}$ has image precisely $\gen{(\shift{n})\rsig_{\omega}}$, thus the short exact sequence  $$1  \to \ker(\rsig_{\omega}\restriction_{G}) \to G \to (G)\rsig_{\omega}$$ splits. Since $\shift{n}$ commutes with every element of $G$ we see that $G = \ker(\rsig_{\omega}\restriction_{G}) \times \gen{\shift{n}}$. Clearly, by definition of the group $G$, $\ker(\rsig_{\omega}\restriction_{G}) = \ker(\rsig_{\omega})$. Thus by Theorem~\ref{thm:exthom} $\ker(\rsig_{\omega}\restriction_{G}) \cong \Kn{n}$.
	
	\begin{comment}
	To see  Part \ref{cor:exthm4} observe that by the proof of Theorem~\ref{thm:exthom}, given an element $f_{\infty} \in F_{\infty}(\xn)$, there is a $k \in \Z$ such that $\shift{n}^{k}f_{\infty} \in \ker(\rsig_{\omega})$ if and only if  for a given $\Gamma \in \xns$ of long enough length, there are minimal numbers $l,r \in \N$, $r < l$, and  $\mu_1, \mu_2, \ldots, \mu_{n^r} \in \xnl{l}$ such that $((\bm{\Gamma})f_{\infty})_{\le 0} = \rev{\left(\bigcup_{1 \le i \le n^r} U_{\mu_i}\right)}$. Thus given such an element $f_{\infty} \in F_{\infty}$, the element of $\aut{\shift{n}}/ \gen{\shift{n}}$ represented by $f_{\infty}$ corresponds to a transducer $T \in \Ln{n,1}$ as described in Part~\ref{cor:exthm4}. This yields the result.
	\end{comment}
	
	Part~\ref{cor:exthm5} follows from the previous part, since when $n$ is prime, $r = 1$, and $\Kn{n} = \Ln{n}$.
\end{proof}

\begin{corollary}
	Let $n \in \N_{2}$, then, $\com{\aut{\xnz, \shift{n}}} \cong \com{\Ln{n}} \cong \com{\Kn{n}}$.  \qed
\end{corollary}

\begin{corollary}
	The order problem in $\Kn{n}$ is undecidable.
\end{corollary}
\begin{proof}
	By a result of \cite{KariOllinger}, the order problem in $\aut{\xnz, \shift{n}}$ is undecidable, the result now follows  Theorem~\ref{thm:exthom}.
	
\end{proof}

We note that it is a result of \cite{BelkBleakCameronOlukoya} that the groups $\Ons{r}$, $1 \le r \le n-1$, have unsolvable order problem. This now also follows from the corollary above since $\Kn{n} \unlhd \Ln{n,1} \le \Ons{1} \le \Ons{r}$.

\section{Reversing arrows and the dimension group} \label{Section:revanddimgroup}

In this section we show that the operation of ``reversing arrows'' as described in Subsection~\ref{subsection:reversingarrows} induces an automorphism of $\On$. We note that the homeomorphism of $\rho$ of $\xnz$ that sends a sequence $x \in \xnz$ to the element $y \in \xnz$ such that $y_i = x_{-i}$, induces, by conjugation, an automorphism of  $\aut{\xnz, \shift{n}}$ sending the shift map to its inverse. Therefore,  the map $\rho$ induces an automorphism of $\Ln{n}$. This automorphism corresponds precisely to ``reversing arrows'' of the elements of $\Ln{n}$. Thus one may recover a representative in $\Ln{n}$. Combinatorially, we note that the map $\rev{\phantom{a}}: \xn^{k} \to \xn^{k}$ induces a bijection, an involution, from $\rwnl{k}$ to itself for all $k \in \N$. Thus conjugation by this bijection induces an automorphism of $\sym{\rwnl{\ast}}$. An interpretation of the results of this section is that $(\On)\Pi$ (see Subsection~\ref{Subsection:MnOnLn} for the definition of the map $\Pi: \On \to \sym{\wns}$) is normalised by the bijection $\rev{\phantom{a}}: \rwnl{\ast} \to \rwnl{\ast}$. Since $(\On)\Pi \le \sym{\rwnl{\ast}}$ is isomorphic to $\On$ we obtain an automorphism of $\On$.  We note that for $T \in \On$, the bijection on the induced circuits of $\rev{T}$ from $\rwnl{\ast}$ to itself, is precisely the map $\rev{\phantom{a}} (T)\Pi \rev{\phantom{a}}$. To prove the main result of this section we develop a procedure that associates to each element $T \in \On$ an element $U \in \On$ such that $(U)\Pi = \rev{\phantom{a}} (T)\Pi \rev{\phantom{a}}$ (note that $U$ must be unique as $\Pi: \On \to \sym{\rwnl{\ast}}$ is a faithful representation ). 

We begin by introducing \textit{non-deterministic transducers}.  

\subsection{Non-deterministic, strongly synchronizing transducers}

In this subsection we often have to distinguish between equality in the monoid $\xns$ and equality of words. Therefore, given elements $w_1,\ldots, w_l, u_1, \ldots, u_{m} \in \xn \sqcup \{ \emptyword\}$,  we write $w_1 \ldots
w_l = u_1 \ldots u_m$ if and only if $l= m$ and $w_i = u_i$ for all $1 \le i \le m$; we write $w_1 \ldots w_l \equiv u_1 \ldots u_m$ if and only if $w_1 \ldots w_l$ and $u_1 \ldots u_l$ represent the same element of the monoid $\xns$.

A \textit{non-deterministic automaton} $A = \gen{\xn, Q_{A}, \pi_{A}}$ is a labelled directed graph with a finite vertex/state set $Q_{A}$, a finite edge set $\pi_{A} \subseteq  \xn^{\ast} \times Q_{A} \times Q_{A} $ ($|\pi_{A}|< \infty$) and a finite alphabet set $\xn$. In other words, a non-deterministic automaton is a finite directed graph with a finite number edges between states and with edges labelled by elements of $\xns$.

Let $A$ be a non-deterministic automaton, a path in $A$ is a word $$(w_1, p_1, q_1), (w_2, p_2, q_2),\ldots, (w_m, p_m, q_{m}) \in \pi_{A}^{\ast}$$ such that $q_{i} = p_{i+1}$ for $1 \le i \le m-1$. To simplify notation, we write $(w, p_1, q_{m})$ for the path $(w_1, p_1, q_1), (w_2, p_2, q_2),\ldots, (w_m, p_m, q_{m})$ where $w  =  w_1 w_2 \ldots w_m$.

We impose the following non-degeneracy condition on all non-deterministic automata in this article: for any circuit $(w, p, p)$, $w  \not \equiv \ew$.

Let $A$ be a non-deterministic automaton. For a state $q \in Q_{A}$, we define $\dom{q} \subseteq \xnn$ as follows. Let $x \in \xno$ be such that for any non-empty proper prefix $x_i$ of $x$ there are elements $w_1, w_2, \ldots, w_m, w \in \xns$, with $x_i$ a prefix of $w$, states $p_0=q,p_1, \ldots, p_m \in Q_{A}$, such that $(w_i, p_{i-1}, p_{i} )$ is an element of $\pi_{A}$ for all $1 \le i \le m$ and  $w_1w_2 \ldots w_m \equiv w$. Then $ x \in \dom{q}$.

A non-deterministic automaton $A = \gen{\xn, Q_{A}, \pi_{A}}$ is called \textit{strongly synchronizing} if it satisfies the following condition. For any word $u \in \xnp$, there is a unique (up-to rotation) sequence  $$(w_1, p_1, p_2), (w_2, p_2, p_3),\ldots, (w_m, p_m, p_{m+1}) \in \pi_{A}$$ with $p_{m+1} = p_1$ such that for $v \in \xnp$ with $v \equiv w_1 \ldots w_m$, the equality $\rotclass{v} = \rotclass{u}$ holds. 

The collision with terminology in Subsection~\ref{Subsection:synchandbisynch} is deliberate as Remark~\ref{rmk:stronglysynchautisstronglysynchnondetaut} makes clear. An interpretation of Proposition~\ref{Prop:recoveringrepinOn} is that strongly synchronzing non-deterministic automaton and strongly synchronizing automaton are essentially ``the same thing''.

A strongly synchronizing non-deterministic automaton $A$ is called \emph{core} if it is strongly connected.

\begin{Remark}\label{rmk:stronglysynchautisstronglysynchnondetaut}
		Every strongly synchronizing automaton is a strongly synchronizing non-deterministic automaton. On the other hand any deterministic automaton which is strongly synchronizing in the non-deterministic sense in strongly synchronizing in the deterministic sense. That is the two notions of being ``strongly synchronizing'' coincide on the class of deterministic automata.
\end{Remark}

\begin{Remark}\label{rmk:consequencesofstrongsynch}
	Let $A$ be a strongly synchronizing non-deterministic automaton. Given a prime word $u \in \xn^{+}$, the unique circuit $(w_1, p_1, p_2)(w_2, p_2, p_3)\ldots (w_m, p_m, p_{m+1})$ in $A$ with $w_1 \ldots w_m$ equivalent to an element of $\rotclass{u}$, is a prime word in $\pi_{A}^{\ast}$. This follows from the non-degeneracy condition and the fact that $u$ is a prime word. On the other hand, given a circuit $(w_1, p_1, p_2)(w_2, p_2, p_3)\ldots (w_m, p_m, p_{m+1})$ in $A$ which is a prime word in $\pi_{A}^{\ast}$,  the word $w_1\ldots w_m$ is equivalent to a prime word in $\xns$. This follows from the existence and uniqueness clause in the definition of strongly synchronizing automata. Thus it follows that given a strongly synchronizing automaton $A$, there will be circuits in $A$ which are prime words in $\pi_{A}$.
\end{Remark} 

\begin{Remark}\label{rmk:singlestateandstrongsynch}
	Let $A$ be a strongly synchronizing non-deterministic automaton with alphabet set $\xn$ and only one state. Then $A$ is the single state automaton with $n$ edges each labelled with a unique element of $\xn$.
\end{Remark}

A \textit{non-deterministic transducer} $T$ is a tuple $T = \gen{\xn, Q_{T}, \pi_{T}, \lambda_{T}}$ where $I(T) = \gen{\xn, Q_{T}, \pi_{T}}$ is a (non-degenerate) non-deterministic automaton called the \textit{input automaton}, and,  an \textit{output function}  $\lambda_{T}: \pi_{T} \to \xn^{\ast}$. We identify a non-deterministic transducer with a labelled digraph whereby an edge $(w,p,q)$  of $I(T)$ is labelled $w|\lambda_{T}(w,p,q)$.

Let $p_1, \ldots, p_{m+1} \in Q_{T}$ and $w_{1}, \ldots, w_m \in \xns$ such  $(w_{i}, p_i, p_{i+1})$ is an element of $\pi_{T}$ for all $1 \le i \le m$. Then, setting $w =  w_1 \ldots w_m $, we write $\lambda_{T}(w, p_1)$ for the word $\lambda_{T}(w_1, p_1, p_2) \ldots \lambda_{T}(w_{m}, p_m, p_{m+1})$ and write $(w, p_1, p_{m+1})$ for the path $(w_1, p_1, p_2) (w_2, p_2, p_3)\ldots (w_m, p_m, p_{m+1})$ labelled $w$ from $p_1$ to $p_{m+1}$.  Setting $\mathsf{w} \in \pi_{T}^{\ast}$ to be the path $(w_2, p_2, p_3)\ldots (w_m, p_m, p_{m+1})$ we  write $\lambda_{T}(\mathsf{w})$ for the word $\lambda_{T}(w_1, p_1, p_2) \ldots \lambda_{T}(w_{m}, p_m, p_{m+1})$.

We impose the following non-degeneracy condition on all non-deterministic transducers in this article: for any circuit $(w, p, p)$, $\lambda_{T}(w, p, p) \not \equiv \emptyword$. 

We identify the output function $\lambda_{T}$ with a subset of $\xns \times Q_{T} \times Q_{T}$. Thus, the non-degeneracy condition above means that the tuple $O(T)= \gen{\xn, Q_{T}, \lambda_{T}}$ is also a non-deterministic automaton called the \textit{output automaton}. Thus, the labelled digraph for the input automaton of $T$ is obtained by deleting output words in the digraph of the transducer; likewise the digraph for the output automaton of $T$ is obtained by input words in the labelled digraph of $T$.

For a non-deterministic transducer $T$ and $q \in Q_{T}$, we identify $\dom{q}$ with $\dom{q}$ as a state of $I(T)$.

At this point, one might expect the definition of the image $\im{q}$ of a state $q$ of a non-deterministic transducer $T$ and the corresponding map $\lambda_{T}(\cdot, q): \dom{q} \to \im{q}$. However, for a general non-deterministic transducer, the map $\lambda_{T}(\cdot,q): \dom{q} \to \im{q}$ is a multifunction.

Let $T$ be a non-deterministic automaton. Then $T$ is called \textit{strongly synchronizing} if $I(T)$ is strongly synchronizing; $T$ is called \textit{bi-synchronizing} if both $I(T)$ and $O(T)$ are strongly synchronizing. If $T$ is strongly synchronizing or bi-synchronizing, then $T$ is called \textit{core} if its input automaton is core.

The results that follow demonstrate that if we restrict to the class of strongly-synchronizing transducers, the proposed map $\lambda_{T}$ is in fact a well-defined function. 

The following function will be used regularly in the proofs that follow.

\begin{Definition}
	Let $A$ be a non-deterministic automaton or non-deterministic transducer. For a state $q \in Q_{A}$ define a function $L_{q}: \xns \to \pi_{A}^{\ast}$ as follows. Set $(\ew)L_{q} = \ew$; for a word $w \in \xnp$ we set $(w)L_{q} $ to be  the greatest common prefix of the set of all infinite paths, $(u_1, p_1, p_2) (u_2, p_2, p_3) \ldots \in \pi_{A}^{\N}$ where $p_1 = q$ and $w$ is equivalent in $\xns$ to a prefix of $u_1u_2 u_3 \ldots$.
\end{Definition}

\begin{lemma}\label{Lemma:injectionfromedgespacetodomain}
	Let $A$ be a core strongly synchronizing non-deterministic automaton. There is an $N \in \N$ such that the following holds. Let $q \in Q_{A}$, $x \in \dom{q}$, paths $(w, q, p)$ and $(w',q, p')$ in $A$ and words $\mu, \mu' \in \xns$, both beginning with the length $N$ prefix $x_{[1,N]}$ of $x$ be given. If $\mu \equiv w$ and $\mu' \equiv w'$, then the paths $(w,q,p)$ and $(w,q,p')$, as words in $\pi_{A}^{\ast}$, have a non-empty prefix in common.
\end{lemma}
\begin{proof}
%	For a state $q \in Q_{A}$ define a function $L_{q}: \xns \to \pi_{A}^{\ast}$ as follows. Set $(\ew)L_{q} = \ew$; for a word $w \in \xnp$ we set $(w)L_{q}$ to be  the greatest common prefix of the set of all infinite paths, $(u_1, p_1, p_2), (u_2, p_2, p_3), \ldots \in \pi_{A}$ where $p_1 = q$ and $w$ is equivalent in $\xns$ to a prefix of $u_1u_2 u_3 \ldots$.
	
	It suffices to show that there is an $N \in \N$ such that given $q \in Q_{T}$, and any word $\nu \in \xn^{N}$ which is a prefix of an element of $\dom{q}$, $(\nu)L_{q} \ne \varepsilon$.
	
	Let $d \in \N$ be such that for any element $(w, s,t) \in \pi_{A}$, $w$ is equivalent to a word of length at most $d$ in $\xns$. (We note that $d$ since $A$ has finitely many edges.) Set $$N:= \left(\frac{n^{d+1}-1}{n-1}\right) |Q_{A}| +1$$ noting that $\sum_{i=0}^{d}n^{i} = (n^{d+1}-1)/(n-1)$ is the size of the set of all words of length at most $d$.
	
	Let $\nu \in \xn^{N}$ be a prefix of an element of $\dom{q}$ and suppose $(\nu)L_{q} = \varepsilon$. Let $$(w_1,q=q_0, q_1)(w_2, q_1, q_2)...(w_a,q_{a-1}, q_a )$$ and $$(u_1,q=p_0, p_1)(u_2, p_1, p_2)...(p_b,p_{b-1}, p_b )$$ be paths beginning at $q$ such that $(w_1,q_0, q_1) \ne (u_1,p_0, p_1)$ and $\nu$ is equivalent to a prefix of both $w_1w_2\ldots w_a$ and $u_1u_2\ldots u_a$. 
	For $1 \le i \le |\nu|$, let $\nu_{[1,i]}$ be the length $i$ prefix of $\nu$. For each $1 \le i \le |\nu|$, set $(w_1,q_0, q_1)(w_2, q_1, q_2)...(w_{a_i},q_{a_i -1}, q_{a_i})$ and $(u_1,p_0, p_1)(u_2, p_1, p_2)...(u_{b_i},p_{b_i -1}, p_{b_i})$ to be, respectively, the maximal prefixes of $(w_1,q=q_0, q_1)(w_2, q_1, q_2)...(w_a,q_{a-1}, q_a )$ and $(u_1,q=p_0, p_1)(u_2, p_1, p_2)...(p_b,p_{b-1}, p_b )$ such that $w_1w_2\ldots w_{a_i}$ and $u_1 \ldots u_{b_i}$ are equivalent to prefixes of $\nu_i$. 
	Define $\alpha_i$ and $\beta_i$ as the unique elements of $\xns$ such that $\nu_i $ is equivalent to $w_1w_2\ldots w_{a_i} \alpha_i$ and $\nu_i $ is equivalent $u_1 \ldots u_{b_i} \beta_{i}$. 
	Note that $|\alpha_i|, |\beta_i| \le d$. 
	By choice of $N = |\nu|$, there are $i_1$ and $i_2$ such that $(q_{a_{i_1}},\alpha_{i_1}) = (q_{a_{i_2}}, \alpha_{i_2})$ and $(p_{b_{i_1}},\beta_{i_1}) = (p_{b_{i_2}}, \beta_{i_2})$. Set $\alpha = \alpha_{i_1}$ and $\beta = \beta_{i_1}$.
	
	Without loss of generality we assume $|\alpha| \le |\beta|$. Let $\nu(\alpha)_{i_1}, \nu(\beta)_{i_1} \in \xns$ be, respectively, equivalent to $w_1 \ldots w_{a_{i_1}}$ and $u_1 \ldots u_{b_{i_1}}$. Now as $|\alpha| \le |\beta|$, there is a $\xi \in \xns$ such $\nu(\alpha)_{i_1} = \nu(\beta)_{i_1} \xi$. Since $\nu(\alpha)_{i_1} \alpha = \nu(\beta)_{i_1} \beta = \nu_{i_1}$ it follows that $\beta = \xi \alpha$.
	
	Consider the paths $$(w_{a_{i_1} +1}, q_{a_{i_1}}, q_{a_{i_1} +1}) \ldots (w_{a_{i_2}}, q_{a_{i_2}-1}, q_{a_{i_2}})$$ and $$(u_{b_{i_1} +1}, p_{b_{i_1}}, p_{b_{i_1} +1}) \ldots (u_{b_{i_2}}, p_{b_{i_2}-1}, p_{b_{i_2}}).$$ There are $\delta(\alpha)$ and $\delta(\beta)$ in $\xns$ such that $w_{a_{i_1}+1} \ldots w_{a_{i_2}}$ is equivalent to $\alpha \delta(\alpha)$ and $u_{b_{i_1}+1} \ldots u_{b_{i_2}}$ is equivalent to $\beta\delta(\beta)$. 
	Using the facts that $\beta = \xi \alpha$, $\nu_{i_1} = \nu(\alpha)_{i_1} \alpha = \nu(\beta)_{i_1} \beta$ and $\nu_{i_1}\delta(\alpha)\alpha = \nu_{i_1} \delta(\beta)\beta  = \nu_{i_2}$ (since $\alpha_{i_2} = \alpha_{i_1} = \alpha$). It follows that $\delta(\beta)\xi = \delta(\alpha)$. Therefore, $\alpha \delta(\alpha) = \alpha \delta(\beta)\xi$ and $\beta \delta(\beta) = \xi \alpha \delta(\beta)$ are rotations of one another.
	
	Observe that $$(w_{a_{i_1} +1}, q_{a_{i_1}}, q_{a_{i_1} +1}) \ldots (w_{a_{i_2}}, q_{a_{i_2}-1}, q_{a_{i_2}})$$ and $$(u_{b_{i_1} +1}, p_{b_{i_1}}, p_{b_{i_1} +1}) \ldots (u_{b_{i_2}}, p_{b_{i_2}-1}, p_{b_{i_2}})$$ are in fact circuits based at $q_{a_{i_1}} = q_{a_{i_2}}$ and $p_{b_{i_1}} = p_{b_{i_2}}$ respectively. 
	Moreover as $w_{a_{i_1}+1} \ldots w_{a_{i_2}}$ is equivalent to $\alpha \delta(\alpha)$ and $u_{b_{i_1}+1} \ldots u_{b_{i_2}}$ is equivalent to $\beta\delta(\beta)$, we conclude, by the strong synchronizing condition, that $(w_{a_{i_1} +1}, q_{a_{i_1}}, q_{a_{i_1} +1}) \ldots (w_{a_{i_2}}, q_{a_{i_2}-1}, q_{a_{i_2}})$ is equal to a rotation of $(u_{b_{i_1} +1}, p_{b_{i_1}}, p_{b_{i_1} +1}) \ldots (u_{b_{i_2}}, p_{b_{i_2}-1}, p_{b_{i_2}})$. 
	This means  there are $j_1, j_2$, with $a_{i_1} \le j_1, j_2 < a_{i_2}$, such that $(w_{{j_1} +1}, q_{ {j_1}},q_{{j_1}+1}) = (u_{{j_2}+1}, p_{{j_2}},p_{{j_2}+1})$ and the words $w_{1} w_{2} \ldots w_{{j_1}+1}$ and $u_{1} u_{2} \ldots u_{{j_2}+1}$ are equivalent. %
	%
	%$(w_{a_i}, q_{a_{i}-1}, q_{a_i}) \ldots (w_{a_{i_2}}, q_{a_{i_2 }-1}, q_{a_{i_2}}) =(u_{b_i}, p_{b_{i}-1}, p_{b_i}) \ldots (u_{b_{i_2}}, p_{b_{i_2 }-1}, p_{b_{i_2}})$. 
	%
	However, the strong synchronizing condition again implies that this happens only when $$(w_1,q_0, q_1)(w_2, q_1, q_2)\ldots(w_{j_1 +1},q_{j_{1}-1}, q_{j_{1} +1} )$$ and $$(u_1,p_0, p_1)(u_2, p_1, p_2)\ldots(u_{j_2 +1},p_{j_{2}-1}, p_{j_{2}+1} )$$ are the same path in $A$. 
	For, otherwise, using the strongly connected hypothesis and the existence of prime circuits in $A$ (Remark~\ref{rmk:consequencesofstrongsynch}),  we may find a word $\varrho \in \pi_{A}^{\ast}$ such that $$(w_1,q_0, q_1)(w_2, q_1, q_2)\ldots(w_{j_1 +1},q_{j_{1}-1}, q_{j_{1} +1} ) \varrho$$ and $$(u_1,p_0, p_1)(u_2, p_1, p_2)\ldots(u_{j_2 +1},p_{j_{2}-1}, p_{j_{2}+1} )\varrho$$ are circuits based at $q = p_0 = q_0$ which are also prime words in $\pi_{A}^{\ast}$.  Remark~\ref{rmk:consequencesofstrongsynch} implies that the output of both of these circuits is equivalent to a prime word in $\xns$. Furthermore, by construction, the outputs of the two circuits are equivalent to the same word in $\xns$. It thereofore follows that, as the circuits cannot be rotations of one another, being prime, they must be equal. 
\end{proof}

\begin{corollary} \label{cor:strongsyncinjectionondomain}
	Let $A$ be a core strongly synchronizing non-deterministic automaton. Let $q \in Q_{A}$. For any $x \in \dom{q}$, there is a unique infinite path $$(w_1, q_0 =q, q_1)(w_2, q_1, q_2) (w_3, q_2, q_3) \ldots$$ in $A$ and a non-decreasing sequence $i_1 \le i_2 \le i_3 \ldots \in \N$ with $ \lim_{j \to \infty} i_j = \infty$ and $w_{1} \ldots w_{i_{j}} \equiv x_{[1, i_{j}]}$ for all $j \in \N_{1}$. 
\end{corollary}
\begin{proof}
	Since $x \in \dom{q}$ such a sequence exists by definition. Suppose 
	
	\begin{align}
	&(w_1, q_0 =q, q_1)(w_2, q_1, q_2) (w_3, q_2, q_3) \ldots  \label{path 1}\\
	\intertext{and}
	&(u_1, p_0 =q, p_1)(u_2, p_1, p_2) (u_3, p_2, p_3) \ldots \label{path 2}
	\end{align}
	
are two such sequences. We prove, by induction that they must be equal.

By Lemma~\ref{Lemma:injectionfromedgespacetodomain} there is an $i_1 \in \N_{1}$ such that $$(w_1, q_0, q_1)\ldots (w_{i_1}, q_{i_1 -1}, q_{i_1}) =  (u_1, p_0, p_1)\ldots (u_{i_1}, p_{i_1 -1}, p_{i_1}) \ldots.$$

Let $j_1 \in \N_{1}$ be such that $w_1 \ldots w_{i_1} = x_{[1, j_1]}$. Let $x'$ be defined by the equality $x = x_{[1, j_1]} x'$. Then it is clear that $x' \in \dom{ p_{i_1}}$. 

Replacing $q$ with $p_{i_1}$ and $x$ with $x'$ we may thus repeat the argument above with the infinite paths 

\begin{align*}\textbf{}
&(w_{i_1 + 1}, q_{i_1}=p_{i_1}, q_{i_1 +1 })(w_{i_1 + 2}, q_{i_1 + 1}, q_{i_1 +2}) \ldots \\
\intertext{and}
&(u_{i_1 + 1}, p_{i_1}=q_{i_1}, p_{i_1 +1 })(u_{i_1 + 2}, p_{i_1 + 1}, p_{i_1 +2}) \ldots.
\end{align*}

Inductively we conclude that the infinite paths \eqref{path 1} and \eqref{path 2} must be equal.

\end{proof}

In light of Corollary~\ref{cor:strongsyncinjectionondomain} above, we have the following definition. Let $T$ be a strongly synchronizing transducer and $q$ a state of $T$. Let  $x \in \dom{q}$ and let $(w_1, q_0 =q, q_1)(w_2, q_1, q_2) (w_3, q_2, q_3) \ldots$ be the unique infinite path in $A$  with  $w_{1} \ldots w_{i_{j}} \equiv x_{[1, i_{j}]}$, for all $j \in \N$ and for a non-decreasing divergent sequence $(i_{j})_{j \in \N}$ of natural numbers.  Write $\lambda_{T}(x, q)$ for the element $y \in \xno$ such that for any $j \in \N$, $\lambda_{T}(w_1, q_0 =q, q_1)\lambda_{T}(w_2, q_1, q_2) \lambda_{T}(w_3, q_2, q_3) \ldots \lambda_{T}(w_{j+1}, q_j, q_{j+1})$ is equivalent to a prefix of $y$. Write $\im{q}$ for the set $\{\lambda_{T}(x,q) \mid x \in \dom{q}\}$. 

\begin{Remark}\label{remark:strongbisynchbijectionfromdomaintorange}
	Let $T$ be a core and strongly synchronizing transducer. The map $\lambda_{T}(\cdot,q): \dom{q} \to \im{q}$ is a well-defined continuous and surjective map. If $T$ is bi-synchronizing then $\lambda_{T}(\cdot, q)$ is a homeomorphism. This follows by applying Corollary~\ref{cor:strongsyncinjectionondomain} to the output automaton of $T$.
\end{Remark}

\begin{comment}
\begin{Remark}\label{remark:loopcondition}
	Let $T$ be a non-deterministic transducer which is strongly connected. Suppose that $T$ satisfies the following condition:
	\begin{enumerate}[label =ND.\arabic*] 
		\item For any word $u \in \xnp$, there is a unique sequence $$(w_1, p_1, p_2), (w_2, p_2, p_3),\ldots, (w_m, p_m, p_{m+1}) \in \pi_{T}$$ with $p_{m+1} = p_1$ such that  for $v \in \xnp$ with $v \equiv w_1 \ldots w_m$, the equality $\rotclass{v} = \rotclass{u}$ holds. \label{hypo:loopconditon}
	\end{enumerate}
	A consequence of \ref{hypo:loopconditon} is that if $(v_1, p_1, p_2), \ldots, (v_m, p_m, p_{m+1}) \in \pi_{T}$ is the label of a path in $T$, then for any other path $(u_1, q_1, q_2), \ldots, (u_m, q_m, q_{m+1})$ such that $q_1 = p_1$ and $q_{m+1} = p_{m+1}$, we have  $v_1 \ldots v_m \not\equiv  u_1 \ldots u_m$.
\end{Remark}
\end{comment}

We have the following results.

\begin{lemma} \label{lemma:ND1impliesclopenimage}
	Let $A$ be a core strongly synchronizing non-deterministic automaton. Then for any state $q \in Q_{A}$, $\dom{q}$ is clopen.
\end{lemma}
\begin{proof}

	 It suffices to show that for $q \in Q_{T}$, there is a $\nu \in \xn$ such that $U_{\nu}^{+} \subseteq \dom{q}$.
	 
	 To see why this suffices let $p \in Q_{A}$ be any state. Using the strong connectivity of $A$, we may find a path $\mathsf{w} = (w_1,q=q_0, q_1)(w_2, q_1, q_2) \ldots (w_i, q_{i-1}, p)$. Write $\mu$ for the element of $\xns$ equivalent to $w_1, w_2, \ldots w_i$ and note that we may chose the path $\mathsf{w}$ such that $\mu$ has prefix $ \nu$. Let $N$ be as in the statement of Lemma~\ref{Lemma:injectionfromedgespacetodomain} and $\xi \in \xn^{N}$ be a prefix of an element of $\dom{p}$. Then, $(\mu \xi)L_{q}$ has prefix $\mathsf{w}$. This follows as $\xi$ has length $N$, and so if $(\mu \xi)L_{q} = (w_1,q=q_0, q_1)(w_2, q_1, q_2) \ldots (w_a, q_{a-1}, q_{a})$ for some $a<i$, then, for $\delta$ the prefix of $\mu$ equivalent to $ w_1\ldots w_{a}$, $(\mu \xi - \delta)L_{q_a} = \ew$ contradicting Lemma~\ref{Lemma:injectionfromedgespacetodomain}. Thus, any element of $\dom{q}$ with prefix $\mu \xi$ is the output of an infinite path beginning with $\mathsf{w}$. Since $U_{\mu}^{+} \subseteq \dom{q}$, it follows that $U_{\xi}^{+} \subseteq \dom{p}$.
	 Therefore, for $X(p) \subseteq \xn^{N}$, the set of length $N$ prefixes elements of $\dom{p}$,  $\dom{p} = \bigcup_{\xi \in X(p)} U_{\xi}^{+}$ proving that $\dom{p}$ is clopen. 
	 
	We now show that for $q \in Q_{A}$, $\dom{q}$ contains a clopen subset of $\xn^{\omega}$. 
	
	Once more using the strong connectivity of $A$, we may find a  circuit $\mathsf{w} \in \pi_{A}$ based at $q$.  Let $w \in \xnp$ be equivalent to the concatenation of labels on the edges in the circuit. We may further assume, since $T$ is strongly synchronizing, that $w$ is a prime word. 
	 Now, it is not hard to see using the fact that $T$ is strongly synchronizing and has finitely many states, that there is an $L \in \N$, such that for any state $p \in Q_{A}$ and any path $(v_1, p_0 = p, p_1)(v_2, p_1, p_2)\ldots$  for which $w^{L}$ is equivalent to a prefix of $v_1 v_2 \ldots$,  there is minimal $i_1 \in \N$ and a maximal $i_2 \in \N$ with $(v_{i_1}, p_{{i_1}-1}, p_{{i_1}})\ldots (v_{i_2}, p_{{i_2}-1}, p_{{i_2}})$ a power of $\mathsf{w}$ and $v_1 v_2 \ldots v_{i_2}$  equivalent to $w^{L}$.
	 
	 Let $\delta \in \xnp$ be any word such that $w \perp \delta$ . There is an $l \in \N$ with $l > L$ such that $w^{l} \delta$ is a prime word. The strong synchronizing condition means we can find a state $p \in Q_{A}$ and a circuit $\mathsf{v} \in \pi_{A}$ based at $p$ whose label corresponds to the word $w^{l} \delta$. By definition of $L$, it therefore follows there is a path leaving $q$ with domain label containing a prefix equivalent to $w^{a}\delta$ for some $a <l$.
	 
	 We observe that for any path $\mathsf{w}' \in \pi_{A}^{\ast}$ beginning at $q$ which does not have $\mathsf{w}$ as a prefix and whose domain label contains a prefix equivalent to $w$, the maximal power of $w$ which is equivalent to the domain label of the path $\mathsf{w}'$ is strictly less than $L$. This fact readily implies that for any word $\delta \in \xnp$ of length greater than $|w|$ and which does not contain $w$ as a prefix, $w^{L}\delta$ is a prefix of an element of $\dom{q}$. Therefore, by  induction, we see that $U_{w^{L}}^{+} \subseteq \dom{q}$. 	 
\end{proof}

\begin{Remark}\label{remark:bisycndomainandimageclopen}
	Let $T$ be a core bisynchronizing non-deterministic transducer, then for any state $q \in Q_{T}$, $\dom{q}$ and $\im{q}$ are clopen.
\end{Remark}

\begin{proposition}\label{Prop:recoveringrepinOn}
	
	Let $T$ be a core strongly synchronizing non-deterministic transducer. 
	Then there is a strongly synchronizing transducer $S$ with the following property. Given $u \in \xnp$ and writing $$(w_1, p_1, p_2) (w_2, p_2, p_3)\ldots (w_m, p_m, p_{m+1}= p_m) \in \pi_{T}^{+}$$ for the unique circuit (up to rotation) in $T$ such that $w_1 \ldots w_m$ is equivalent to an element in   $\rotclass{u}$, for $v \in \xnp$ such that $v\equiv \lambda_{T}(w_1, p_1, p_2) \ldots \lambda_{T}(w_m, p_m, p_{m+1})$, we have, $\rotclass{v} \equiv \rotclass{\lambda_{S}(u, q_u)}$.
\end{proposition}
\begin{proof}
%For a state $q \in Q_{T}$ define a function $L_{q}: \xns \to \pi_{T}^{\ast}$ as follows. Set $L_{q}(\ew) = \ew$; for a word $w \in \xnp$ we set $L_{q}(w) $ to be  the greatest common prefix of the set of all infinite paths, $(u_1, p_1, p_2), (u_2, p_2, p_3), \ldots \in \pi_{T}$ where $p_1 = q$ and $w$ is equivalent in $\xns$ to a prefix of $u_1u_2 u_3 \ldots$.

We begin by constructing the transducer $S$ from $T$, we then argue that $S$ has the required properties.
	
	Set  $Q_{S}$ to be the set of all pairs $(w, q)$ in $\xns \times Q_{T}$ where $U^{+}_{w} \subseteq \dom{q}$  and $(w)L_{q} = \emptyword$.
	
	We argue that $Q_{S}$ is  finite. 
	
	Let $q \in Q_{T}$ and $x \in \xns$. Since $\dom{p}$ is clopen for any $p \in Q_{T}$, then for any such $p$, there is a minimal  $l \in \N$ and words $\mu_1, \ldots, \mu_{r} \in \xn^{l}$  such that $\bigcup_{1 \le i \le r} U^{+}_{\mu_i}  = \dom{p}$. Furthermore, by the non-degeneracy condition, there is a  $j \in \N$, $j\ge 1$, such that for any state $p \in Q_{T}$ and any path $(w_1, p, p_1), \ldots, (w_j, p_{j-1}, p_{j})$, the word $w_1\ldots w_{j}$ is equivalent to an element of $\xnp$ of length at least $l$. 
	
	Suppose $(x, q, p)$ is an edge in $T$. Let $v,\nu \in \xns$ such that  $v \equiv x$ and $ \nu \equiv w_1 \ldots w_{j}$ where $(w_1, p, p_1), \ldots, (w_j, p_{j-1}, p_{j})$ is a path of length $j$  beginning at $p$. We note that $|\nu| \ge l$. 
	
	Suppose for a contradiction that $(v \nu)L_{q}$ is the empty word in $\pi_{T}^{\ast}$. This means that there are elements $y= y_1, y_2, \ldots, y_{m} \in \xn$, states $q_1, q_2, \ldots, q_{m} \in Q_{T}$, such that $(y_1,q, q_1),(y_2, q_1, q_2), \ldots,(y_m, q_{m-1}, q_m)$ is a path beginning at $q$, $y_1y_2 \ldots y_m$  has a prefix equivalent to $v\nu$ and  $(y_1, q, q_1) \ne (x,q,p)$. 
	
	Using the fact that $T$ is strongly connected, we may find a circuit $$(y_{m+1}, q_{m}, q_{m+1}), \ldots ,(y_{m+k}, q_{m+k-1}, q_{m+k} = q_m).$$ Let $\delta \equiv y_{m+1}  \ldots y_{m+k}$, noting that $\delta$ is non-empty by the non-degenracy assumption. Then $\delta^{\omega}$ is an element of $\dom{q_{m}}$.
	
	Since $|\nu| \ge l$, then $U^{+}_{\nu} \subseteq \dom{p}$ and  $\nu \delta^{\omega}$ is an element of $\dom{p}$. Thus there is a infinite path $(w_1', p, p_1'), (w_2', p_1', p_2') \ldots$  such that $w_1' w_2' \ldots \equiv v\nu\delta^{\omega}$. By the strong synchronizing condition,  it must be the case that there is an $N \in \N$ such that $(w_{N}', p_{N-1}', p_{N}')(w_{N+1}', p_{N}', p_{N+1}') \ldots = ((y_{m+1}, q_{m}, q_{m+1})\ldots (y_{m+k}, q_{m+k-1}, q_{m}))^{\omega}$. This follows since by a compactness type argument, using the fact that $T$ has finitely many states and $\delta$ is finite, there are $N, a \in \N$ such that $(w'_{N}, p'_{N-1}, p_{N}') (w_{N+1}', p_{N}', p_{N+1}') \ldots (w_{N+a+1}', p_{N+a}', p_{N+a}')$, $p_{N+a}' = p_{N-1}'$, is a circuit in $T$  and $w'_{N}w_{N+1}' \ldots w_{N+a+1}'$ is equivalent to an element of $\xns$ in $\rotclass{\delta^{b}}$ for some $b$.
	
	Therefore there are $a,b \in \N$ such that $$(x, q, p)(w_1', p, p_1') (w_2', p_1', p_2') \ldots (w_{N-1}', p_{N-2}', p_{N-1}')((y_{m+1}, q_{m}, q_{m+1})\ldots (y_{m+k}, q_{m+k-1}, q_{m}))^{a}$$ and $$(y, q, q_1), \ldots (y_{m}, q_{m}, q_{m+1})((y_{m+1}, q_{m}, q_{m+1})\ldots (y_{m+k}, q_{m+k-1}, q_{m}))^{b}$$ are paths in $T$, and $$xw_1' w_2' \ldots w_{N-1}' (y_{m+1} \ldots y_{m+k})^{a} \equiv y_1 y_2 \ldots y_{m}(y_{m+1}\ldots y_{m+k})^{b} \equiv v\nu \delta^{b}.$$ 
	
	We obtain the desired contradiction as follows. Set $q_{m+k+1} := q_{m}$. Using the strong connectivity hypothesis, we may  find a path $(y_{m+k+1}, q_{m+k}, q_{m+k+1}) \ldots (y_{m+k+r}, q_{m+k+r-1}, q_{m+k+r}) \in \pi_{T}$, where $q_{m+k+r} = q$.  By the strong synchronizing condition, it follows that 
	\begin{IEEEeqnarray*}{rCl}
	\mu:=&(&x, q, p)(w_1', p, p_1') (w_2', p_1', p_2') \ldots (w_{N-1}', p_{N-2}', p_{N-1}')\\&(&(y_{m+1}, q_{m}, q_{m+1})\ldots (y_{m+k}, q_{m+k-1}, q_{m}))^{a} \\ &(&y_{m+k+1}, q_{m+k}, q_{m+k+1}), \ldots, (y_{m+k+r}, q_{m+k+r-1}, q_{m+k+r})
	\end{IEEEeqnarray*}
	
	and 
	\begin{IEEEeqnarray*}{rCl}      \tau:=&(&y, q, q_1), \ldots (y_{m}, q_{m}, q_{m+1})((y_{m+1}, q_{m}, q_{m+1})\ldots (y_{m+k}, q_{m+k-1}, q_{m}))^{b}\\ &(&y_{m+k+1}, q_{m+k}, q_{m+k+1}), \ldots, (y_{m+k+r}, q_{m+k+r-1}, q_{m+k+r})
	\end{IEEEeqnarray*}
	are non-trivial rotations of one another. Since both are circuits based at $q$, there must be distinct words $\eta_1, \ldots
	\eta _c \in \pi_{T}^{+}$,  such that for $1 \le i \le c$,  $\eta_i$ is a minimal circuit based at $q$ (that is there is no non-empty prefix (in $\pi_{T}^{+}$) of $\eta_i$ which is a circuit based at $q$) and the following condition is satisfied. There is a word $w_1 \ldots w_d \in  \{ \eta_i \mid 1 \le i \le c \}^{+}$ such that  $\mu = w_1 \ldots w_d$ and $\tau = w_{i}\ldots w_d w_1 \ldots w_{i-1}$ where $w_1 \ne w_i$ (since $(x,q,p) \ne (y,q, q_1)$). Let $\overline{w_e}$, $1 \le e \le d$, be the element of $\xnp$ corresponding to the label of the circuit  $w_e$ in $T$. 
	Then $\overline{\mu}:=\overline{w_1}\ldots \overline{w_d} = \overline{w}_{i}\ldots \overline{w}_d\overline{w}_{1}\ldots \overline{w}_{i-1} =:\overline{\tau}$. However this now means that $\overline{\nu}$ and $\overline{\tau}$ are powers of the same prime word $\gamma \in \xnp$. In particular  $\overline{w}_{i}\ldots \overline{w}_d$ and $\overline{w}_{1}\ldots \overline{w}_{i-1}$ are both equal to  powers of $\gamma$.
	 Remark~\ref{rmk:consequencesofstrongsynch}  now means that there is a prime word $\xi \in \pi_{T}^{+}$ which is also a circuit based at $q$ such that $w_1 \ldots w_{i-1}$ is a power of $\xi$ and $w_i \ldots w_{d}$ is also a power of $\xi$. This now contradicts the fact that $w_1 \ne w_i$.
%	  a non-trivial rotation $\xi'$ of $\xi$. Let  $\xi = \xi_1 \xi_2$ and $\xi' = \xi_2 \xi_1$ for non-empty elements   $\xi_1, \xi_2 \in \pi_{T}^{+}$.  Let $\overline{\xi}_1$ and $\overline{\xi}_2$ be the elements of $\xnp$ corresponding to the label of the circuits $\xi_1$ and $\xi_2$ and $\xi'$ in $T$. Observe that, as $\overline{\xi}_1 \overline{\xi}_2 =\overline{\xi}_2 \overline{\xi}_1$ is a power of $\gamma$, then $\overline{\xi}_1$ and $\overline{\xi}_2$ are both powers of $\gamma$. Since $|\xi_1|$ and $|\xi_2|$ are both strictly less than  $\xi$ this contradicts hypothesis \ref{hypo:loopconditon}.
%	

	It therefore follows, since $q \in Q_{T}$ was chosen arbitrarily and $v \nu$ was an arbitrarily chosen prefix of an element of $\dom{q}$ corresponding to a path of length $j$, that the subset of $\xns \times Q_{T}$ consisting of elements $(u,q)$ where $U^{+}_{u}\subseteq \dom{q}$ and $(u)L_{q} = \ew$, is finite.

	We now argue that $Q_{S}$ is non-empty. 
	
	Let $q \in Q_{T}$ and let $\nu \in \xnp$ be a word such that $U^{+}_{\nu} \subseteq \dom{q}$. If $(\nu)L_{q} = \ew$, we are done. Therefore suppose  $(\nu)L_{q} = (x_1, q, q_1) (x_2,q_1, q_2) \ldots (x_m, q_{m-1}, q_m)$.  Let $v \in \xns$ be such that $v \equiv x_1 x_2 \ldots x_m$. By definition of the function $L_{q}$, $v$ is a prefix of $\nu$. Set $\mu:= \nu - v$, and consider the pair $(\mu, q_m)$. Since $U^{+}_{\nu} \subseteq \dom{q}$, it follows that  $U^{+}_{\mu} \subseteq \dom{q_m}$. Since $(\nu)L_{q} = (x_1, q, q_1) (x_2,q_1, q_2) \ldots (x_m, q_{m-1}, q_m) (\mu)L_{q_{m}}$, it follows that $(\mu )L_{q_m} = \ew$.

	We may now define a transducer $S = (\xn, Q_S, \pi_{S}, \lambda_{S})$ as follows. The transition function $\pi_{S}: \xn \times Q_S \to Q_S$, is defined such that $\pi_{S}(x, (w, q)) = (v, p)$ where $v$ and $p$ are determined as follows: set $(wx)L_{q} = (y_1, q, q_1) \ldots (y_m, q_{m-1}, q_m)$, and let  $y \in \xns$ satisfy $y \equiv y_1 \ldots y_m$, then $v = wx-y$, and $p = q_m$.  By an argument similar to that above which shows that $Q_{S}$ is non-empty, $\pi_{S}$ is well defined. The output function is determined by the following rule:  for $x \in \xn$ and $(w, q) \in Q_{S}$, if $(wx)L_{q} = (y_1, q, q_1) \ldots (y_m, q_{m-1}, q_m)$ and $y \in \xn$ is such that $y \equiv \lambda_{T}(y_1, q, q_1) \ldots \lambda_{T}(y_m, q_{m-1}, q_m)$, then $\lambda_{S}(x, (w,q)) = y$.
	
	We demonstrate that $S$ satisfies the conclusion of the proposition. 
	
	We begin by showing that given $u \in \xnp$ and $$(v_1, q_0, q_1) (v_2, q_1, p_2)\ldots (v_m, q_{m-1}, q_{m}= q_0) \in \pi_{T}^{+}$$ the unique circuit (up to rotation) in $T$ such that $v_1 \ldots v_m$ is equivalent to an element in   $\rotclass{u}$, there is a state $q_u \in Q_{S}$ with the following properties:
	\begin{itemize}
		\item  $\pi_{S}(u, q_u) = q_u$, and,
		\item for $w \in \xnp$ such that $w\equiv \lambda_{T}(v_1, q_0, q_1) \ldots \lambda_{T}(v_m, q_{m-1}, q_{m})$, we have, $\rotclass{w} \equiv \rotclass{\lambda_{S}(u, q_u)}$. 	
	\end{itemize}
    Once this fact is demonstrated, the proof will be concluded by showing the state $q_u \in Q_{S}$ satisfying $\pi_{S}(u, q_u) = q_u$ is unique --- this is sufficient to demonstrate that $S$ is strongly synchronizing (see Remark~\ref{rmk:stronglysynchautisstronglysynchnondetaut}). The proof of both aspects hinge on  the  strong synchronizing assumption on $T$ and the construction of the input and output function of $S$.
	
	Let $u \in \xnp$ be arbitrary. Since $T$ is strongly synchronizing, there exists a state $q \in Q_{T}$  a unique (up-to-rotation) circuit  $(v_1, q, q_1)(v_2, q_{1}, q_2) \ldots (v_m,q_{m-1},q_{m} = q_0)$ and a rotation $v$ of $u$ such that $v \equiv v_1v_2 \ldots v_{m}$.  
	Since  $\dom{q}$ is clopen, there is an $i \in \N$ such that $U_{v^{i}}^{+} \subseteq \dom{q}$.  
	Let $i \in \N$ be such that $$(v^{i})L_{q} = ((v_1, q, q_1)(v_2, q_{1}, q_2) \ldots (v_m,q_{m-1},q_{m}))^{i_1} (v_1, q, q_1)(v_2, q_{1}, q_2) \ldots (v_j,q_{j-1},q_{j}),$$  where  $i = i_1 + i_2 +1$ for some $i_2 \in \N$.  
	Set $\nu_{1} \in \xns$ such that  $\nu_1 \equiv v_1 v_2 \ldots v_{j}$ .
	  It follows that $v = \nu_1 \nu_2$ for some suffix $\nu_2$ of $v$.  We also observe that $U_{\nu_2v^{i_2}}^{+} \subseteq \dom{q_{j}}$ and $(\nu_2 v^{i_2})L_{q_{j}} = \varepsilon$ so that $(\nu_2 v^{i_2},q_{j})$ is an element of $Q_{S}$. Further observe that consequently, $U_{v^{i_2 +1}}^{+} \subseteq \dom{q}$.
	  
	 We claim that $$(\nu_2v^{i_2}v)L_{q_{j}}  = (v_{j+1}, q_{j}, q_{j+1})(v_{j+2}, q_{j+1}, q_{j+2}) \ldots (v_m,q,q_{m} )(v_1, q, q_1) \ldots (v_{j-1},q_{j-1},q_{j} ).$$
	This follows as  suppose there is an infinite path  $\mathsf{w}'$ in $T$ starting at $q_{j}$ such that $\lambda_{T}(\mathsf{w'})$ has a prefix equivalent to $\nu_2v^{i_2}v$ and does not begin with  $(v_{j+1}, q_{j}, q_{j+1})(v_{j+2}, q_{j+1}, q_{j+2}) \ldots (v_m,q,q_{m} )$.
	 Then we may find an infinite path $\mathsf{w}$ beginning at $q$ such that $\lambda_{T}(\mathsf{w})$  has a prefix equivalent to $v^{i_2 + 2}$ and does not begin with $(v_{1}, q, q_{1})(v_{2}, q_{1}, q_{2}) \ldots (v_m,q,q_{m} )$. However, using the fact that $U_{\nu^{i_2 +1}} \subseteq \dom{q}$,  there is a path $\mathsf{u}$ starting that $q$ beginning with $(v_{1}, q, q_{1})(v_{2}, q_{1}, q_{2}) \ldots (v_m,q,q_{m} )$  such that $\lambda_{T}(\mathsf{w})$ and $\lambda_{T}(\mathsf{u})$ are equivalent to the same element of $\im{q_{j}}$. This contradicts Corollary~\ref{cor:strongsyncinjectionondomain}.  The claim now follows by observing that $(v^{i_2+1})L_{q} = (v_1, q, q_1)(v_2, q, q_2) \ldots (v_j,q_{j-1},q_{j} = q_0)$  since $(\nu_2v^{i_2})L_{q_{j}} = \ew$ and $U_{\nu_2 v^{i_2}} \subseteq \dom{q_{j}}$. 
	
	We therefore see that $\pi_{S}(v, (\nu_{2}v^{i_2}, q_{j})) =(\nu_{2}v^{i_2}, q_{j})$.
	 Also notice that, by construction, $\lambda_{S}(v, (\nu_{2}v^{i_2},q_{j}))$ is the element of $\xnp$ equivalent to $$ \lambda_{T}(v_{j+1}, q_{j}, q_{j+1})\lambda_{T}(v_{j+2}, q_{j+1}, q_{j+2}) \ldots \lambda_{T}(v_m,q,q_{m} )\lambda_{T}(v_1, q, q_1) \ldots\lambda_{T} \lambda_{T}(v_{j-1},q_{j-1},q_{j} ).$$ 
	 Therefore  $\lambda_{S}(v, (\nu_{2}v^{i_2},q_{j}))$ is in the rotation class of the element of $\xnp$ equivalent to $\lambda_{T}(v_{1}, q, q_{1})\lambda_{T}(v_{2}, q_{1}, q_{2}) \ldots \lambda_{T}(v_m,q_{m-1},q_{m} )$. Thus, it follows that for any $u \in \xnp$  and $$(v_1, q_0, q_1) (v_2, q_1, p_2)\ldots (v_m, q_{m-1}, q_{m}= q_0) \in \pi_{T}^{+}$$ the unique circuit (up to rotation) in $T$ such that $v_1 \ldots v_m$ is equivalent to an element in   $\rotclass{u}$, there is a state $q_u \in Q_{S}$ with the following properties:
	\begin{itemize}
		\item  $\pi_{S}(u, q_u) = q_u$, and,
		\item for $w \in \xnp$ such that $w\equiv \lambda_{T}(v_1, q_0, q_1) \ldots \lambda_{T}(v_m, q_{m-1}, q_{m})$, we have, $\rotclass{w} \equiv \rotclass{\lambda_{S}(u, q_u)}$. 	
	\end{itemize}
		
Now, to conclude the proof of the proposition, as discussed, we need only argue that the state  $q_u$ is unique.
	
Once more let $u \in \xnp$ be arbitrary. 
 Let $(w,p)$ be a state of $S$ such that  $\pi_{S}(u, (w,p)) = (w,p)$.   Let $(v_1, q_0, q_1) \ldots (v_{m}, q_{m-1}, q_{m})$ be the unique circuit of $T$ (up-to-rotation) such that for $v \in \xnp$ with $v \equiv v_1 \ldots v_m$, we have $\rotclass{v} = \rotclass{u}$. Let $j \in \N$ be such that $v^{j} \subseteq \dom{ q_0}$ and $(v^{j})L_{q_0} = (v_1, p_0, p_1) \ldots (v_k, p_{k-1}, p_k)$ where $p_0 = q_0$. Let $v^{i_1} \overline{v}$ be such that $v_1 \ldots v_k \equiv v^{i_1}\overline{v}$ and let $\underline{v}v^{i_2}$ be such that $v = \overline{v}\underline{v}$ and $j = i_1 + i_2 +1$.  We note that $(\underline{v}v^{i_2})L_{p_{k}} = \ew$ and $U^{+}_{\underline{v}v^{i_2}} \subseteq \dom{p_k}$.  Consequently,  $v^{i_2 + 1} \subseteq \dom{ p_0}$. Set  $j = i_2 + 1$. 
 %This follows since as $(w_1, q_0, q_1) \ldots (w_{m}, q_{m-1}, q_{m})$ is either a prefix of or has $(v_1, p_0, p_1) \ldots (v_k, p_{k-1}, p_k)$ as a prefix, then, as $U^{+}_{\underline{v}v^{i_2}} \subseteq \dom{p_k}$, there is $0 \le l \le k-1$ such that $$(v_1, p_0, p_1) \ldots (v_k, p_{k-1}, p_k) = ((w_1, q_0, q_1) \ldots (w_{m}, q_{m-1}, q_{m}))^{i_1} (v_{p_l}, p_0, p_1) \ldots (v_k, p_{k-1}, p_k).$$   
 It follows that $v_1 \ldots v_{k} = \overline{v}$. Moreover, replicating an argument above, we see that $\pi_{S}(v, (\underline{v}v^{i_2}, p_k)) = (\underline{v}v^{i_2}, p_{k})$. 
	
	Let $l$ be sufficiently large and set $$(x_1, t_0, t_1), \ldots, (x_a, t_{a-1},  t_a) := (wu^{l})L_{p}.$$
	 Now as $\pi_{S}(u, (w,p)) = (w,p)$, for $l$ large enough it must be the case that   there are $b, N \in \N$ with $1 \le b \le a$  and $N$ maximal such that $$(x_1, t_0, t_1) \ldots (x_b t_{b-1}, t_b)((v_1, q_0, q_1)\ldots (v_m, q_{m-1}, q_m))^{N}$$ is a prefix of $(x_1, t_0, t_1), \ldots, (x_a, t_{a-1},  t_a)$.  
	This follows by the strong synchronizing hypothesis, the finiteness of $T$, and the fact that  $x_1 \ldots x_a$ must be equivalent to the length $|wu^{l}| - |w|$ prefix of $wu^{l}$.
	
	Let $\nu \equiv x_1 \ldots x_b$ so that $x_1 \ldots x_b (v_1 \ldots v_m)^{N} \equiv \nu v^{N}$. Note that $\nu v^{N} - w$ is a prefix of $u^{l}$.  By choosing $l$ large enough we may assume that $N > i_2+1$. As above this now means that $(\nu v^{N})L_{p} =  (x_1, t_0, t_1) \ldots (x_b t_{b-1}, t_b) (v^{N})L_{q_0}$. Hence, $(\nu v^{N})L_{p} = (x_1, t_0, t_1) \ldots (x_b t_{b-1}, t_b)((v_1,q_0, q_1)\ldots (v_{m}, q_{m-1}, q_m))^{N-i_2-1}(v_1, p_0, p_1)\ldots(v_k, p_{k-1}, p_k)$. This means that $\pi_{S}(\nu v^{N} - w,(w,p)) = (\underline{v}v^{i_2}, p_k)$. Thus there are $v_1, v_2 \in \xns$ such that $v = v_1v_2$, $u = v_2v_1$, $\pi_{S}(v_1, (\underline{v}v^{i_2}, p_k)) = (w, p)$ and $\pi_{S}(v_2, (w,p)) = (\underline{v}v^{i_2}, p_k)$. From this we see that $(w,p)$ is the unique state of $S$ satisfying $\pi_{S}(u,(w,p)) = (w,p)$ since the state $(\underline{v}v^{i_2}, p_k))$ depends only on the circuit $(v_1, q_0, q_1) \ldots (v_{m}, q_{m-1}, q_{m})$.

\end{proof}

\begin{comment}
Set the following property for a non-deterministic transducer $T$:
\begin{enumerate}[label= ND.\arabic*] \setcounter{enumi}{1}
	\item $T$ is strongly-connected and $\dom{q}$ is a clopen subset of $\xnn$ for any $q \in Q_{T}$. \label{hypo:clopenimage}
\end{enumerate}
\end{comment}

\begin{Definition}
	Let $T$ be a core strongly synchronizing non-deterministic transducer. Write $\rec{T}$ for the transducer with state set $Q_{\rec{T}} = \{ (w,q) \mid U^{+}_{w} \subseteq \dom{q}, (w)L_{q} = \ew \}$. The transition and output functions of $\rec{T}$ are defined as follows:
	\begin{itemize}
		\item  $\pi_{\rec{T}}: \xn \times Q_{\rec{T}} \to Q_{\rec{T}}$, is defined such that $\pi_{{\rec{T}}}(x, (w, q)) = (v, p)$ where $v$ and $p$ are determined as follows: set $(wx)L_{q} := (y_1, q, q_1) \ldots (y_m, q_{m-1}, q_m)$, and  let $y \in \xn$ satisfy $y \equiv y_1 \ldots y_m$, then $v = wx-y$, and $p = q_m$;
		\item  for $x \in \xn$ and $(w, q) \in Q_{\rec{T}}$, if $(wx)L_{q} = (y_1, q, q_1) \ldots (y_m, q_{m-1}, q_m)$ and $y \in \xn$ satisfies $y \equiv \lambda_{T}(y_1, q, q_1) \ldots \lambda_{T}(y_m, q_{m-1}, q_m)$, then $\lambda_{{\rec{T}}}(x, (w,q)) = y$. 
	\end{itemize}  As demonstrated in Proposition~\ref{Prop:recoveringrepinOn},   $Q_{\rec{T}}$ is non-empty and  $\pi_{\rec{T}}$ and $\lambda_{\rec{T}}$ are well defined. 
\end{Definition}

\begin{Remark}\label{remark:remarksonrecoveredsynchtransducer}
	We note that for a core strongly synchronizing non-deterministic transducer $T$, $\rec{T}$ is non-degenerate (since by assumption $T$ is non-degenerate). This can be argued by using induction to show that for a state $(w,q) \in Q_{\rec{T}}$ and a word  $v \in \xnp$, $\lambda_{\rec{T}}(v, (w,q))$ is equivalent to the output of the path $(wv)L_{q}$ in $T$. The claim is therefore a consequence of   Proposition~\ref{Prop:recoveringrepinOn}. It is a simple exercise to demonstrate that for a core strongly synchronizing deterministic transducer $T$, $\rec{T}$ and $T$ are in fact isomorphic as transducers (that is there is a bijection from the state set of $T$ to the state set of $\rec{T}$ which commutes with the transition and output function; indeed such a bijection maps a state $q$ of $T$ to the state $(\ew, q)$ of $\rec{T}$).
\end{Remark}

We note that for a given core strongly synchronizing non-deterministic transducer $T$, $\rec{T}$ is not necessarily minimal. We write $\mrec{T}$ for the minimal representative of $\rec{T}$. We denote the states of $\mrec{T}$  by the label of the corresponding states of $\rec{T}$ as,  having removed incomplete response, a state of $\rec{T}$ is $\omega$-equivalent to a unique state of $\mrec{T}$. That is, given a state $q $ of $\rec{T}$ by the phrase``the state $q$ of $\mrec{T}$'' (or words to the same effect)  we mean the unique state of $\mrec{T}$ which is $\omega$-equivalent to the state $q$ of $T$.

We have the following lemma.

\begin{lemma}\label{lem:determiningimage}
	Let $T$ be a core strongly synchronizing non-deterministic transducer and let $ (w,q ) \in Q_{\rec{T}}$. Then  the image of the state $(w,q)$ is precisely the set of all elements $\rho \in \xno$ such that $\rho$ is equivalent to $\lambda_{T}(w_1, q_0, q_1)\lambda_{T}(w_2, q_1, q_2)\lambda_{T}(w_3, q_2, q_3) \ldots$ where $q_0 = q$ , $(w_1, q_0, q_1)(w_2, q_1, q_2)(w_3, q_2, q_3)\ldots$ is an infinite path in $T$ beginning at $q$, and $w_1w_2\ldots$ as an element of $\xno$ has $w$ as a prefix. In other words, $\im{(w,q)} = (U_{w}^{+})\lambda_{T}(\cdot, q) \subseteq \im{q}$.
\end{lemma}
\begin{proof}
	This follows essentially from the definition of the output function of $\rec{T}$.  We demonstrate this with a straight-forward induction argument.
	
	Let $q \in Q_{T}$ and $w \in \xns$ such that $(w,q)$ is a state of $\rec{T}$
	
	Let $ v \in \xno$ be any element. By definition $wv \in \dom{q}$ and so there is an infinite path $(w_1, q_0, q_1)(w_2, q_1, q_2)(w_3, q_2, q_3) \ldots$ in $T$ with $q_0 = q$, such that $w_1w_2w_3\ldots \equiv wv$.
	
	Let $(v_i)_{i \in \N}$, be a sequence of prefixes of $v$ such that $|v_i|$ tends to infinity with $i$. Since $\rec{T}$ is non-degenerate and finite, $|\lambda_{\rec{T}}(v_i, q_0)|$ also tends to infinity with $i$. Moreover, as $\lambda_{\rec{T}}(v_i, q_0)$ is equivalent to the output of the path $(wv_i)L_{q}$ in $T$ (see Remark~\ref{remark:remarksonrecoveredsynchtransducer}), we see that $(\lambda_{\rec{T}}(v_i, q_0))_{i \in \N}$ is equivalent to a sequence of prefixes of  $\lambda_{T}(w_1, q_0, q_1)\lambda_{T}(w_2,q_1, q_2)\lambda_{T}(w_3, q_2, q_3)\ldots \equiv wv$.
	
	On the other hand given $y \in \xno$ an element of $\im{(w,q)}$,  by construction, there is a sequence of elements $v_1 < v_2 \ldots$ in $\xnp$, such that the output of the path $(wv_i)L_{q}$ is equivalent to $\lambda_{T}(v_i, q)$ and is a prefix of $y$. However, for $i$ large enough, $w$ is equivalent to a prefix of the  concatenation of the input labels of the path $(wv_i)L_{q}$ yielding the result.
	
\end{proof}

\subsection{Reversing arrow automorphism}\label{Section:revarrowaut}

We formalise the construction sketched out in the introduction to the section.

\begin{Definition}
	Let $T = \gen{\xn,Q_{T}, \pi_{T}, \lambda_{T}} \in \On$. 
	Define $\pi_{\rev{T}}$ as the  subset  of $\xn \times Q_{T} \times Q_{T}$ such that an element $(x, q, p) \in \xn \times Q_{T} \times Q_{T}$ is an element of $\pi_{\rev{T}}$ if and only if $\pi_{T}(x, p) = q$. 
	Define $\lambda_{\rev{T}} : \pi_{\rev{T}} \to \xns$ as follows for $(x,q,p) \in \pi_{\rev{T}}$, $\lambda_{\rev{T}} = w$ if and only if $\lambda_{T}(x,p) = \rev{w}$.  
	Set $\rev{T} = \gen{\xn,  Q_{T}, \pi_{T}, \lambda_{T}}$ a non-deterministic transducer called the \emph{reverse of $T$}.
\end{Definition}

Given an element $T \in \On$, in order to distinguish states of $q$ and states of $\rev{T}$ , given a state $q \in Q_{T}$, we write $\rev{q}$ for the corresponding state of $\rev{T}$.

\begin{Remark}\label{remark:revTbisynch}
	Let $T \in \On$. Then $\rev{T}$ is a strongly connected and bisynchronizing non-deterministic transducer. This follows since the maps $\rev{{}}: \wns \to \wns$  and $(T)\Pi: \wns \to \wns$ are bijective. 
\end{Remark}

\begin{notation}
	Let $T \in \On$ and $q \in Q_{T}$, we write $h_{\rev{q}}$ for the map $\lambda_{\rev{T}}(\cdot, \rev{q}): \dom{\rev{q}} \to \im{\rev{q}}$.
\end{notation}

\begin{Remark}\label{rem:synchinreverse}
	Let $T \in \On$ and let $k \in \N$ be a synchronizing level of $T$. The following fact follow easily from the synchronizing property and the definition of $\rev{T}$. For any word $\gamma \in \xns$  of length at least $k$, there is a unique state $\rev{q} \in Q_{\rev{T}}$ such that $U^{+}_{\gamma} \subseteq \dom{\rev{q}}$. Moreover, for this vertex $\rev{q}$, for any state $p$ in $Q_{T}$ there is exactly one path $(\gamma, \rev{q}, \rev{p})$ from $\rev{q}$ to $p$. This can be summarised by saying that a word $\rev{\gamma} \in \xn^{k}$  forces a state $q \in Q_{T}$ if and only if $U^{+}_{{\gamma}} \subseteq \dom{\rev{q}}$.
\end{Remark}

The following result establishes a connection between the construction of the transducer $\rec{T}$  for $T \in \On$ and Construction~\ref{construction:inverse} of the inverse of $T$. The proof follows straight-forwardly from the definitions and is left  to the reader.

\begin{lemma}\label{lemma:swapdomainandrangeforinverse}
	Let $T \in \On$. Set $Q_{T^{-1}}:= Q_{T}$ and set
	\begin{itemize}
		\item $\pi_{T^{-1}} \subseteq \xns \times Q_{T} \times Q_{T}$ to be the set of all elements $(w,q,p)$ such that there is a $v \in \xns$ with $\pi_{T}(v,q) = p$ and $\lambda_{T}(v,q) = w$;
		\item $\lambda_{T^{-1}}: \pi_{T^{-1}} \to \xns$ by $(w,q,p) \mapsto v$  if and only if $\pi_{T}(v,q) = p$ and $\lambda_{T}(v,q) = w$.
	\end{itemize}
Let $T^{-1}$ be the non-deterministic transducer $\gen{\xn, Q_{T^{-1}},\pi_{T^{-1}},\lambda_{T^{-1}}}$. Then $T^{-1}$ is non-degenerate and bisynchronizing. Moreover $\rec{T^{-1}}$ is precisely the transducer $T'$ obtained by applying Construction~\ref{construction:inverse} to $T$.
\end{lemma}

The following useful fact is a  corollary of Lemma~\ref{lemma:swapdomainandrangeforinverse} and follows from Proposition~\ref{Prop:recoveringrepinOn}:

\begin{lemma}\label{lemma:constructedinversesynch}
	Let $T \in \On$  and let $T'$ be the inverse of $T$ as in Construction~\ref{construction:inverse}. Then $T'$ is strongly synchronizing.
\end{lemma}

 Now we prove the main result of the subsection:

\begin{Theorem}
	The map $\overleftarrow{\mathfrak{r}}: \On \to \On$ defined by
	$$T \mapsto \mrec{\rev{T}}$$ is an automorphism of $\On$. Moreover the restriction $\overleftarrow{\mathfrak{r}} \restriction_{\Ln{n}}$ is an automorphism of $\Ln{n}$.
\end{Theorem}
\begin{proof}
	The map $\overleftarrow{\mathfrak{r}}$ is well-defined by Proposition~\ref{Prop:recoveringrepinOn}, Remark~\ref{remark:revTbisynch} and the definition of $\mrec{\rev{T}}$.
	
	We show that $\overleftarrow{\mathfrak{r}}$ is an automorphism. First note that as the homomorphism, $\Pi: \On \to \sym{\rwnl{\ast}}$ is faithful, then, by Proposition~\ref{Prop:recoveringrepinOn}, $\rev{\mathfrak{r}}$ is injective. This follows since, by Proposition~\ref{Prop:recoveringrepinOn}, for $T \in \On$, $(\revrec{T})\Pi$ is precisely the map ${(\rev{\phantom{a}}}(T)\Pi\rev{\phantom{a}}): \wns \to \wns$. It is also straight-forward to see using Proposition~\ref{Prop:recoveringrepinOn}, that given $T,U \in \On$, then $(\rec{T})\Pi (\rec{U})\Pi = (\rec{TU})\Pi$. Therefore, as $\Pi$ is faithful, $\mrec{T}\mrec{U} = \mrec{TU}$.
	
	The moreover statement follows from the fact that for $\overleftarrow{\mathfrak{r}}$ preserves the defining condition, condition \ref{Lipshitzconstraint}, of $\Ln{n}$ in $\On$.
\end{proof}

\begin{Remark}
	It is clear that for $T \in \On$, $\revrec{\revrec{T}} = T$.   
\end{Remark}

We now show that for $n > 2$, for any element $T \in \hn{n}$, with $|T| > 1$, $\revrec{T} \notin \hn{n}$. In particular $\revrec{\hn{n}} \cap \hn{n}$ has size $n!$ and consists entirely of single state transducers. The assumption that $n>2$ is essential since $\hn{2}$ is isomorphic to the cyclic group of order $2$.

\begin{lemma}\label{lem:detectinghomeo}
	Let $T$ be a non-deterministic strongly connected and strongly synchronizing transducer and let  $(w,q) \in Q_{\rec{T}}$. Then  $(w,q)$  corresponds to a homeomorphism state of $\mrec{T}$ if and only if $(U^{+}_{w})h_{q} = U^{+}_{\nu}$ for some $\nu \in \xns$.
\end{lemma}
\begin{proof}
	Suppose $(U^{+}_{w})h_{\rev{q}} = U^{+}_{\nu}$ for some $\nu \in \xns$. By Lemma~\ref{lem:determiningimage} $\im{(w,q)} =U^{+}_{\nu}$. Thus, after removing incomplete response,  the state $(w,q)$ corresponds to a a state of $\mrec{T}$ with image equal to $\CCn$. That is, $(w,q)$ corresponds to a homeomorphism state of $\mrec{T}$.
	
	Now suppose that $(w,q)$ corresponds to a homeomorphism state of $\mrec{T}$. This must mean that after removing incomplete response from the state $(w,q)$ of $\mrec{T}$ the image of $(w,q)$ is $\CCn$. From this it follows that $\im{(w,q)} = U^{+}_{\nu}$ for some $\nu \in \xns$, where $\nu$ is the longest guaranteed output from all sufficiently long inputs read from the state $(w,q)$. Thus  $(U^{+}_{w})h_{\rev{q}} = U^{+}_{\nu}$ by Lemma~\ref{lem:determiningimage}.
\end{proof}

\begin{proposition}\label{prop:hnthrownoffitself}
	Let $T \in \hn{n}$, then $\revrec{T} \in \hn{n}$ if and only if  the minimal representative of $T$ has size $1$.
\end{proposition}
\begin{proof}
	Let $T \in \hn{n}$ and suppose that $T$ is minimal.	If $|T| =1$ then $\rec{T} = T$.
	
	 Thus suppose that $|T|> 1$ and $\revrec{T} \in \hn{n}$. We show that $|T|=1$. We mainly work with the transducer $\rec{\rev{T}}$ as  $\mrec{\rev{T}} = \revrec{T}$  is obtained form $\rec{\rev{T}}$ by removing states of incomplete response and identifying $\omega$-equivalent states.
	
	By Lemma~\ref{lem:detectinghomeo}, since every state of $\revrec{T}$ is a homeomorphism state, it follows that follows that for any state $(w,q) \in  Q_{\rec{\rev{T}}}$, $\im{(w,q)} = U^{+}_{\nu}$ for some  $\nu \in \xns$.
	
	By Remark~\ref{rem:synchinreverse} it is not hard to see that  a pair $(w,q)$ is a state if $Q_{\rec{T}}$ if and only if  $\pi_{T}(\rev{w}, \cdot): Q_{T} \to Q_{T}$ takes only the value $q$ and for any proper prefix $v$ of $\rev{w}$, $\pi_{T}(v, \cdot): Q_{T} \to Q_{T}$ has image size at least $2$.
	
	We first consider the case that $(\emptyword, q)$ is a state of $Q_{\rev{T}}$ for some $q \in Q_{T}$. This means that $\dom{\rev{q}} = \xno$.  From which we deduce that $\pi_{T}(\centerdot, q) : \xns \to Q_{T}$ takes only the value $q$. Therefore, as $T$ is core, $|T| = 1$. 
	
	We may now assume that for any element $(w,q) \in Q_{\rec{\rev{T}}}$, $w \ne \emptyword$.
	
	Let $(w,q) \in Q_{\rec{\rev{T}}}$ be a state, and  suppose that $(U^{+}_{w})h_{\rev{q}} = U^{+}_{\nu}$. Set $w = w_1 w_2 \ldots w_l$. Thus it follows that for any state $p \in  Q_{T}$, and any element  $v \in \xns$ such that $l+|v| = |\nu|$, $\lambda_{T}(vw_{l}\ldots w_1, p) = \rev{\nu}$. 
	
	If $|\nu| > |w|$, then for any state $p \in Q_{T}$, and any word $v \in \xn^{|\nu|-l}$, $\lambda_{T}(v, p)$ is a non-empty prefix of $\rev{\nu}w_{l}$.  This contradicts the fact that $\lambda_{T}(\cdot, p): \xn \to \xn$ is a bijection for any state $p \in Q_{T}$. 
	
	Suppose that $|\nu| < |w|$. Then swapping $T$ with $T^{-1}$ and repeating the argument above with $q^{-1}$ in place of $q$ and swapping the roles of $\nu$ and $w$, we obtain a contradiction as before. Therefore $|w| = |\nu|$.
	
	We deduce that for any state $p \in Q_{T}$, $\lambda_{T}(\rev{w}, p) = \rev{\nu}$ and $\pi_{T}(\rev{w}, p) = q$. 
	
	Consider the set	
	$$X:= \{ w \in \xnp: (w,q) \in Q_{\rec{\rev{T}}} \}.$$
		 We observe that $X$ forms a prefix code for $\xns$.  Now, since for distinct $p, p'$ in $Q_{T}$, $\lambda_{T}(\cdot, p)$ and $\lambda_{T}(\cdot, p')$ coincide on $X$ and $\pi_{T}(\cdot, p)$ and $\pi_{T}(\cdot, p')$ coincide on $X$, it follows that all the states of $T$ are $\omega$-equivalent. Since  $T$ is  assumed to be minimal, $|T| = 1$.
\end{proof}

\begin{Remark}
	We make a few comments. 
	
	Proposition~\ref{prop:hnthrownoffitself} shows that whenever  $n\ge 3$, $\overleftarrow{\mathfrak{r}}$ is a non-trivial automorphism. However, it is not hard to construct examples of elements $T \in \On$ such that $\revrec{T} \neq T$. For example, the transducer $T$ in  Figure~\ref{fig:L2withloopstatenonhomeo} is the image under $\overleftarrow{\mathfrak{r}}$ of an element $U \in \Ln{n}$ for which every loop state is a homeomorphism state. Since the $0$ loop state of $T$ is not a homeomorphism state, $T \ne U$.

\end{Remark}

\subsection{Reversing arrows and the map \texorpdfstring{$\rsig$}{Lg}}\label{subsection:reversingarrows}

In this section we study the homomorphism $$\rev{\sig}: \On \to \Un{n-1}$$ defined by $$T \mapsto \revrec{T}\rsig.$$ The restriction to the subgroup $\Ln{n}$ induces a homomorphism which is precisely the inverse of the map $\rsig: \Ln{n} \to \Un{n-1}$.  As with the map $\rsig$, we extend $\rev{\sig}$ to a map from $\aut{\shift{n}, \xn^{\Z}}$ onto the group $\Pi_{k \in \N}\Un{n^{k}-1, n^{k}}$. The resulting map is more straightforwardly related to the dimension group of the full two-sided shift as exposited, for example, in \cite{BoyleLindRudolph88, BoyleMarcusTrow, KriegerDimension}.

Using Lemma~\ref{lem:determiningimage}, given $T \in \On$ we relate the map $\rev{\sig}$ to the images of the states of the non-deterministic transducer $\rev{T}$.

\begin{Definition}\label{Def:mqgamma}
	Let $T \in \On$ be synchronizing and let $q \in Q_{T}$. Let $\gamma \in \xn^{\ast}$ be a synchronizing word of any length such that $U^{+}_{\gamma}  \subseteq \dom{\rev{q}}$.  Set $m_{\rev{q}, \gamma}$ to be the smallest $r \in \N$, such that there exists $\mu_1, \mu_2, \ldots, \mu_{r} \in \xns$ satisfying $\bigcup_{1 \le i \le r} U^{+}_{\mu_i} = (U^{+}_{\gamma})h_{\rev{q}}$.  Set $M_{\rev{q}, \gamma} := m_{\rev{q},\gamma}\pmod{n-1}$. Lastly set $m_{\rev{q}}$ to be the smallest $s \in \N$, such that there exists $\mu_1, \mu_2, \ldots, \mu_s \in \xns$ satisfying $\bigcup_{1 \le i \le s} U^{+}_{\mu_i} = \im{\rev{q}}$. Note that for an antichain $\gamma_1, \gamma_2, \ldots, \gamma_{l}$ satisfying $\bigcup_{1 \le i \le l} U^{+}_{\gamma_{i}} = \dom{\rev{q}}$, we  have $m_{\rev{q}} \equiv \sum_{i=1}^{l}m_{\rev{q}, \gamma_{i}} \pmod{n-1}$.
 \end{Definition}

\begin{proposition}\label{prop:revsig}
	Let $T \in \On$ and let $q \in Q_{T}$. Then for any synchronizing word $\gamma \in \xns$ such that $\pi_{T}(\gamma, \cdot): Q_{T} \to Q_{T}$ has image equal to $\{q\}$, 
	\begin{equation}\label{eqn:sigisequaltosum}
	M_{\rev{q}, \gamma} \equiv \sum_{q \in Q_{T}} m_{\rev{q}} \pmod{n-1}.
	\end{equation}
	Consequently,  $M_{\rev{q},\gamma}$ depends only on $T$.  
	
\end{proposition}
\begin{proof}
Let $\gamma \in \xns$ be a synchronizing word for $T$ and let $q \in Q_{T}$ be the state of $T$ forced by $\gamma$. Note that this means $U^{+}_{\rev{\gamma}} \subseteq \dom{\rev{q}}$. 

Let $q_1, q_2, \ldots
, q_t$ be the states of $T$.  It follows from Remark~\ref{rem:synchinreverse} that  $$(U^{+}_{\rev{\gamma}})h_{\rev{q}} =  \bigcup_{1 \le i \le t} \lambda_{\rev{T}}(\rev{\gamma}, \rev{q}, \rev{q_{i}})\im{\rev{q_{i}}}.$$	

For $1 \le i \le t$, and $1 \le j \le m_{\rev{q_{i}}}$, let $\mu_{i,j} \in \xns$ be such that $\bigcup_{1 \le j \le m_{\rev{q_i}}}U^{+}_{\mu_{i,j}} = \im{\rev{q_{i}}}$. Thus we have, $$\bigcup_{1 \le i \le t} \bigcup_{1 \le j \le m_{\rev{q_i}}} U^{+}_{\lambda_{\rev{T}}(\gamma, \rev{q}, \rev{q_i})\mu_{i,j}}  =  (U^{+}_{\rev{\gamma}})h_{\rev{q}}.$$

Since $T \in \On$, $h_{\rev{q}}$ is a homeomorphism from $\dom{\rev{q}}$ to $\im{\rev{q}}$. Therefore, for $1\le i_1, i_2\le t$, $$\lambda_{\rev{T}}(\gamma, \rev{q}, \rev{q_{i_1}}) \im{\rev{q_{i_1}}} \cap \lambda_{\rev{T}}(\gamma, \rev{q}, \rev{q_{i_2}}) \im{\rev{q_{i_2}}} = \emptyset.$$ From this it follows that $m_{\rev{q}, \gamma} \equiv \sum_{i=1}^{t} m_{\rev{q_i}} \pmod{n-1}$. As the latter sum depends only of $T$, we see that $M_{\rev{q}, \gamma}$ depends only on $T$. 
\end{proof}

Let $T \in \On$ and let $q \in Q_{T}$ be arbitrary. Let $\gamma \in \xnp$ have length at least the synchronizing length of $T$ and satisfy $U_{\gamma}^{+} \subseteq \dom{\rev{q}}$. In light of Proposition~\ref{prop:revsig} we write $m_{T}$ for the element $M_{\rev{q}, \gamma} $ of $\Un{n-1}$

\begin{proposition}
	Let $T \in \On$. Then  $$(T)\rev{\sig}  =  m_{T}.$$
\end{proposition}
\begin{proof}

The proof is a consequence Lemma~\ref{lem:determiningimage} and the definition of the map $\rsig$ (Definition~\ref{Def:rsig}). 

Given a state $q \in \revrec{T}$, set $s \in \N$ to be minimal such that there is an antichain $\mu_1, \mu_2, \ldots, \mu_s \in \xns$ satisfying $\bigcup_{1 \le i \le s} U_{\mu_i}^{+} = \im{q}$. By definition $(\revrec{T})\rsig$ is the element of $\Un{n-1}$ congruent to $s$ (module $n-1$).

Now Lemma~\ref{lem:determiningimage} indicates that the image of a state $(w, \rev{p})$ of $\rev{T}$ is precisely $(U_{w}^{+})h_{\rev{p}}$ where by Remark~\ref{rem:synchinreverse}, $\rev{w}$ is a synchronizing word for $T$ which forces the state $p$ of $T$. The result now follows by the definition of $m_{T}$ (Definition~\ref{Def:mqgamma}) and the fact that removing incomplete response (i.e. in going from $\rec{\rev{T}}$ to $\revrec{T}$) does not affect the image of states.

\end{proof}

The following results relate the map $\rev{\sig}$ to the number of words of a given fixed length which force as particular state i.e. we establish a connection between the distribution of words at a given synchronizing level across the states of a given element of $\On$ and the image of $\rev{\sig}$ for the given element.

\begin{corollary}
	Let $T \in \On$. Let $k$ be a synchronizing level of $T$ and, for a state $q \in Q_{T}$, write let $X(q,k)$ for the  set of all elements of $\xn^{k}$ which force the state $q$. Then $(T)\rev{sig}$ is the  unique element $s$ of $\Un{n-1}$ satisfying $s\cdot |X(q,k)| \equiv m_{\rev{q}} \pmod{n-1}$ for all $q \in Q_{T}$.
\end{corollary}
\begin{proof}
	Let $\xn^{k}(q)$ be the set of all elements of $\xn^{k}$ that force the state $q$, then it is clear that $\bigcup_{\mu \in \xn^{k}(q)}U^{+}_{\rev{\mu}} = \dom{\rev{q}}$. Let $m_{q,k}:= |\xn^{k}(q)|$ Thus we have, by Proposition~\ref{prop:revsig}, that $m_{\rev{q}} \equiv (T)\rev{\sig}\cdot m_{q,k} \pmod{n-1}$.
	
	Uniqueness follows since, by Proposition~\ref{prop:revsig}, for let $s \in \Un{n-1}$ be such that $$s \cdot m_{q,k} \equiv m_{\rev{q}} \pmod{(n-1)}.$$ Then,
	$$
	s \equiv \sum_{q \in Q_{T}} s \cdot m_{q,k} \equiv \sum_{q \in Q_{T}} m_{\rev{q}} \equiv (T)\rev{\sig}.
	$$
\end{proof}

\begin{proposition}
	Let $T \in \On$ and $k \in \N$ be such that $T$ and $T'$ are synchronizing at level $k$.  Let  $q \in Q_{T}$ and $w \in \xns$ be such that $(w,q)$ is a state of $T'$. Let $X(k,q)$ be the subset of $\xn^{k}$ consisting of all elements which force the state $q$ of $T$ and define $X(k,(w,q))$ analogously. If $m$ is  the size of $X(k,q)$ and $m'$ is the size of $X(k,(w,q))$, then $m \equiv m_{\rev{(w,q)}} \pmod{n-1}$ and $m' \equiv m_{\rev{q}} \pmod{n-1}$. 
\end{proposition}
\begin{proof}
	Fix $j \in \N$ such that for any word $\gamma \in \xns$ and any pair of states $p \in Q_{T}$ and $(v,t) \in Q_{T'}$, $U^{+}_{\lambda_{T}(\gamma, p)} \subseteq \im{p}$ and $| \lambda_{T'}(\gamma, (v,t))| \ge k$.
	
	Let $q$ and $w$ be as in the statement of the proposition and let $p \in Q_{T}$ be arbitrary. Let $\gamma \in \xn^{\ast}$ be such that $|\gamma| > j+k$. Let $q$ be the state of $T$ forced by $\gamma$. Set $\gamma_1$ to be the length $j$ prefix of $\gamma$, $\psi := \lambda_{T}(\gamma,p)$, and $\psi_1 = \lambda_{T}(\gamma_1, p)$. Since $U^{+}_{\psi_1} \subseteq
	 \im{p}$, it follows, setting $v =  \psi_1 - \lambda_{T}((\psi_1)L_{p},p)$ and $p' = \pi_{T}((\psi_1)L_{p},p)$, that $(v,p')$ is a state of $T'$. Let $\mu, \psi_2 \in \xns$ be such that $\mu := \lambda_{T}((\psi_1)L_{p},p)$ and $\mu v \psi_2 = \psi$.
	 
	 Consider the word $(\psi w)L_{p}$. Since $U^{+}_{w} \subseteq \im{q}$ and $(w)L_{q} = \emptyword$, it follows that $(\psi w)L_{p} = \gamma$. Thus, as $(\psi w)L_{p} = (\psi_1)L_{p}(v \psi_2w)L_{p'} $, it follows that $\lambda_{T'}(\psi_2w, (v, p')) = \gamma - (\psi_1)L_{p}$ and $\pi_{T'}(\psi_2w, (v, p')) = (w, q)$.
	 
	 Notice that since $j$ is fixed, as the size of $\gamma$ increases, the size of  $\gamma - (\psi_1)L_{p}$ increases also. Now since $\dom{\rev{q}} = \bigcup_{\gamma \in X(k,q)} U_{\rev{\gamma}}$,we see by  induction that $\dom{\rev{q}} \subseteq \im{\rev{(w,q)}}$ and $\{ \rev{w} \delta \mid \delta \in \im{\rev{q}}\} \subseteq \dom{\rev{(w,q)}}$.
	 
	 Now let $\phi$ be an element of $\im{\rev{(w,q)}}$. This means that for any prefix $\phi_1$ of $\phi$, there is a word $\delta \in \xns$  and a state $(v,p)$ of $T'$ such that $\pi_{T}(\delta, (v,p)) = (w,q)$ and $\phi_{1} \le \rev{\lambda_{T'}(\delta, (v,p))} \le \phi$ is a prefix of $\phi$. 
	 
	 Thus let $\delta \in \xns$ be any word of length at least $j$, and $(v,t)$ be a state of $T$ such that $\pi_{T'}(\delta, (v,t)) = (w,q)$. Let $\gamma = \lambda_{T'}(\delta, (v,t))$. Then It follows that $(v\delta)L_{t} = \gamma$, $\pi_{T}(\gamma,t) = q$ and $v\delta - \lambda_{T}(\gamma, t) = w$. Thus we see that  the state of $T$ forced by $\gamma$ is $q$ and $\lambda_{T}(\gamma,t)w = v \delta$. An easy induction argument now shows that $\im{\rev{(w,q)}} \subseteq \dom{\rev{q}}$ and  $\dom{\rev{(w,q)}} \subseteq \{\rev{w}\theta \mid \theta \in \im{\rev{q}}\}$. 
	 
	 To conclude the proof we have only to make the following observation: 
	 \begin{IEEEeqnarray*}{rCl}
	 \{ x \in \xno  \mid \exists \mu \in X(q, k) : x \in \rev{U^{+}_{\mu}}  \}& = & \dom{\rev{q}} \nonumber \\
	 \{ x \in \xno  \mid \exists \mu \in X((w,q), k) : x \in \rev{U^{+}_{\mu}}  \}& = & \dom{\rev{(w,q)}}.
	 \end{IEEEeqnarray*}
	\end{proof}

The results below demonstrates that the restriction $\rev{\sig}: \Ln{n} \to \Un{n-1}$ is precisely the map from $\Ln{n} \to \Un{n-1}$ given by $T \mapsto (T^{-1})\rsig$.

\begin{proposition} \label{prop:pastinverseoffuture}
	Let $T \in \Ln{n}$, write $T^{-1}$ for the inverse of $T$ in $\Ln{n}$ and let $T'$ the let $k \in \N$ be a synchronizing level for $T$ and for $T'$. Let $\gamma \in \xn^{k}$ and $q \in Q_{T}$ be such that the state of $T$ forced by $\gamma$ is $q$. Then there are numbers $l,s,a,b \in \N$ and distinct words $\mu_1, \mu_2, \ldots, \mu_{s} \in \xn^{b}$, such that $s = s_{T^{-1}}n^a$,   $\bigcup_{1 \le i \le s} U^{+}_{\mu_i} = (U^{+}_{\rev{\gamma}})h_{\rev{q}}$ and for any annotation $\alpha$ of $T$, $b+(q)\alpha = l$.
\end{proposition}
\begin{proof}
	Fix a state  $q \in Q_{T}$ and a word $\gamma \in \xn^{k}$ such that the state of $T$ forced by $\gamma$ is $q$. 
	
	Let $\mu_1, \mu_2, \ldots, \mu_r \in \xns$ be minimal such that $\bigcup_{1 \le i \le r} U^{+}_{\mu_i} = (U^{+}_{\rev{\gamma}})h_{\rev{q}}$. Let $N:= \max\{ | \mu_i| \mid 1 \le i \le r \}$. It is clear that for any $\rho \in \xno$ and any $N' \ge N$, the size of the set $$\{ \nu \rho \mid \nu \in \xn ^{N'} \} \cap (U^{+}_{\rev{\gamma}})h_{\rev{q}}$$ is equal to  $r n^{j}$ for some $j \in \N$. We show that  $ rn^{j} = s_{{{T^{-1}}}}n^{a}$ for some $a \in \N$.
	
	Let $\rho \in \xno$ be arbitrary. Let $\phi \in \xn^{k}$, $w \in \xns$ and $p \in Q_{T}$ be such that $\rev{\phi}$ is a prefix of $\rho$ and  $(w,p)$ is the state of $T'$ forced by $\phi$. Thus, it follows that for any word $\delta \in \xns$ of long enough length and any state $s \in Q_{T}$ such that $\lambda_{T}(\delta, s)$ contains $\phi$ as a subword, there is a prefix $\delta_1$ of $\delta$ such that $\pi_{T}(\delta_1, s) = p$ and $\lambda_{T}(\delta_1,s)w$ has $\phi$ as a suffix.
	
	Let $l \in \N$ be minimal such that for $\tau \in \xnp$ a minimal word satisfying $w \le  \lambda_{T}(\tau, p)$, $|\tau| \le l$. Furthermore, since $T \in \Ln{n}$, there is a $d \in \Z$, such that for any word $\tau \in \xns$, satisfying $\pi_{T}(\tau, p) = q$, $|\lambda_{T}(\tau, p)| - |\tau| = d$.
	
	Let $N' \in \N$ be such that $N' - d -k \ge \max\{N-d-k, l \}$. We  count the number of elements $x_0x_1 \ldots \in (U^{+}_{\rev{\gamma}})h_{\rev{q}}$ such that $x_{N'}x_{N'+1}\ldots  = \rho$. Note that since all these words have suffix $\rho$, and $\rev{\phi}$ is a prefix of $\rho$, the following statement follows by a simple induction  argument and comments above. There is a fixed element $y \in \xno$ such that for any element $x:=\nu \rho \in (U^{+}_{\rev{\gamma}})h_{\rev{q}}$, $\nu \in \xn^{N'}$, the following things hold:
	\begin{itemize}
		\item there is an element $z \in \xns$ such that $zy \in U^{+}_{\rev{\gamma}}$ and $(zy)h_{\rev{q}} = x$,
		\item the unique path in $\rev{T}$ starting at $\rev{q}$ and labelled $zy$, reads the prefix $z$ into $\rev{p}$, and, 
		\item $\lambda_{\rev{T}}(z, q, p) = \nu\rev{w}$.
	\end{itemize}
	 Thus it suffices to count how many words $\tau \in \xns$ with suffix $\gamma$, satisfy $w \le \lambda_{T}(\tau, p)$, and $|\lambda_{T}(\tau, p)| = N'$. However, the set of elements of $\xns$, satisfying these conditions,  is precisely the set $$\{ \nu \gamma \mid \nu \in \xn^{N'-d - k}, w \le \lambda_{T}(\nu, p)\}.$$
	 
	 To conclude, we observe that by a result in \cite{OlukoyaAutTnr}, if $W$ is the subset of $\xn^{N'-d-k}$  consisting of all those elements $\nu$ for which $w \le \lambda_{T}(\nu, p)$, then $\bigcup_{\tau \in  W} U^{+}_{\tau} = \im{(w,p)}$. Since there is a state of $p'$ of $T^{-1}$ such that $\im{p'} = \im{(w,p)}$, the result is a consequence of Remark~\ref{rem:consequencesofsynchronousrep}.

\end{proof}

We now develop some consequences of Proposition~\ref{Prop:lnpn} for $\rev{T}$ where $T \in \Ln{n}$.

\begin{Remark}\label{rem:consequencesofsynchronousrep2}
	Let $T \in \pn{n}$ and $p \in Q_{T}$ and let  $T'$ be the transducer obtained from $T$ by removing incomplete response. The following facts are straight-forward to verify: if $p'$ is the state of $T'$ corresponding to $p$, then $\im{\rev{p'}} = \{ \rev{\Lambda(\emptyword,p)}\rho \mid \rho \in \im{\rev{p}}\}$; if $[T]$ is the  transducer $T'$ with $\omega$-equivalent states identified, then for a state $[p']$ of $[T']$, $\im{\rev{p'}} = \cup_{p' \in [p']} \im{\rev{p'}}$. 
	
	Using the facts above, we have the following.
	
	 By Proposition~\ref{prop:pastinverseoffuture}, given an element $P \in \pn{n}$ synchronizing at level $k$, $k$ large enough, there is a number $l \in \N$, and a divisor $s$ of some power of $n$, such that for any state $q \in Q_{P}$, and any word $\gamma \in \xn^{k}$ which forces the state $q$, there are elements $\mu_1, \mu_2, \ldots, \mu_s \in \xn^{l}$ such that $\bigcup_{1 \le i \le s} U^{+}_{\mu_i} = (U^{+}_{\rev{\gamma}})h_{\rev{q}}$. Setting $T$ to be the minimal representative of $P$, and fixing $\gamma \in \xn^{k}$, then, as observed above, for every $q \in Q_{P}$, $(U^{+}_{\rev{\gamma}})h_{\rev{[q]}} =\{ \rev{\Lambda(\ew, q)}x \mid x \in (U^{+}_{\rev{\gamma}})h_{\rev{q}} \}$. Setting, $\alpha: Q_{T}  \to \N$ to be the annotation of $T$ such that $[q] \mapsto \Lambda(\emptyword, q)$,  we see that for every state $t \in Q_{T}$ and any word $\tau  \in \xn^{k}$ which forces the state $t$, there are words $\mu'_1, \mu'_2, \ldots, \mu'_{s} \in \xn^{l+ ([q])\alpha}$ such that $\bigcup_{1 \le i \le s} U^{+}_{\mu'_{i}} = (U^{+}_{\rev{\tau}})h_{\rev{t}}$.  
	 
	 The last thing we note is the following fact: let $\beta$ be any other annotation of $T$, then for any states $[p], [q] \in Q_{T}$, $l+ ([p])\alpha - ([p])\beta  = l+([q])\alpha  - ([q])\beta $. This follows from the fact that there is an $m \in \Z$ such that $\beta = \alpha + m$ by definition of annotations. 
\end{Remark}

\begin{Remark}\label{remark:pastinverseoffuture}
 Following similar arguments to those  in Section~\ref{sec:extmorph} we may define maps $\rev{\sig}_{\omega}: \Ln{n} \to \mn{n}$ and $\rev{\sig}_{\omega}: \aut{\xnz, \shift{n}} \to \Pi_{k \in \N} \Un{n^{k}-1, n^{k}}$. By Proposition~\ref{prop:pastinverseoffuture}, it follows that $\rev{\sig}_{\omega} = \rsig_{\omega}^{-1}$ in both cases. 
\end{Remark}

As a corollary of Proposition~\ref{prop:revsig} we have the following result.

\begin{Theorem}\label{thm:pastinverseoffuture}
   Let $T \in \Ln{n}$, then $(T)\rsig\cdot (T)\rev{\sig} \equiv 1 \pmod{n-1}$.
\end{Theorem}
\begin{proof}
	This follows from the observation that $s_{T^{-1}} \pmod{n-1} = (T^{-1})\rsig$.
\end{proof}

It remains open whether Proposition~\ref{prop:pastinverseoffuture} holds also in $\On$. In particular,

\begin{enumerate}[label= {\bfseries Q.\arabic*}] 	\item Is it true that for $T \in \On$, $(T)\rev{\sig} = (T^{-1})\rsig$? \label{Question:pastinverseoffutureinOn}
	
\end{enumerate}

The following related questions represent the current extent of our ignorance.
\begin{enumerate}[label= {\bfseries Q.\arabic*}] \label{Question:pastinverseoffutureinOn}
\setcounter{enumi}{1}
	\item Is it true that $\rev{\sig}$ induces an automorphism of $\Ons{1}$?
	\item Is the group $\Ons{1}$ a characteristic subgroup of $\On$?
\end{enumerate}

%In light of Lemma~\ref{lem:determiningimage},  question~\ref{Question:pastinverseoffutureinOn} asks whether or not $((T)\rsig)^{-1}  = (\mrec{T})\rsig$ for an element $T \in \On$.

\subsection{The dimension group}
As promised, in this section we show that the map $\rev{\sig}_{\omega}: \aut{\xnz, \shift{n}} \to \Pi_{k \in \N} \Un{n^{k}-1, n^{k}}$ is precisely the dimension representation $\mathfrak{d}_{n}$. This enables us to characterise the kernel of $\mathfrak{d}_{n}$ as the group $\Kn{n}$. We begin with a brief exposition of the dimension group and for this we follow mainly the articles \cite{BoyleLindRudolph88} and \cite{BoyleMarcusTrow}.

Let $\sim_{n}$ be the smallest equivalence relation on $\Z \times \Z$ such that $(x, i) \sim_{n} (nx,  i+1)$ for all $x,i \in \Z$. That is for $x,y, i,j \in \Z$ $(x,i) \sim_{n}(y,j)$ if and only if  $xn^{j-i}= y$. The set $(\Z \times \Z)/ \sim_{n}$ forms an abelian group with binary operation defined as follows: for $x, y, i, j \in \Z$, $[(x,i)] + [(y,j)] = [(x n^{j} + yn^{i}, (i+ j)]$. 

The group $G_{n}$ is precisely the group $(\Z \times \Z)/\sim$ and is called \emph{the  dimension group corresponding to the space $\aut{\xnz, \shift{n}}$}.  Note that $G_{n} \cong \Z[ 1/n]$ is isomorphic to the $n$-adic rationals via the map that sends $[(x, i)] \mapsto x/n^{i}$. An element $[(x,i)]$ of $G_{n}$ is therefore called \textit{positive} if $x \ge 0$.

We define the action of the group $\aut{\xnz, \shift{n}}$ on the group $G_{n}$, by relating the positive elements of the group to closed subsets of the space $\xnz$.

\begin{Definition}
	Let $i \in \Z$ and $x \in \xnz$. The \emph{$i$-ray corresponding to $x$} is the set $\{ y \in \xnz \mid y_{j} = x_{j} \mbox{ for all } j \ge i \}$. An \emph{$i$-ray} is a subset of $\xnz$ which is equal to the $i$-ray corresponding to some element $x$ in $\xnz$. An $i$-beam is a finite union of $i$-rays. If the $i$-rays making up an $i$-beam correspond to elements $x(1), \ldots, x(m) \in \xnz$, then we say that the $i$-beam \emph{corresponds} to the elements $x(1), \ldots, x(m)$. A \emph{beam} is an $i$-beam for some $i \in \Z$. Let $B$ be an $i$-beam and $R_1, R_2, \ldots, R_k$ be $i$-rays such that $B = \sqcup_{1 \le j \le k} R_{k}$, then we say that the size of $B$ is $k$ and denote this by $|B| = k$.
\end{Definition}

Note that if $B$ is an $i$-beam for some $i \in \Z$, then $B$ is a $j$ beam for any $j \in \Z$ with $j \le i$. In particular, suppose $x(1), x(2), \ldots, x(m)$ are such that $B$ is the disjoint union of the $i$-rays corresponding to  the $x(k)$'s. Then $B$ is the union of the $mn^{i-j}$  $j$-rays corresponding to the elements $y \in \xnz$ such that $y_{l} = x(k)_{l}$  for all  $l \in \Z \backslash\{i-1, \ldots, j\}$, and  $y_{j}\ldots y_{i-1}:= \nu$ for $1 \le k \le m$ and $\nu \in \xn^{i-j}$. 

\begin{Definition}
	Let $B_1, B_2 \subseteq  \xnz$ and let $i,j \in \Z$ be such that $B_1$ is an $i$-beam and $B_2$ is a $j$-beam. Then we say that $B_1$ and $B_2$ are related if and only if $(|B_1|, -i) \sim_{n}  (|B_2|,-j)$. Abusing notation, we denote also by $\sim_{n}$ the equivalence relation on beams generated by this relation.
\end{Definition}

The map which sends the equivalence class of an $i$-beam $B$ to the element $[(|B|, -i)] \in \Z \times \Z/ \sim_{n}$ is a bijection from the set of equivalence classes of beams to the positive elements of $G_{n}$. We note the need for the minus sign follows since we define rays according to a fixed right-infinite tail. Defining rays according to a fixed left-infinite tail, gives rise to an equivalent theory, and is what is sometimes referred to as the \textit{future dimension representation}. We have chosen the ``past'' formulation to fit with  \cite{BoyleLindRudolph88} and \cite{BoyleMarcusTrow}.

\begin{Remark}\label{remk:inducedactionondimensiongroup}
	Let $\phi \in \aut{\xnz, \shift{n}}$. Then, it is a well known-fact that  $\phi$ maps beams to beams  and preserves the equivalence relation $\sim_{n}$ on beams. 
	This is not hard to see. Let $T \in \Ln{n}$ and $\alpha$ be an annotation of $T$ such that $(T, \alpha) = \phi$.
	 Let $k \in \N$ be a synchronizing length of $T$ and fix an $i$-ray $R_x$ corresponding to an element $x \in \xn$. 
	 Let $p \in Q_{T}$ be the state of $T$ forced by $x_i x_{i+1} \ldots x_{i+k-1}$, and let $r = (p)\alpha$. 
	Set $\rho:= \lambda_{T}(x_{i+k}x_{i+k+1}\ldots, p) \in \xno$, then, $(R_{x})\phi$ is precisely the subset of $\xnz$ consisting of those elements $y \in \xnz$ such that $y_{i+k+r}y_{i+k+1+r}\ldots = {\rho}$ and $\rev{\ldots y_{i+k-2+r}y_{i+k-1+r}} \in (U^{+}_{x_{i+k-1}\ldots x_i})h_{\rev{p}}$. By Lemma~\ref{lemma:ND1impliesclopenimage}, $(R_{x})\phi$ is a beam. 
	Moreover, by Remark~\ref{rem:consequencesofsynchronousrep2}, there are $l,s,b \in \N$ and $\mu_1, \mu_2, \ldots, \mu_{s} \in \xn^{b}$ such that $(U^{+}_{\rev{\gamma}})h_{\rev{p}}  = \sqcup_{1 \le j \le s} U^{+}_{\mu_j}$ where $\gamma = x_{i}\ldots x_{i+k-1}$, $l$ and $s$ depend only on $T$,  and $b =r+ l$. 
	
	Thus, there is a fixed $\tau \in \xno$ such that $(R_{x})\phi$ is the union of the $s$ $(i+k+r-b)$-beams consisting of elements $y \in  \xnz$ satisfying $y_{i+k+r}y_{i+k+1+r}\ldots  = {\rho}$, $y_{i+k-b+r}y_{i+k-2+r}y_{i+k-1 +r} = \rev{\mu_{d}\tau}$ for some $1 \le d \le s$. Therefore $\phi$ maps an $i$-ray to an $i+k-l$-beam of size $s$.  In particular we see that $\phi$ induces an automorphism of $G_{n}$, that sends an element $[(j, -i)]$ to $[(jn^{k}s, -i + l)]$, $j,i \in \Z$.
	
	For $\phi \in \aut{\xnz, \shift{n}}$ write $\hat{\phi}$ for the automorphism of $G_{n}$ induced by $\phi$. The map $\mathfrak{d}_{n}:\phi \mapsto \hat{\phi}$, is a homomorphism from $\aut{\xnz, \shift{n}}$ to $\aut{G_{n}}$.
\end{Remark}

We have the following result:
\begin{Theorem}
	
	Let $\phi \in \aut{\xnz, \shift{n}}$, then $\phi$ induces the trivial automorphism of $G_{n}$ if and only if there is an element $T \in \mathcal{K}_{n}$ and an annotation $\alpha$ of $T$ such that $\phi = (T,\alpha)$. In particular $\ker(\mathfrak{d}_{n}) \cong  \Kn{n}$ via the map $(T,\alpha) \mapsto T$.
\end{Theorem}
\begin{proof}
	First suppose that  $\phi \in \aut{\xnz, \shift{n}}$ is in  $\ker{\mathfrak{d}_{n}}$ so that $\hat{\phi}$ is the trivial automorphism of $G_{n}$. Let $(T,\alpha) \in \ALn{n}$ such that  $(T,\alpha) = \phi$. Let $k$ be the minimal synchronizing length of $T$. Since $\hat{\phi}$ is trivial it follows, by the computations in Remark~\ref{remk:inducedactionondimensiongroup}, that given $[(j,i)] \in  \Z \times \Z / \sim_{n}$, then $([(j,i)])\hat{\phi} = [(jn^{k}s, i+l)] = [ (j,i)]$. Thus, it must be the case that $n^{k}s = n^{i+l}$. In particular $s$ is a power of $n$,  and the result follows by Proposition~\ref{prop:pastinverseoffuture} say.  
	
	Now suppose that $T \in \mathcal{K}_{n}$. Let $\alpha$ be the canonical annotation of $T$; $k$ the synchronizing length of $T$; and  $l$ and $s$ as in Remark~\ref{remk:inducedactionondimensiongroup}
	
	Let $n = p_1^{r_1}p_2^{r_2}\ldots p_m^{r_m}$. Observe that $\widehat{T(p_a, N_a)}$, for $1 \le a \le m$, is the map $[(j,i)] \mapsto [(j p_a, i)]$. In particular $\widehat{\shift{n}}$ induces the map  $[(j,i)] \mapsto [(j n, i)]$.  Since $T \in \Kn{n}$ then $s$ is a power of $n$,  let $b \in \N$ be such that $s = n^{b}$ for  some $b \in \N$.   Set $\psi := (T,\alpha) $. Then we have that $[(j,i)]\psi = [(jn^{k+b}, i+l)]$. 
	
	If $k+b \ge i+l$, then $\psi \shift{n}^{i+l - k-b}$ induces the trivial automorphism of $G_{n}$. Otherwise $i+l \ge k+b$ and  $\psi \shift{n}^{k+b -i-l }$ induces the trivial automorphsim of  $G_{n}$. In particular setting $d$ to be the (unique) power of $\shift{n}$ such that $\psi \shift{n}^{d} \in \ker{\mathfrak{d}_{n}}$, then $\alpha -d$ is the unique annotation of $T$ such that $(T,\alpha -d) \in \ker{\mathfrak{d}_{n}}$.
\end{proof}

Combining Proposition~\ref{prop:pastinverseoffuture} and Remark~\ref{remark:pastinverseoffuture}, with Lemma~\ref{lem:determiningimage}, it follows that for $T \in \Ln{n}$, $((T)\rsig_{\omega})^{-1} =  (\mrec{T})\rsig_{\omega}$ . In particular, $T \in \Kn{n}$ if and only if $\rec{T} \in \Kn{n}$.

\section{Marker Automorphisms, homeomorphism states and embeddings} \label{Section:markersandhomeostate}

A \emph{marker automorphism} is a general method for constructing elements of $\aut{\xnz , \shift{n}}$ and can be found for example in the articles papers \cite{Hedlund69, BoyleLindRudolph88,VilleS18,BoyleKrieger}. In this section we show that several \emph{marker constructions} appearing in the literature yield elements of $\aut{\xnz, \shift{n}}$ of the form $(T,\alpha)$ where all the loop states of $T$ are homeomorphism states and $\alpha$ assigns the value $0$ to all the loop states. We use this fact to show that several group theoretic results proven about the groups $\aut{\xnz, \shift{n}}$ and $\aut{\xnz, \shift{n}}/\gen{\shift{n}}$ also hold for the group $\Kn{n}$. Since $\Kn{n} \le \Onr$ for all $1 \le r \le n-1$, we then extend some results in the paper \cite{BelkBleakCameronOlukoya} for the  group $\On$ to any of the groups $\Ons{r}$. We further make use of the homeomorphism condition on the loop states to lift these results to the groups $\aut{G_{n,r}}$. In particular we show that $\aut{x_{m}^{\Z}, \shift{m}}$ embeds as a subgroup of $\aut{G_{n,r}}$ for all $n > r \ge 1$. 

 It turns out that the set of all elements of $\Ln{n}$ all of whose loop states are homeomorphism states is a subgroup of $\Kn{n}$.  We define this group below.
 
\begin{Definition}
	Let $\mathcal{D}_{n} \subseteq \Kn{n}$ be the subset of $\Kn{n}$ consisting of those elements all of whose loop states are homeomorphism states.
\end{Definition}

It is not hard to see that $\mathcal{D}_{n}$ is a subgroup of $\Kn{n}$. It follows essentially from results in \cite{AutGnr}. For if $T, U \in \mathcal{D}_{n}$, with $p_{l(x)}$ the $x$ loop state of $T$ and $q_{l(y)}$ the $y$ loop state of $U$ whereby $\lambda_{T}(x, p_{l(x)}) = y$, then $(p_{l(x)}, q_{l(y)})$ is the $x$ loop state of the transducer $\core(T \ast U)$.  Since the product of $TU$ is the minimal representative of $\core(T \ast U)$, and as $(p_{l(x)}, q_{l(y)})$ induces a homeomorphism of $\xno$, it follows that the corresponding $x$ loop state of $TU$ is also a homeomorphism state. Closure under inversion follows from the definition of the transducer $T'$ as constructed in Construction~\ref{construction:inverse}. For if $q$ is a homeomorphism state of $T$, then $(\emptyword, q)$ is a state of $T'$. Moreover, if $q_{l(x)}$ is the $x$ loop state, and $ y = \lambda_{T}(x, q_{l(x)})$, then $(\emptyword, q_{l(x)})$ is the $y$ loop state of $T'$. As $T^{-1}$ is the minimal representative of $T'$ and $(\emptyword, q_{l(x)})$ is a homeomorphism state of $T'$, for  $q_{l(x)}$ the $x$ loop state of $T$, then the $x$ loop state of $T'$ is also a homeomorphism state.

\begin{proposition}
	Under the multiplication induced from $\Kn{n}$, the set $\mathcal{D}_{n}$ is a group.
\end{proposition}

The result below is likewise straight-forward. However we require some additional notation in order to prove it. 

Throughout this section, set $\dotr:= \{\dot{1}, \dot{2}, \ldots{\dot{r}}\}$ and write $\CCnr$ for the Cantor space $\{ a \rho \mid a \in \dot{r}, \rho \in \xno \}$ with the usual topology.

A \textit{transducer for $\CCnr$} is an initial transducer $A_{q_0} = \gen{\dotr \sqcup X_{n}, Q_{A}, \pi_{A}, \lambda_{A}}$ such that the following conditions are satisfied:
\begin{itemize}
	\item $\pi_{A} = (\pi_{A}: \dotr \times \{q_0\} \to Q_{A} \backslash \{q_0\}) \sqcup (\pi_{A}: \xn \times Q_{A}\backslash \{q_0\} \to Q_{A} \backslash \{q_0\})$;
	\item $\lambda_{A}(\cdot, q_0): \dotr\xn^{*} \mapsto  \{\ew\} \sqcup \dotr \xn^{\ast}$;
	\item $\lambda_{A}(\cdot, q_0): \CCnr \to \CCnr$.
\end{itemize}

Let $A_{q_0}$ be a transducer for $\CCnr$. Then $A_{q_0}$ is :
\begin{itemize}
	\item strongly synchronizing, if the non-initial transducer $T = \gen{\xn Q_{A} \backslash \{q_0\}, \pi_A, \lambda_{A}}$ is strongly synchronizing;
	\item invertible if the map $h_{q_0}:=\lambda_{A}(\cdot q_0): \CCnr \to \CCnr$ is invertible;
	\item bi-synchronizing if $A_{q_0}$ is invertible and $h_{q_0}^{-1}$ is induced by a transducer $B_{p_0}$ for $\CCnr$.
	\end{itemize}
We may extend the notion of minimality for initial transducers in the natural way to transducers for $\CCnr$. That is, a transducer for $\CCnr$ is minimal if it has no states of incomplete response, no pair of $\omega$-equivalent states and all states are accessible from the initial one.

We note that we naturally identify $\CCmr{n,1}$ with the Cantor space $\xno$.

The following result is from \cite{AutGnr}

\begin{Theorem}
	The set of homeomorphisms of $\CCnr$ which may be induced by bi-synchronizing transducers for $\CCnr$ forms a group $\Bnr$ under composition of maps. Moreover we have $\aut{G_{n,r}} \cong \Bnr$.
\end{Theorem}

\begin{lemma}\label{lem:dnsubgpofAutgnr}
	Let $n > r \ge 1$. Then $\mathcal{D}_{n}$ is a subgroup of $\aut{G_{n,r}}$. 
\end{lemma}
\begin{proof}
	 We first define an embedding of $\mathcal{D}_{n}$ in to $\mathcal{B}_{n,1}$. Then we observe  that $\mathcal{B}_{n,1}$ has a natural embedding into the group $\mathcal{B}_{n,r}$ for any $n> r \ge 1$.
	
	Let $D \in \mathcal{D}_{n}$ and  $q_0$ be a symbol not in $Q_{D}$. For $x \in \xn$, write $q_{l(x)}$ for the unique state of $D$ satisfying $\pi_{D}(x, q_{l(x)}) = q_{l(x)}$.
	
	 Define an initial  transducer  $D_{q_0}= \gen{ \xn, Q_{D} \sqcup \{ q_0\}, \pi_{D_{q_0}}, \lambda_{D_{q_0} }}$ as follows. For $x \in \xn$, $\pi_{D_{q_0}}(x, q_0) = q_{l(x)} $ and $\lambda_{D_{q_0}}(x, q_0) = \lambda_{D}(x, q_{l(x)}) $;  for $x \in \xn$ and $q \in Q_{D}$, $\pi_{D_{q_0}}(x, q) = \pi_{D}(x, q)$ and $\lambda_{D_{q_0}}(x, q) = \lambda_{D}(x, q)$. Write $\overline{D}_{q_0}$ for the minimal representative of $D_{q_0}$ maintaining the symbol $q_0$ for the initial state.
	 
	 We note that since all the loop states of $D$ are homeomorphism states, then $\overline{D}_{q_0}$ induces a homoemorphism of $\xno$ and $\overline{D}_{q_0}$ is bisynchronizing. 
	 
	 Let $D, E \in \mathcal{D}_{n}$. Then the following facts are straight-forward to check by direct computation:
	 
	 \begin{itemize}
	 	\item If $F$ is the inverse of $D$ in $\mathcal{D}_{n}$ then $\overline{F}_{q_0}$ is equal to $\overline{D}_{q_0}^{-1}$.
	 	\item $\overline{(DE)}_{q_0}$ is the minimal representative of the initial transducer $(D \ast E)_{(q_0, q_0)}$.
	 	\item $\overline{D}_{q_0}$ is induces the identity map on $\xno$ if and only if $D$ is the identity transducer.
	 \end{itemize}	
 
 For $D \in \mathcal{D}_{n}$, write $h(D)$ for the homeomorphism of $\xno$ induced by  $\overline{D}_{q_0}$. Then, by the facts above, the map $D \mapsto h(D)$ is an embedding of $\mathcal{D}_{n}$ into the group $\mathcal{B}_{n,1}$. 
 
 However the group $\mathcal{B}_{n,1}$ embeds as a subgroup of $\mathcal{B}_{n,r}$ in a natural way which we describe below.
 
 Let $A_{q_0}$ be a minimal transducer representing an element of $h$  of $\mathcal{B}_{n,1}$. Set $p_0$ to be a symbol not in $Q_{A}$. Define a transducer $\mathsf{A}_{p_0} = \gen{ \dotr \sqcup \xn, Q_{A} \sqcup \{p_0\}, \pi_{\mathsf{A}}, \lambda_{\mathsf{A}}   }$ as follows. For $a \in \dot{r}$ we set $\pi_{\mathsf{A}}(a, p_0) = q_0$ and $\lambda_{\mathsf{A}}(a, p_0) = a$; $\pi_{\mathsf{A}}(a, p) $ and $\lambda_{\mathsf{A}}(a, p)$ are undefined for $p \in Q_{A}$; for $x \in \xn$, and $q \in Q_{A}$ we set $\pi_{\mathsf{A}}(x, q) = \pi_{A}(x, q)$ and $\lambda_{\mathsf{A}}(x, q) = \lambda_{A}(x, q)$.
 	
 Set $h_r$ to be the map induced by $\mathsf{A}_{p_0}$. It is easy to check that $h_r$  is a homeomorphism of $\CCnr$ and $\mathsf{A}_{p_0}$ is bi-synchronizing. Thus $h_{r} \in \Bnr$. It is also straightforward to check that the map $h \mapsto
  h_{r}$ from $\mathcal{B}_{n,1} \to \mathcal{B}_{n,r}$ is a monomorphism.
	 
\end{proof}

We begin by defining the various marker constructions we consider and these  are essentially of two types.

\begin{Definition}
	Let $A,B \in \xnp$ be distinct words. Then we say that \emph{$A$ overlaps non-trivially with $B$} if either some non-empty initial  prefix of $A$ coincides with  some terminal suffix of $B$ or some non-empty terminal suffix of $A$ coincides with  some  initial prefix of $B$. If $A$ and $B$ do not overlap non-trivially then they are said to \emph{overlap trivially}. If $a$ is a non-empty prefix of $A$ that coincides with a suffix of $B$, then $A$ is said to \emph{overlap terminally with $B$ at $a$}. Analogously if $a$ is a non-empty suffix of $A$ that coincides with a prefix of $B$, then $A$ is said to \emph{overlap initially with $B$ at $a$}. 
	
	A non-empty word $A$ is said to have \emph{only trivial overlaps with itself} if and only if no initial proper prefix of $A$ coincides with a terminal suffix of $A$.

\end{Definition}

\begin{Definition}
	Let $M \subset \xnp$ be a set of non-empty words. Then $M$ is said to have  \emph{only trivial overlaps} if  any two distinct words in $M$ overlap trivially, and any word in $M$ only trivially overlaps with itself. 
\end{Definition}

\begin{comment}
\begin{Definition}
	Let $k \in \N_{1}$ and  $M \subset \xn^{k}$ be a finite non-empty set of words. Then $M$ is called a \emph{system of markers at level $k$} if one of the following things holds:
	
	\begin{enumerate}[label = M\arabic*]
		\item  $M$ has only trivial overlaps. \label{Def:normalmarkers}
		\item   There is  subset $B \subset \xnp$ such that every element $w \in M$ is of the form $b_1 m b_2$ for some $b_1, b_2 \in B$ and $m \in \N_{1}$ and for any two elements $w:=b_1 mb_2$ and $w':=b_1' m' b'_2$, then $w$ overlaps with $w'$ non-trivially if and only if it overlaps initially with $w'$ at a non-empty suffix of $b_2$ or it overlaps terminally with $w'$ at a non-empty prefix of $b_1$ or  it overlaps with $w'$ terminally or initially with $w'$ at the entire word $w$ (note that in this last case $w= w'$). \label{Def:borderedmarkers}
		
	\end{enumerate}
	
\end{Definition}
\end{comment}

We now define the marker automorphisms that we consider. 

The first appears in the paper \cite{BoyleKrieger}. Let $a, b \in \xn^{+}$ be words of equal length such that for $x \in \{a,b\}$ and $y \in \{a,b\}\backslash\{x\}$,  if $xx = u x v$ for some $u,v \in \xns$ then either $u$ or $v$ is the empty word, and there are no $u,v \in \xns$ such that $xx = uyv$. Note that $a$, $b$ and $ab$, by construction, are distinct prime words. 
Moreover if $c:=c_1c_2c_3c_4 c_5, d:=d_1d_2d_3d_4d_5 \in \{a,b\}^{5}$ are two not necessarily distinct words which overlap, then it cannot be the case that $c$ overlaps initially with $d$ at some suffix of $c$ that begins with a non-empty proper suffix of $c_3$ and it also cannot be the case that $c$ overlaps terminally with $d$ at some prefix of $c$ that ends in a non-empty proper prefix of $c_3$. 
Define an automorphism $f_{a,b} \in \aut{\xnz, \shift{n}}$ as follows. Let $l= |a|$ and $x \in \xnz$. Then $(x)f_{a,b}:=y \in \xnz$ is as follows. 
Given $i \in \Z$, if $x_{i-2l}x_{i-2l+1} \ldots x_{i+3l-1} = c_1c_2c_3c_4c_5$, for $c_j \in \{a,b\}$, then $y_{i}y_{i+1}\ldots y_{l} = \bar{c}_3$ for  $\bar{c}_{3} \in \{a,b\}\backslash \{c_3\}$, otherwise $y_{i}y_{i+1}\ldots y_{l} = x_{i}\ldots x_{l}$. The overlapping conditions above means that $y$ is a well-defined element $\xnz$. Thus, $f_{a,b}$ is a well-defined element of $\aut{\xnz, \shift{n}}$ (\cite{BoyleKrieger}) and is in fact an involution.

 For $k \in \N_{1}$, and $f \in \aut{\xnz, \shift{n}}$, write $\overline{f}_{k}$ for permutation of $\xn^{k}$, given by, $v \mapsto w$ if and only if the element $x \in \xnz$ defined by $x_{i|w|}x_{i|w|+1}\ldots x_{(i+1)|w| - 1} = w$, for all $i \in \xnz$, satisfies $(x)f = y$, for $y \in \xnz$ defined by $y_{i|v|}y_{i|v|+1}\ldots y_{(i+1)|v| - 1} = v$, $i \in \Z$. It is shown in the paper \cite{BoyleKrieger} that $\overline{(f_{a,b})}_{k}$ is trivial whenever $k < l$; fixes all words of length $l$ that are not rotationally equivalent to $a$ or $b$; and maps  a word rotationally equivalent to $a$ to one rotationally equivalent to $b$.
 
 The following lemma is useful.
 
 \begin{lemma}\label{lem:detectinghomeofromactiononxnz}
 	Let $T \in \Ln{n}$  and  $k \in \N$ be such $T$ is bi-synchronizing at level $k$.  Fix a state $q \in Q_{T}$ and a word $\Gamma \in \xn^{k}$ such that $\pi_{T}(\Gamma, q) = q$. Let $q'$ be the state of $T^{-1}$ forced by $\lambda_{T}(\Gamma, q)$. Fix an annotation $\alpha$ of $T$, and let $\beta$ be the annotation of $T^{-1}$ such that $(T,\alpha)(T^{-1},\beta) = (\id, 0)$. Then $-(q)\alpha = (q')\beta$ if and only if $q$,  and so $q'$, induces a homeomorphism of $\xno$.
 \end{lemma}
\begin{proof}
	Let $r \in \Z$ be such that $-(q)\alpha = (q')\beta =r$.  Since $\shift{n}^{i}(T,\alpha) = (T, \alpha +i)$ for $i \in \Z$,  we may assume that $(q)\alpha  = (q')\beta =r =  0$.
   	
	Set $\Delta:= \lambda_{T}(\Gamma, q)$. Let $\rho \in \xno$ be arbitrary, and let $y \in \xnz$ be such that ${y_{1}y_{2} \ldots } = \rho$ and $\ldots y_{-2}y_{-1}y_0\ldots =  \ldots \Delta\Delta$.
	
	Since $(q')\beta = 0$ and  $q'$ is the state  of $T^{-1}$ forced by $\Delta$, then, setting $x: = (y)(T^{-1}, \beta)$,  $ \ldots x_{-2}x_{-1}x_0\ldots =  \ldots \Gamma\Gamma$. 
	
	Set $\phi := {x_{1}x_{2}x_{3}\ldots}$. Since $(q)\alpha = 0$, and  $(x)(T,\alpha) = y$, we must have $\lambda_{T}(\phi, q) = \rho$. Thus we conclude that the state $q$ is a surjective. Since $q$ is also an injective state, by definition of $\Ln{n}$, it follows that $q$,  is a homeomorphism state of $T$ and so $q'$ is a homeomorphism state of $T^{-1}$ as well.
	
	Suppose now that $q$ is a homeomorphism state. Let  $r_1 = (q)\alpha$ and $r_2 = (q')\beta$.  Let $x \in \xnz$ be a word such that $\ldots x_{-1}x_0 = \ldots\Gamma\Gamma$. Set $y = (x)(T,\alpha)$. %and $z = (y)(T^{-1},\beta)$. 
	Then  $ \ldots y_{r_1 -1}y_{r_1}  = \ldots \Delta {\Delta}$ and $ \ldots x_{r_1 + r_2 -1} x_{r_1 + r_2} = \ldots \Gamma {\Gamma}$.
	
	If $r_1 + r_2 >0$, then  $x_{1} \ldots x_{(r_1+r_2)} $ is a suffix of  $\ldots\Gamma \Gamma $. However,  as  we may choose $x$ such that $x_{(r_1+r_2)} \ldots x_{1}$ is not a prefix of $\rev{\Gamma}^{\omega}$ this yields a contradiction.
	
	Suppose $r_1 + r_2 < 0$.  In this case,  $x_{r_1 + r_2+1}\ldots x_{0}\ldots$ is precisely the word $\lambda_{T'}({y_{r_1+1} y_{r_1 + 2}\ldots, q')}$. %However since $z_0\ldots z_{-(r_1 + r_2)} = x_0 \ldots x_{-(r_1+r_2)}$, it follows that  $x_{-(r_1+r_2)} \ldots x_{0} $ $z_{(-r_1+r_2)}\ldots z_0$  is a prefix of $\lambda_{T'}(\rev{\ldots y_{-r_1-2}y_{-r_1 -1}}, q')$.
	Since, $q$ is a homeomorphism state, and ${y_{r_1+1}y_{r_1+2}\ldots} = \lambda_{T}({x_{1}x_{2}\ldots }, q)$, then, given any $\rho \in \xno$, there is a choice of $x_{i}$ for $i \ge 1$ such that ${y_{r_1+1}y_{r_1+2}\ldots} = \rho$. This therefore implies that for any $\rho \in \xno$, $x_{r_1+r_2+1}\ldots x_0$ is a prefix of  $\lambda_{T}(\rho,q')$ and so $q'$ is a state of incomplete response. Since $T^{-1}$ is minimal, this yields the desired contradiction.
	
 Therefore, $r_1 + r_2 = 0$.
\end{proof}
 
 \begin{Remark}
 	We observe that a consequence of Lemma~\ref{lem:detectinghomeofromactiononxnz} is that the subset of $\mathcal{D}_{n}$ of elements which admit annotations that assign all loop states the value $0$ is in fact a subgroup. Closure under inversion follows from the quoted lemma, and closure under products follows by computation using the fact that loop states of the product are again homeomorphism states and so have no incomplete response. All the embedding results that follow, arising via marker constructions, yield groups in this proper subgroup of $\mathcal{D}_{n}$. 
 \end{Remark}
 
We have the following lemma.

\begin{lemma}\label{Lemma:swapshomeo}
	Let $a, b \in \xn^{l}$, for $l \in \N_{2}$, be such that $f_{a,b} \in \aut{\xnz, \shift{n}}$ is a well-defined marker-automorphism. Let $(T, \alpha) \in \ALn{n}$ satisfy $(T,\alpha)  = f_{a,b}$. Then  $T \in \mathcal{D}_{n}$, that is, all loop states of $T$ are homeomorphism states. 
\end{lemma}
\begin{proof}
	 Let $c \in \xn$ be a letter, $q_{c}$ be the $c$ loop state of $T$ and $d = \lambda_{T}(c, q_{c}) \in \xn$. We show that $(q_{c})\alpha$ must be $0$. 
	 
	 Let $r \in \Z$ be such that $(q_c)\alpha = r$ and let $x \in \xnz$ be a word such that $x_{i} = c$ for all $i < 0$. Set $y = (x)f_{a,b}$.
	
	Since $c^{2l} \notin \{a,b\}^{2}$ as $l \ge 2$, and $a$ and $b$ are distinct prime words, then, by definition of $f_{a,b}$, $y_{i} = c$ for all $i \le -1$ and so $d = c$.  By a similar argument, for all $ 0 \le i \le l-1$, $y_{i} = x_{i}$ since, by definition of $f_{a,b}$, $y_i$ is determined by the word $c^{2l} x_0 x_{1}\ldots x_{3l-1}$. From this we deduce that $r$ must be less than or equal to $0$, since, if $r>0$, then $y_0 \ldots y_{r-1} = c^{r}$. 
	
	Suppose $r <0$, then we have that $y_{r-1} = \lambda_{T}(x_{-1}, q_{c}) = c$. Therefore, as $x_{i} = c$ for all $i < 0$, it must be the case that $\lambda_{T}({x_0 x_{1}\ldots}, q_{c})$ must begin with  $d^{-r}$. However, as $x_i$, $ i \ge 0$, is arbitrary, we see that $q_{c}$ must then be a state of incomplete response. This results in a contradiction, as $T$ is minimal by definition. Therefore we conclude that  $r = 0$.
	
	Let $T^{-1}$ be the inverse of $T$ and  $\beta$ the annotation of $T^{-1}$ such that $(T,\alpha)(T^{-1}, \beta)$ is the identity map on $\xnz$. Since $f_{b,a}$ is the inverse of  $f_{a,b}$, it follows, by the argument above, that for for $c \in \xn$ and $p_c$ the $c$ loop state of $T^{-1}$, $\lambda_{T}(c, p_c) = c$ and $(p_c)\beta  = 0$. Lemma~\ref{lem:detectinghomeofromactiononxnz} now implies that $q_c$ and $p_c$ are homeomorphism states for all  $c \in \xn$.

    %Now for any word $y$ such that $y_{i} = d$ for all $i>0$, $x:=(y_{i})f_{a,b}$ must satisfy $x_{i} = c$ for all $i >0$. Therefore, as $f_{a,b}$ is surjective, for any $\rho \in \xno$, there is an element $x \in \xnz$ with $x_i = c$ for all $i >0$ and $y =(x)f_{a,b}$, satisfying $\rev{\ldots y_{-1}y_0} = \rho$. Thus, we deduce that $\lambda_{T}(\rev{\ldots x_{-1}x_{0}}, q_{c}) = \rho$. Hence the state $q_{c}$ is surjective, as it is also injective, $q_{c}$ is a homeomorphism state.
\end{proof}

There  are a few corollaries of this lemma.

\begin{corollary}\label{cor:transhomeo}
	Let $[a], [b] \in  \rwnl{l}$ for some $l \in \N$. Then there is an element $T$ of $\mathcal{D}_{n}$, such that $(T)\Pi$ fixes all elements of $\rwnl{k}$ for $k < l$ and, induces the transposition swapping $[a]$ and $[b]$ when restricted to $\rwnl{l}$.
\end{corollary}
\begin{proof}
	Let $a, b$ be prime words of length $l$, such that $a \in [a]$ and $b \in [b]$. Assume for the moment that $l \ge 2$. By Theorem 3.6 of \cite{BoyleKrieger} we may replace $a,b$ with rotations $a'$ and $b'$, such that for  $x \in \{a',b'\}$ and $y \in \{a',b'\}\backslash\{x\}$,  if $xx = u x v$ for some $u,v \in \xns$ then either $u$ or $v$ is the empty word, and there are no $u,v \in \xns$ such that $xx = uyv$. Thus we may construct the map $f_{a',b'}$ as before. As remarked above, in \cite{BoyleKrieger} it is shown that $\overline{(f_{a',b'})}_{k}$ is trivial whenever $k < l$; fixes all words of length $l$ that are not rotationally equivalent to $a$ or $b$; and maps  a word rotationally equivalent to $a$ to one rotationally equivalent to $b$. Thus if $T$ is the element of $\Kn{n}$ with annotation $\alpha$, satisfying $(T, \alpha) = f_{a',b'}$, then $T$ satisfies the requirements of the corollary by  Lemma~\ref{Lemma:swapshomeo}.
	
	If $l=1$, the single transducer which induces the transposition swapping $a$ with $b$ is an element of $\mathcal{D}_{n}$ satisfying the conclusion of the corollary. 
\end{proof}

\begin{corollary}\label{cor:denseinpermrep}
	Let $(\rho_1, \rho_2, \ldots) \in \Pi_{k \in \N_{1}} \sym{\rwnl{k}}$. Then there is an element $T$ of $\mathcal{D}_{n}$ such that the restriction of $(T)\Pi$ to the set $\rwnl{k}$ is the map $\rho_{k}$.
\end{corollary}
\begin{proof}
	This is clear from Corollary~\ref{cor:transhomeo}.
\end{proof}

\begin{corollary}
	The groups $\Kn{n}$ and $\mathcal{D}_{n}$ are centerless.
\end{corollary}
\begin{proof}
	This is clear from Corollary~\ref{cor:denseinpermrep}.
\end{proof}

Thus we have a different proof of a result in \cite{BelkBleakCameronOlukoya}.

\begin{corollary}
	The group $\Onr$ is centerless for all $1 \le r \le  n-1$.
\end{corollary}
\begin{proof}
	Let $T \in \Onr$. We may assume that $T \in \Onr \backslash \Ln{n}$ since, otherwise, by Corollary~\ref{cor:denseinpermrep}, we may find an element of $\mathcal{D}_{n}$ that does not commute with $T$.  
	
	Now, as in the paper \cite{BelkBleakCameronOlukoya},we may find $l \in \N$, words $a, b \in \rwnl{l}$ for some $l \in \N$  and distinct elements $j, k \in \N$ such that $(a)(T)\Pi \in \rwnl{j}$ and $(b)(T)\Pi \in \rwnl{k}$. 
	
	There is an element  $U \in \Kn{n}$, which maps $a$ to $b$. Thus we have that $(a)(TU)\Pi \ne (a)(UT)\Pi$ as  $(a)(UT)\Pi \in \rwnl{k}$ and $(b)(TU)\Pi \in \rwnl{j}$. Since $\Kn{n} \le \Ons{1}$, then $\Kn{n} \le \Ons{r}$ and so $T$ is an element of $\Onr$ that does not commute with $T$. 
\end{proof}

The second type of marker automorphism we consider appears in the paper \cite{VilleS18}. 

Let $w \in \xnp$ be a word and $U \subset \xn^{l}$ be a collection of non-empty words of length $l$. Suppose that any two elements $wvw$ and $wuw$ of the set $wUw = \{ w u w \mid u \in U \}$ can only overlap either initially or terminally at the prefix or suffix $w$. Note that consequently any element of the set $wUw$ is a prime word and $w$  is also a prime word. 

Let $f: (wU)^{+}w \to (wU)^{+}w$ be a map such that the following holds: 
\begin{enumerate}[label=H.\arabic*]
	\item for all $k \in \N$, $f: (wU)^{k}w \to (wU)^{k}w$ is a bijection, and, \label{preserveslength}
	\item  there is a $\delta \in \N$, such that for any pair of words $a,b \in (wU)^{+}w$ of length strictly bigger than $\delta$ and any  $1 \le i \le |a| - \delta$, such that $a_{\max\{1, i- \delta\}} \ldots a_{i+ \delta} = b_{\max\{1, i- \delta\}} \ldots b_{ i+ \delta}$, then the $i$'th letters of $(a)f$ and $(b)f$ coincide and the $i$'th letters of $(a)f^{-1}$ and $(b)f^{-1}$ coincide. \label{uniformcontinuity}
\end{enumerate}

 Define a map  $f_{wUw}$ from $ \xnz$ to itself as follows. 
 Let $x \in \xnz$. We say $x$ does not contain a left or right infinite subword in $wUw$ if for any  $i \in \Z$, $x_{i}x_{i+1}\ldots \ne wu_1wu_2wu_3\ldots$ and $\ldots x_{i-2}x_{i-1}x_{i} \ne \ldots wu_{-2}wu_{-1}w u_{0}w$ for any choice of  $u_{j}$'s in $U$. 
 Let $x \in \xnz$ be a sequence that does no contain a left or right infinite subword in $wUw$. 
 Let $i \le j \in \Z$, then we say that $i$ and $j$ are maximal  such that  $x_{i}x_{i+1}\ldots x_{j} \in (wU)^{+}w$ if for any other pair $(i',j') \in \Z\times \Z \backslash\{(i,j)\}$ with $i' \le i$, $j' \ge j$, we have $x_{i'}x_{i'+1}\ldots x_{j'} \notin (wU)^{+}w$. 
 Fix $x \in \xnz$ such that $x$ does not contain a left or right infinite subword in $wUw$. 
 Define a sequence $y \in \xnz$ as follows. For every $i \le j \in \Z$ for which $i,j$ are maximal such that $x_{i}x_{i+1}\ldots x_{j} \in (wU)^{+}w$, set $y_{i}\ldots y_{j} =  (x_{i}x_{i+1}\ldots x_{j})f$. 
 At all other indices $i \in \Z$, we set $y_i = x_{i}$. 
 Set $(x)f_{wUw} = y$.  
 Given a sequence of $x^{(i)}$, $i \in \N$ of elements of $\xnz$ all of which do not contain a left or right infinite subword in $wUw$ and which converges to a point $x \in \xnz$, then, by hypothesis~\ref{uniformcontinuity}, the sequence $y^{(i)} = (x^{(i)})f_{wUw}$ also converges to a point $y \in \xnz$. 
 We set $(x)f_{wUw} = y$. 
 This therefore extends $f_{wUw}$ to all points of $\xnz$ to yield a well-defined shift-commuting map from $\xnz \to \xnz$. 
 Note that since the inverse of  $f_{wUw}$ is the map $(f^{-1})_{wUw}$, $f_{wUw}$ is homeomorphism, and so an element of $\aut{\xnz, \shift{n}}$.

\begin{lemma}\label{lemma:conveyorhomeo}
	Let $U \subset \xn^{l}$ for some $l \in \N$ and $w \in \xnp$ be such that $|w| \ge 2$ and  any two elements $wvw$ and $wuw$ of the set $wUw = \{ w u w \mid u \in U \}$ can only overlap either initially or terminally at the prefix or suffix $w$. Let $f: (wU)^{+}w \to (wU)^{+}w$ satisfy hypotheses~\ref{preserveslength} and \ref{uniformcontinuity}. Then there is a transducer $T \in \mathcal{D}_{n}$ and an annotation $\alpha$ of $T$ such that  $f_{wUw} = (T, \alpha)$ and all loop states of $T$ have value $0$ under $\alpha$.
\end{lemma}
\begin{proof}
	The proof approach is very similar to the proof of Lemma~\ref{Lemma:swapshomeo}.
	
	By results in \cite{BelkBleakCameronOlukoya} there is an element $T \in \Ln{n}$ and an annotation $\alpha$ of $T$ such that $f_{wUw} = (T, \alpha)$.
	
	By an observation above $w$ is a prime word and $wuw$ is a prime word of length at least $5$ for any $u \in U$. Let $a \in \xn$ be a letter and $q_{a}$ be the a loop state of $a$. We show that $(q_{a})\alpha = 0$.
	
	First we observe that $\lambda_{T}(a, q_a) = a$. Since the element $x \in \xnz$ given by $x_i = a$ for all $i \in \Z$ is fixed by $f_{wU}$ as $w$ is a prime word of length at least $2$.
	
	Let $r \in \Z$ be such that $(q_a)\alpha = r$ and consider a sequence $x \in \xnz$ such that $x_i =a$ for all $i<0$ and let $y = (x)f_{wUw}$.  We note that $y_{i} = a$ for all $i < 0$ as $w$ is a prime word and by definition of $f_{wUw}$ a word in $(wU)^{+}w$ is mapped to a word in $(wU)^{+}w$ of the same length.
	
	If $r >0$, then $y_{0}\ldots y_{r-1} = a^{r}$. However, since for all $k \in \N$, $f: (wU)^{k}w \to (wU)^{k}w$ is a bijection, and since $\{a\}^{+} \cap (wU)^{+}w = \emptyset$, then there is a choice of $x_{0} x_{1}\ldots$ such that $y_{0}\ldots y_{r-1} \ne a^{r}$. Specifically, we note that, as $x_i = a$ for all $i < 0$, and $|w|, n \ge  2$, then we may choose $x_{0}\ldots x_{|w|}$ not equal to $w$ and such that $x_0 \ne a$. In this case,  we have, $y_{0} = x_0 \ne a$, yielding the desired contradiction. Therefore $r \le 0$.
	
	 If $r <  0$, then the state $q_{a}$ is a state of incomplete response as the word $\lambda_{T}(x_0x_{1}x_{2}\ldots, q_a) = y_{r}y_{r+1}\ldots$ must begin with a non-trivial power of $a$. This contradicts the minimality of $T$ and so $r =0$.
	 
	 Let $T^{-1} \in \Ln{n}$ be the inverse if $T$ and let $\beta$ be the annotation of $T^{-1}$ such that $(T^{-1}, \beta) = f_{wUw}^{-1}$. Observe that  $(f^{-1})_{wUw}$ is the inverse of $f_{wUw}$. Therefore, by the arguments above, for any $a \in \xn$,  the loop state $p_a$ of $a$ in $T^{-1}$, satisfies, $\lambda_{T^{-1}}(a, p_a) = a$ and $(p_a)\beta = 0$. By Lemma~\ref{lem:detectinghomeofromactiononxnz}, $q_a$ and $p_a$ are therefore homeomorphism states.
	 
	 %Now we observe that since $(f^{-1})_{w,U}$ is the inverse of $f_{w, U}$, then for any word $\rho \in \xnz$, and any element $x$ with $x_{i} = a$ for all $i >0$ and $x_{0}x_{-1}x_{-2}\ldots = \rho$, there is an element $y \in \xnz$ with $y_i = a$ for all $i>0$ such that $(x)f^{-1}_{w,U}  = y$. From this we deduce that the state $q_{a}$ is surjective. As $q_{a}$ is injective as well $q_{a}$ is a homeomorphism state.	 
\end{proof}

\begin{corollary}\label{cor:Dncontainsautshift}
Let $m, n \in \N_{2}$, then $\aut{X_{m}^{\Z}, \shift{m}}$ is isomorphic to a subgroup of $\mathcal{D}_{n}$.  
\end{corollary}
\begin{proof}
	In the paper \cite{VilleS18} the following facts are demonstrated. 
	
	The existence of a subset $U$ of $\xn$ and  $w$ of length at least $2$ such that any two elements $wvw$ and $wuw$ of the set $wUw = \{ w u w \mid u \in U \}$ can only overlap either initially or terminally at the prefix or suffix $w$. 
	
	An embedding of $\aut{X_{m}^{\Z}, \shift{m}}$ into the group of functions $f: (wU)^{+}w \to (wU)^{+}w$ which satisfy hypothesis~\ref{preserveslength} and \ref{uniformcontinuity}. 
	
	This embedding then yields an embedding  into the group $\aut{\xnz, \shift{n}}$ by taking the maps $f_{uWw}$. 
	
	By Lemma~\ref{lemma:conveyorhomeo} this gives an embedding of $\aut{X_{m}^{\Z}, \shift{m}}$ into the group $\mathcal{D}_{n}$ as follows. 
	
	Set $G$ to be the subgroup of $\aut{\xnz, \shift{n}}$ consisting of elements $(T,\alpha)$ for which there is some function $f: (wU)^{+}w \to (wU)^{+}w$ with $f_{wUw} = (T,\alpha)$. Define a map from $G \to \mathcal{D}_{n}$ by $(T,\alpha) \mapsto T$. Noting that by Lemma~\ref{lemma:conveyorhomeo} for any element $(T, \alpha) \in G$, $\alpha$ assigns the value $0$ to all loop states of $T$. Thus, since an annotation is uniquely determined by the image of a single state, it follows by Theorem~\ref{theorem:repsbyLnandannotation} that the map $(T,\alpha) \mapsto T$ from $G \to \mathcal{D}_{n}$ is an injective homomorphism. 
\end{proof}

\begin{corollary}\label{Cor:autshiftembedsinOnr}
	Let $m \in \N_{2}$, then $\aut{X_{m}^{\Z}, \shift{m}}$ is isomorphic to a subgroup of $\Onr$ for any $n,r \in \N$ satisfying $n > r \ge 1$.
\end{corollary}
\begin{proof}
	This is a consequence of Corollary~\ref{cor:Dncontainsautshift} and the following inclusions: $$ \mathcal{D}_{n} \le \Ons{1} \le \Onr.$$
\end{proof}

\begin{corollary}
	Let $m \in \N_{2}$, then $\aut{X_{m}^{\Z}, \shift{m}}$ is isomorphic to a subgroup of $\aut{G_{n,r}}$ for any $n,r \in \N$ satisfying $n > r \ge 1$.
\end{corollary}
\begin{proof}
	This follows from Lemma~\ref{lem:dnsubgpofAutgnr} and Corollary~\ref{cor:Dncontainsautshift}.
\end{proof}

It is natural to ask if the group $\mathcal{D}_{n}$ is a proper subgroup of $\Kn{n}$. The example below demonstrates that this is in fact the case. We fix $x$ an arbitrary symbol between $3$ and $n-1$ with the convention that if $n =2$, then $x = \emptyword$ and so the output when $x$ is read from any state is empty. The reader can verify that the state with loop labelled $0$ is not a homeomorphism state.

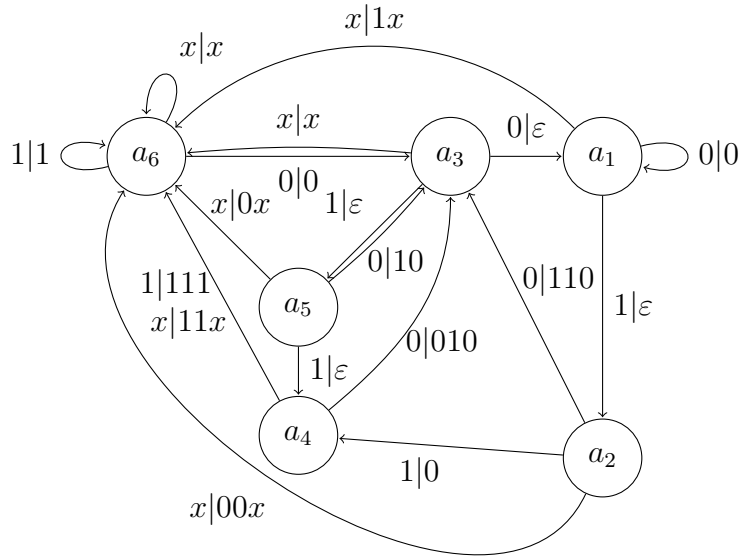
\begin{figure}[H] 
		\begin{center}
		\begin{tikzpicture}[shorten >=0.5pt,node distance=3cm,on grid,auto]
		\tikzstyle{every state}=[fill=none,draw=black,text=black]
		\node[state] (q_6)   {$a_6$};
		\node[state] (q_3) [xshift=4cm] {$a_3$};
		\node[state] (q_1) [xshift=6cm] {$a_1$};
		\node[state](q_5) [xshift=2cm, yshift=-2cm] {$a_5$};
		\node[state](q_4) [xshift=2cm, yshift=-3.7cm] {$a_4$};
		\node[state](q_2) [xshift=6cm, yshift=-4cm] {$a_2$};
		\node (d) [xshift=1cm, yshift=-4cm]{};
		\path[->]
		(q_6) edge [loop left] node [swap] {$1|1$} ()
                 edge [in=90, out=60, loop] node [swap] {$x|x$} ()
		         edge node[swap]  {$0|0$} (q_3)
		(q_3) edge node{$0|\emptyword$} (q_1)
		         edge node[swap]{$1|\emptyword$} (q_5)
		         edge[in=5, out=175] node[swap]{$x|x$} (q_6)
		(q_1) edge[loop right] node{$0|0$} ()
		         edge node{$1|\emptyword$} (q_2)
		         edge[in=45, out=135] node[swap]{$x|1x$} (q_6)   
		(q_5) edge[out=40, in=230] node[xshift=-0.3cm, swap]{$0|10$} (q_3)
		         edge node{$1|\emptyword$} (q_4)
		         edge node[xshift=-0.3cm, swap]{$x|0x$} (q_6)
		(q_4) edge[out=40, in=270] node[xshift=-0.3cm, swap]{$0|010$} (q_3)
		        edge node[yshift=0.5cm]{$1|111$}  node[xshift=0.2cm]{$x|11x$}(q_6)    
		(q_2) edge node [xshift=-0.2cm, swap]  {$0|110$} (q_3)
		         edge node{$1|0$} (q_4)
		        edge [out=245, in=237]  node {$x|00x$} (q_6);
		\end{tikzpicture}
	\end{center}
\caption{An element demonstrating the inclusion $\mathcal{D}_{n} \subsetneq \Kn{n}$.}    
\label{fig:L2withloopstatenonhomeo}

\end{figure}

One can in fact construct examples to show that the index of $\mathcal{D}_{n}$ in $\Kn{n}$ is infinite.

\printbibliography 

\end{document}